\documentclass[10pt]{article}

\usepackage[english]{babel}
\usepackage[utf8]{inputenc}
\usepackage{amsmath}
\usepackage{amssymb}
\usepackage{amsthm}
\usepackage[all,cmtip]{xy}
\usepackage[bookmarks=true]{hyperref}
\usepackage{tikz}
\usepackage{todonotes}
\usepackage{slashed}
\usepackage{mathtools}
\usepackage{dsfont}
\usepackage{tensor}
\usepackage{verbatim}
\usepackage{graphicx}
\usepackage[shortlabels]{enumitem}
\usepackage{xstring}
\usepackage[left=3cm,right=3cm,top=3cm]{geometry}
\usepackage[capitalize]{cleveref}
\usepackage{tensor}
\usepackage{mathrsfs}

\definecolor{darkolivegreen}{rgb}{0.33, 0.42, 0.18} 
\definecolor{cobalt}{rgb}{0.0, 0.24, 0.43}

\newcounter{comments}
\newenvironment{displaycomment}{\begin{list}{}{\rightmargin=1cm\leftmargin=1cm}\item\sf\begin{small}}{\end{small}\end{list}}

\newcommand{\C}{\mathbb{C}}
\newcommand{\R}{\mathbb{R}}
\newcommand{\Z}{\mathbb{Z}}

\newcommand{\mc}[1]{\mathcal{#1}}
\def\smooth{\mathrm{s}} %Command for smooth vectors
\newcommand{\Clsm}{\Cl(V)^{\smooth}} %Command for the subalgebra of smooth vectors in the Clifford algebra

\newcommand*{\defeq}{\mathrel{\vcenter{\baselineskip0.5ex \lineskiplimit0pt
                        \hbox{\scriptsize.}\hbox{\scriptsize.}}}%
        =}

\newcommand{\ph}{\varphi}

\newcommand{\eps}{\varepsilon}

\newcommand{\Aut}{\operatorname{Aut}}
\newcommand{\SO}{\operatorname{SO}}
\newcommand{\Spin}{\operatorname{Spin}}
\newcommand{\U}{\operatorname{U}}

\newcommand{\Imp}{\operatorname{Imp}}

\newcommand{\Cl}{\operatorname{Cl}}
\newcommand{\ClvN}{\operatorname{ClvN}}
\newcommand{\lie}[1]{\mathfrak{#1}}

\newcommand{\opp}{\operatorname{opp}}
\newcommand{\pr}{\operatorname{pr}}

\newcommand{\lact}{\triangleright}
\newcommand{\ract}{\triangleleft}

\newcommand{\Hom}{\operatorname{Hom}}

\newcommand{\Diff}{\mathcal{D}i\!f\!\!f}

\newcommand{\Open}{\mathcal{O}pen}
\newcommand{\Frech}{\mathcal{F}\!r\acute{e}ch}
\newcommand{\Cat}{\mathcal{C}\!at}

\def\grbtech#1{\mathcal{G}\hspace{-0.06em}r\hspace{-0.06em}b_{\hspace{-0.07em}{#1}}}

\def\grbcon#1#2{\grbtech{#1}^{\nabla\!}\left(#2\right)}
\def\buntech#1#2{\mathcal{B}\hspace{-0.01em}un_{\hspace{0.05em}#1}^{#2}}
\def\ufusbun#1{\mathcal{F}\!us\buntech{}{}(#1)}

\newcommand{\String}{\operatorname{String}}
\newcommand{\SStruct}{\smash{\widetilde{L\Spin}}(M)}
\newcommand{\BCE}{\smash{\widetilde{L\Spin}}(d)}

\renewcommand{\O}{\operatorname{O}}
\renewcommand{\d}{\operatorname{d}}
\newcommand{\der}[2]{\frac{\d #1}{\d #2}} %Command for derivative

\def\res{L}
\def\osplit{\O_{\res}^{\theta}(V)_{0}}

\newcommand{\id}{\mathds{1}}

\def\quot#1{``#1''}

\def\cstar{C$^{*}\!$}

\makeatletter
\def\thm@space@setup{%
  \thm@preskip=0.6em
  \thm@postskip=0.00em
}
\makeatother

\theoremstyle{definition}
\newtheorem{definition}{Definition}[subsection]
\newtheorem{remark}[definition]{Remark}

\theoremstyle{plain}
\newtheorem{theorem}[definition]{Theorem}
\newtheorem{proposition}[definition]{Proposition}
\newtheorem{lemma}[definition]{Lemma}
\newtheorem{corollary}[definition]{Corollary}

\setlist{topsep=-0.2em,itemsep=0em}

\crefname{enumi}{\unskip}{\unskip}
\crefname{equation}{}{}

\title{Connes fusion of spinors on loop space}

\author{Peter Kristel and Konrad Waldorf}
\date{}

\makeatletter
\hypersetup{
colorlinks=true,
urlcolor=cobalt,
linkcolor=cobalt,
citecolor=darkolivegreen,
pdfmenubar=false,
pdftoolbar=true,
pdftitle = {\@title},
pdfauthor = {\@author},
}
\makeatother

\begin{document}

\maketitle

\begin{abstract}
\noindent
The loop space of a string manifold supports an infinite-dimensional Fock space bundle, which is an analog of the spinor bundle on a spin manifold.
This spinor bundle on loop space appears in  the description of 2-dimensional sigma models as the bundle of  states over the configuration space of the superstring. We construct a product on this bundle covering the fusion of loops, i.e., the merging of two loops along a common segment.
For this purpose, we exhibit it as a bundle of bimodules over a certain von Neumann algebra bundle, and realize our product fibrewise using the Connes fusion of von Neumann bimodules. Our main technique is to establish a novel relation between string structures, loop fusion, and the Connes fusion of Fock spaces. The fusion product on the spinor bundle on loop space was proposed by Stolz and Teichner as a part of a programme to explore the relation between generalized cohomology theories, functorial field theories, and index theory. It is related to the pair of pants worldsheet of the superstring, to the extension of the corresponding smooth functorial field theory down to the point, and to a higher-categorical bundle on the underlying string manifold, the stringor bundle. 
%\medskip
%
%\noindent
%MSC 2010:  
\end{abstract}

%\listoftodos

     \setlength{\parskip}{0ex}
     \begingroup
     \tableofcontents
     \endgroup
     \setlength{\parskip}{1.5ex}

\section{Introduction}

In this article we construct a Connes fusion product on the spinor bundle on the loop space of a string manifold, and thus solve a problem formulated by Stolz and Teichner in 2005 \cite{stolz3}. We recall that a string manifold is a spin manifold $M$ such that
\begin{equation*}
\textstyle\frac{1}{2}p_1(M)=0\text{.}
\end{equation*}
If $M$ is a string manifold, then its free loop space $LM=C^{\infty}(S^1,M)$ is a spin manifold in the sense of Killingback \cite{killingback1}, i.e., it comes equipped with a certain principal bundle for the basic central extension
\begin{equation*}
\U(1) \to \BCE \to L\Spin(d)
\end{equation*}
of the loop group of $\Spin(d)$, where $d=\dim(M)$. In our previous work \cite{Kristel2020} we have constructed an infinite-dimensional Hilbert space bundle $\mathcal{F}(LM)$ on $LM$ by associating a  certain unitary representation
\begin{equation*}
\BCE \to \U(\mathcal{F})
\end{equation*}
to this principal bundle, where $\mathcal{F}$ is the Fock space of \quot{free fermions on the circle}. We have proved \cite{Kristel2020} that the bundle $\mathcal{F}(LM)$ realizes precisely what Stolz-Teichner called the spinor bundle on loop space \cite{stolz3}.

In the present article, we construct a hyperfinite type III$_1$ von Neumann algebra bundle $\mathcal{N}$ over the space $PM$ of smooth paths in $M$, and prove (\cref{lem:vonNeumannBimodule}) that the spinor bundle $\mathcal{F}(LM)$ is a $p_1^{*}\mathcal{N}$-$p_2^{*}\mathcal{N}$-bimodule bundle, where $p_1,p_2: LM \to PM$ are the maps that divide a loop into two halves. From a fibrewise point of view, if $\beta_1,\beta_2\in PM$ are paths with common initial points and common endpoints, and $\beta_1 \cup\beta_2$ denotes the corresponding loop, then $\mathcal{F}(LM)_{\beta_1\cup\beta_2}$ is an $\mathcal{N}_{\beta_1}$-$\mathcal{N}_{\beta_2}$-bimodule.
Our main result (\cref{th:fusionfibrewise}) is the construction of unitary intertwiners
\begin{equation*}
\chi_{\beta_1,\beta_2,\beta_3}: \mathcal{F}(LM)_{\beta_1\cup\beta_2} \boxtimes \mathcal{F}(LM)_{\beta_2\cup\beta_3} \to \mathcal{F}(LM)_{\beta_1\cup\beta_3}
\end{equation*}
of $\mathcal{N}_{\beta_1}$-$\mathcal{N}_{\beta_3}$-bimodules, where $\boxtimes$ is the Connes fusion of bimodules over $\mathcal{N}_{\beta_2}$, and $\beta_1,\beta_2,\beta_3\in PM$ is any triple of paths with common initial points and common endpoints.
One may regard the loop $\beta_1\cup\beta_3$ as the fusion of the loops $\beta_1\cup\beta_2$ and $\beta_2\cup\beta_3$ along the common segment $\beta_2$, and regard the Connes fusion product $\chi_{\beta_1,\beta_2,\beta_3}$ as lifting this loop-fusion product.  
We prove  that the intertwiners $\chi_{\beta_1,\beta_2,\beta_3}$ equip the spinor bundle with an associative product over the fusion of loops (\cref{th:fusionass}). 

Our construction of the Connes fusion product became possible because we found a specific way  to relate string structures, loop fusion, and  Connes fusion:
\begin{enumerate}[(1)]

\item 
We use the recent discovery of a certain loop-fusion product on Killingback's spin structure on loop space \cite{waldorfa,waldorfb}. This fusion product belongs to a loop space equivalent formulation of a string structure on $M$. 

\item
A similar,  Lie group-theoretical loop-fusion product exists on the basic central extension $\BCE$, and it satisfies a certain compatibility condition with the fusion product of (1). In our previous work \cite{Kristel2019} we have found  an operator-theoretic description of this loop-fusion product, using Tomita-Takesaki theory of the free fermions $\mathcal{F}$. 

\item
We use functoriality of Connes fusion, together with the fact that the free fermions $\mathcal{F}$ are a standard form, to construct a Connes fusion product of certain operators on $\mathcal{F}$ (\cref{eq:ImpFusionDiagram}). We derive a new and crucial relation (\cref{thm:FusionProductsAgree}) between   the loop-fusion product of (2) and this Connes fusion product. 

\end{enumerate}
Our construction of the Connes fusion product of spinors in \cref{th:fusionfibrewise} makes use of all these three fusion products and of the relations among them.

The question for the existence of  an associative Connes fusion product on $\mathcal{F}(LM)$ was formulated as \quot{Theorem 1} in \cite{stolz3}. In addition to proving its (fibrewise) existence, we address and solve the question in which way these structures can be upgraded into a \emph{smooth} setting. In our previous work \cite{Kristel2020}, we started this by generalizing the concept of a \emph{rigged Hilbert space} to  \cstar-algebras and bundle versions of both. In this article, 
we further extend this framework to von Neumann algebras, and exhibit our von Neumann algebra bundle $\mathcal{N}$ over $PM$ as a \emph{rigged} von Neumann algebra bundle, and the spinor bundle $\mathcal{F}(LM)$ as a \emph{rigged} von Neumann $p_1^{*}\mathcal{N}$-$p_2^{*}\mathcal{N}$-bimodule bundle (\cref{lem:vonNeumannBimodule}). Moreover, we show that our Connes fusion product on the spinor bundle is smooth with respect to these rigged structures (\cref{prop:smoothness2}).

Our results here are part of a programme to explore the relation between functorial field theories, generalized cohomology theories, and index theory. 
Within this programme there are, among others, two big (and still open) problems, and we believe that the results of this article make a contribution to them. 

The first problem is to realize the 2-dimensional supersymmetric sigma model (the \quot{free superstring}) rigorously as a smooth  functorial field theory (FFT) in the sense of Atiyah, Segal, and Stolz-Teichner \cite{atiyah1,Segal1991,ST04}. In this sigma model, the Hilbert space  $\mathcal{F}(LM)_{\gamma}$ is the \quot{state space} for  superstrings with underlying \quot{world-sheet embedding} $\gamma:S^1 \to M$; in other words, it is the value of the FFT on the circle (equipped with map $\gamma$). Our result that the spinor bundle is a \emph{rigged} Hilbert space bundle is at the basis of the statement that this FFT is smooth. The main contribution is, however, that our Connes fusion product on $\mathcal{F}(LM)$ gives one ingredient of the answer to the question what the FFT assigns to a pair of pants. Another contribution is that our rigged von Neumann algebra bundle $\mathcal{N}$ over $PM$ is part of the answer to the question how the FFT can be extended to a point: we discuss in \cref{sec:stringorbundle} that $\mathcal{N}$ furnishes a certain  \emph{2-vector bundle} over $M$, the \quot{stringor bundle} of the string manifold $M$. 

Yet, two major steps  are still missing in the construction of the supersymmetric sigma model as a smooth extended FFT. First, a proper definition of an appropriate smooth bordism category over $M$ (probably, it will be a sheaf of $(\infty,2)$-categories). There have been very recent promising proposals \cite{Ludewig2020,Grady2020} raising hope that this is accomplishable. Second,  a construction and investigation of a \emph{connection} on the spinor bundle $\mathcal{F}(LM)$. Some ideas and results in this direction have already been reported by Stolz-Teichner under the name \quot{string connection} \cite{ST04} or \quot{conformal connection} \cite{stolz3}. Further results have been obtained in \cite{waldorf8,waldorfb}, discussing string connections in terms of Killingback's spin structures on loop spaces. All together, we believe that a complete solution to the problem to cast the free superstring as an extended smooth FFT is now in reach.

The second problem that arises in the above-mentioned programme is the quest for a Dirac type operator acting on spinors on loop space \cite{witten2,Taubes1989}. Admittedly, our results do not yet place this goal in easy reach.
However, they might be relevant in order to specify the precise space of spinors on which such an operator may act. At this point, we would like to highlight the philosophy behind fusion on loop space, following Stolz-Teichner \cite{stolz3}: fusion characterizes geometric structure on $LM$ that subtly encodes geometry on $M$. This is supported by the following list of incarnations of this philosophy:
\begin{enumerate}[(a)]

\item
A fusion product on a line bundle on $LM$ encodes a bundle gerbe on $M$; similarly, a fusion product on a central extension of a loop group $LG$ encodes a multiplicative bundle gerbe on $G$. The precise statements are the main results of \cite{waldorf11,Waldorfc}.

\item 
A fusion product on a Killingback spin structure on $LM$ encodes a string structure on $M$, as  mentioned above. The precise statement is in \cite{waldorfb}.

\item
Our Connes fusion product on the spinor bundle on $LM$ encodes the stringor bundle on $M$, as outlined above and described in \cref{sec:stringorbundle}. 

\item
Finally, the following lower-categorical analog of (b) has been proved by Stolz-Teichner \cite{stolz3}: the orientation bundle of $LM$ carries a canonical fusion product, and the fusion-preserving sections encode precisely the spin structures on $M$.  

\end{enumerate}
In analogy to (d), we expect that the relevant space of spinors on loop space  consists of \emph{fusion-preserving} sections of  the spinor bundle $\mathcal{F}(LM)$. In this sense, we believe that our construction of the Connes fusion product on the spinor bundle will help to approach the mysterious Dirac operator on loop space.

For completeness, we remark that both problems described above are conjectured to be related by a yet unknown index theory with family indices
taking values in a generalized cohomology theory, whose objects  can be represented by certain smooth FFTs. It is expected that the index of the Dirac operator on loop space corresponds  precisely to the free superstring FFT \cite{ST04,stolz3}.

This article is organized as follows.
In \cref{sec:HilbertBundle} we develop the theory of rigged von Neumann algebra bundles and bimodule bundles, first over Fr\'echet manifolds and then over diffeological spaces (the space of tuples of paths with common initial point and common end point are diffeological spaces). In \cref{sec:freefermions} we discuss the free fermions $\mathcal{F}$, exhibit it as a rigged von Neumann bimodule, and construct the Connes fusion of certain operators on $\mathcal{F}$. \cref{sec:fusion} is devoted to all aspects of loop-fusion, including a discussion of the relation between Killingback spin structures and string structures. In \cref{sec:SpinorBundleOnLoopSpaceI} we carry out our main constructions: the rigged von Neumann algebra bundle $\mathcal{N}$, the $\mathcal{N}$-$\mathcal{N}$-bimodule structure on the spinor bundle $\mathcal{F}(LM)$, and finally, the Connes fusion product on $\mathcal{F}(LM)$. We include two appendices about bimodules of von Neumann algebras: \cref{sec:StandardForms}  collects  results about the theory of standard forms that we mainly use in \cref{sec:ReflectionFreeFermions}, and \cref{sec:ConnesFusion} contains definitions and properties of Connes fusion.

\paragraph{Acknowledgements. } This project was funded by the German Research Foundation (DFG) under project code WA 3300/1-1. We would like to thank Andr\'e Henriques, Bas Janssens, Matthias Ludewig, and Peter Teichner for helpful discussions.

\section{Rigged von Neumann algebras} 
\label{sec:HilbertBundle}

One of the central objects of consideration in this article are certain bundles of von Neumann algebras over the space of smooth paths/loops in a manifold.
It is highly desirable that these bundles are equipped with a smooth structure, because these bundles are expected to host interesting differential operators.
However, to the best of our knowledge, no treatment of smooth bundles of Hilbert spaces, \cstar-algebras, or von Neumann algebras is available.
In our previous work \cite{Kristel2020} on spinor bundles on loop space, we found that these smoothness issues are best addressed in the setting of \emph{rigged} Hilbert spaces, which we extended there to rigged \cstar-algebras and smooth bundles thereof.

In this section, we expand these notions further to include  rigged von Neumann algebras and smooth bundles of rigged von Neumann algebras, as well as rigged bimodules over von Neumann algebras and bundles of these. Then, in order to discuss fusion in loop spaces, we further extend these structures from Fr\'echet manifolds to  diffeological spaces.

\subsection{Representations on rigged Hilbert spaces}
\label{sec:HilbertBundle:1}

In this paper, we work with rigged Hilbert spaces and rigged \cstar-algebras as introduced in Section 2 of our paper \cite{Kristel2020}; we briefly review all required notions and results, and refer the reader to that paper for more motivation and context.

\begin{definition}
\label{def:rhs}
A \emph{rigged Hilbert space}  is a Fr\'echet space equipped with a continuous (sesquilinear) inner product. A \emph{morphism of rigged Hilbert spaces} is simply a continuous linear map.
 A morphism of rigged Hilbert spaces is called \emph{bounded}/\emph{isometric} if it is bounded/isometric with respect to the  inner products, and it is called \emph{unitary}, if it is an isometric isomorphism.
\end{definition}
\begin{comment}
 Classes of morphisms that appear:
 \begin{itemize}
  \item Isometric isomorphisms: Local trivializations; the action of a fixed element of $\O_{L}(V)$ on $V$, etc.
  \item Bounded morphisms: The action of a fixed element of $\Cl(V)^{\smooth}$ on $\mc{F}_{L}^{\smooth}$.
 \end{itemize}
\end{comment}

Given a rigged Hilbert space $E$, one obtains an honest Hilbert space, denoted $E^{\langle \cdot, \cdot \rangle}$, by completion with respect to the inner product.
The prototypical example of a rigged Hilbert space is the Fr\'echet space of smooth functions on the circle, equipped with the $L^2$-inner product; its Hilbert completion is the space of square-integrable functions on the circle.

By a \emph{smooth representation} of a Fr\'echet Lie group $\mathcal{G}$ on a rigged Hilbert space $E$ we mean an action  of $\mathcal{G}$ on $E$ by unitary morphisms of rigged Hilbert spaces, such that the map $\mathcal{G} \times E \to E$ is smooth.
 The typical example in this paper is the rigged Hilbert space $\mathcal{F}_L^{\smooth}$ of smooth vectors in a Fock space of a Lagrangian $L\subset V$, which carries a smooth representation of a  Lie group $\Imp_L^{\theta}(V)$ of implementers, see \cref{prop:smoothFockSpace}.

\begin{definition}
\label{def:riggedcstaralgebra}
 A \emph{rigged \cstar-algebra} is a Fr\'echet algebra $A$, equipped with a continuous norm $\| \cdot \|: A \rightarrow \R_{\geqslant 0}$ and a continuous complex anti-linear involution $*:A \rightarrow A$, such that its completion  with respect to the norm is a \cstar-algebra. A \emph{morphism of rigged \cstar-algebras} is a morphism of Fr\'echet algebras that is bounded with respect to the norms and intertwines the involutions. A morphism of rigged \cstar-algebras is called \emph{isometric}, if it is an isometry with respect to the norms.
\end{definition}
\begin{comment}
 Classes of morphisms that appear:
 \begin{itemize}
  \item Isometric isomorphisms: Local trivializations; the action of a fixed element of $\O_{L}(V)$ on $\Cl(V)^{\smooth}$, etc.
  \item Isometric morphisms: The inclusion of $\Cl(V_{-})^{\smooth}$ into $\Cl(V)^{\smooth}$.
  \item Bounded morphisms: The left/right multiplication of a fixed element of $\Cl(V)^{\smooth}$ on $\Cl(V)^{\smooth}$.
 \end{itemize}
\end{comment}

By definition, an actual \cstar-algebra $A^{\|\cdot\|}$ can be  obtained from a rigged \cstar-algebra $A$ by norm completion. By a \emph{smooth representation} of a Fr\'echet Lie group $\mathcal{G}$ on a rigged \cstar-algebra $A$ we mean an action of $\mathcal{G}$ on $A$ by isometric isomorphisms of rigged \cstar-algebras, such that the map $\mathcal{G} \times A \to A$ is smooth.
The typical example in this paper is the Fr\'echet algebra $\Cl(V)^{\smooth}$ of smooth vectors in the Clifford algebra of a real Hilbert space $V$, which carries a smooth representation of the orthogonal group $\O(V)$ via Bogoliubov automorphisms, see \cref{prop:CliffordFrechetAlgebra}.

Of major importance for this article is a certain type of representations of rigged \cstar-algebras on rigged Hilbert spaces.

\begin{definition}\label{def:RiggedRepresentation}
Let $A$ be a rigged \cstar-algebra. A \emph{rigged $A$-module} is a rigged Hilbert space $E$ together with a  representation $\rho$
 of (the underlying algebra of) $A$ on (the underlying vector space of) $E$, such that the map $\rho:A \times E \to E$ is smooth and the following conditions hold for all $a \in A$ and all $v,w \in E$:
 \begin{equation}\label{eq:BoundednessCriteria}
  \langle \rho(a,v), \rho(a,v) \rangle \leqslant \| a \|^{2} \langle v, v \rangle \quad \text{ and } \quad \langle \rho(a,v),w \rangle = \langle v, \rho(a^{*},w) \rangle.
 \end{equation}
 A \emph{(unitary)} \emph{intertwiner} from a rigged $A_1$-module $E_1$ to a rigged $A_2$-module $E_2$ is a pair $(\phi,\psi)$ of a  morphism $\phi:A_1 \to A_2$ of rigged \cstar-algebras and a bounded (unitary) morphism $\psi:E_1\to E_2$ of rigged Hilbert spaces that intertwines the representations along $\phi$.
\end{definition}

\begin{comment}
Since $\rho:A \times E \to E$ is bilinear, it is smooth if and only if it is continuous. 
\end{comment}
\begin{comment}
Note that if $\psi:E_1 \to E_2$ is unitary, then $\phi$ is automatically isometric. In particular, in an unitary intertwiner  $(\phi,\psi)$,  $\phi$ is an isometric $\ast$-isomorphism. 
\end{comment}

The typical example is the rigged Hilbert space $\mathcal{F}_L^{\smooth}$, which becomes under Clifford multiplication a  rigged $\Cl(V)^{\smooth}$-module, see \cref{prop:SmoothCliffordActionOnF}. A couple of remarks are in order. 

\begin{remark}
\label{ex:OppositeRep}
\begin{enumerate}[(a)]

\item 
Conditions \cref{eq:BoundednessCriteria} in \cref{def:RiggedRepresentation} are chosen such that a rigged $A$-module $E$ with underlying representation $\rho$ induces a  $\ast$-homomorphism $\rho^{\vee}: A^{\|\cdot \|} \to \mathcal{B}(E^{\left \langle  \cdot,\cdot \right \rangle})$, i.e., a representation of the \cstar-algebra $A^{\|\cdot\|}$ on the Hilbert space $E^{\left \langle \cdot,\cdot  \right \rangle}$, see \cite[Rem. 2.2.11]{Kristel2020}.

\item
\label{re:inducedcstarrep}
If $\phi:A \to B$ is a morphism of rigged \cstar-algebras, and $E$ is a rigged $B$-module with representation $\rho$, then $E$ becomes a rigged $A$-module under the induced representation $(a,v) \mapsto \rho(\phi(a),v)$. Moreover, the pair $(\phi,\id)$ is an isometric  intertwiner. 

\item
Rigged \cstar-algebras and rigged modules  are compatible with dualization.
If $A$ is a rigged \cstar-algebra, then its opposite algebra, $A^{\mathrm{opp}}$, is a rigged \cstar-algebra in a natural way, and its completion $(A^{\opp})^{\|\cdot\|}$ is the usual opposite \cstar-algebra.
 Let $E$ be a rigged $A$-module.
 The inner product on $E$ gives us a complex anti-linear injection $\iota: E \rightarrow E^{*}$ mapping $E$ into its continuous linear dual $E^{*}$.
 Denote the image of $\iota$ by $E^{\sharp}$.
 We turn $E^{\sharp}$ into a  Fr\'{e}chet space using the identification with $E$.
 If $\rho$ denotes the representation of $A$ on $E$, then the map 
 \begin{align*}
  \rho^{\sharp}: A^{\mathrm{opp}} \times E^{\sharp} \rightarrow E^{\sharp},
  (a,\ph) &\mapsto \ph \circ \rho(a,-),
 \end{align*}
 is a  representation of  $A^{\mathrm{opp}}$ on  $E^{\sharp}$, and it is straightforward to show that it turns $E^{\sharp}$ into a rigged $A^{\opp}$-module.

\end{enumerate}
\end{remark}

Next we extend our framework, which we have so far recalled from \cite{Kristel2020}, by introducing new structures that ultimately lead to the notion of rigged bimodules over rigged von Neumann algebras.

\begin{definition}
\label{def:rbim}
Let $A_1$ and $A_2$ be rigged \cstar-algebras.
 A \emph{rigged $A_1$-$A_2$-bimodule} is a rigged Hilbert space $E$ 
with commuting representations of $A_1$ and $A_2^{\opp}$, in such a way that $E$  is a rigged $A_1$-module and a rigged $A_2^{\opp}$-module. A \emph{(unitary)} \emph{intertwiner} from a rigged $A_1$-$A_2$-bimodule $E$ to a rigged $A'_1$-$A_2'$-bimodule $E'$ is a triple $(\phi^1,\phi^2,\psi)$ consisting of   morphisms $\phi^1:A_1 \to A'_1$ and $\phi^2:A_2 \to A_2'$ of rigged \cstar-algebras, and of a bounded (unitary) morphism $\psi:E \to E'$ of rigged Hilbert spaces that intertwines both representations along $\phi^1$ and $\phi^2$. 
\end{definition}

\begin{comment}
Thus,
for all $a \in A$, $b \in B^{\mathrm{opp}}$ and all $v \in E$ we have
 \begin{equation*}
  \rho_{A}(a,\rho_{B}(b,v)) = \rho_{B}(b,\rho_{A}(a,v)).
 \end{equation*}
\end{comment}

Under completion, a rigged $A$-$B$-bimodule $E$ becomes a Hilbert space $E^{\left \langle \cdot,\cdot  \right \rangle}$ with commuting representations of the \cstar-algebras $A^{\|\cdot\|}$ and $(B^{\opp})^{\|\cdot\|}$. 
Our main example of a rigged bimodule is again the space $\mathcal{F}_L^{\smooth}$ of smooth vectors in a Fock space, which we equip with a rigged bimodule structure over a certain subalgebra of $\Cl(V)^{\smooth}$, see \cref{sec:ReflectionFreeFermions}. Next, we develop the setting of rigged von Neumann algebras. 

\begin{definition}
\label{def:riggedvonneumannbundle}
 A \emph{rigged von Neumann algebra} is a pair $N=(A,E)$ consisting of a rigged \cstar-algebra $A$ and a rigged $A$-module $E$, with the property that the induced \cstar-representation  $A^{\| \cdot \|} \rightarrow \mc{B}(E^{\langle \cdot, \cdot \rangle })$ is faithful.
 The representation underlying the rigged $A$-module $E$ is called the \emph{defining representation} of $N$.
\end{definition}

\begin{remark}
\label{re:vnacl}
 If $N=(A,E)$ is a rigged von Neumann algebra, and $\rho$  its defining representation, then the assumption that $\rho^{\vee}: A^{\| \cdot \|} \rightarrow \mc{B}(E^{\langle \cdot, \cdot \rangle})$ is faithful implies that it is an isometry, and hence a homeomorphism onto its image.
\begin{comment}
Indeed, an injective $*$-homomorphism between \cstar-algebras is always an isometry.
Moreover, one checks that $\rho^{\vee}$ is injective, if and only if, the restriction of $\rho^{\vee}$ to $A$ is an isometry.
\end{comment}
We may thus identify $\rho^{\vee}(A^{\|\cdot \|})$ with $A^{\| \cdot \|}$.
Then, we define the von Neumann algebra 
\begin{equation*}
N''\defeq(A^{\| \cdot \|})'' \subseteq \mc{B}(E^{\langle \cdot, \cdot \rangle})\text{,}
\end{equation*}
and conclude that it contains $A^{\| \cdot \|}$ as a weakly dense (and thus also $\sigma$-weakly dense) subspace.
In particular, any \emph{rigged} von Neumann algebra $N$ gives rise to an \emph{ordinary} von Neumann algebra $N''$. 
\end{remark}

An example of a rigged von Neumann algebra is the Clifford algebra $\Cl(V)^{\smooth}$ with its defining representation on Fock space $\mathcal{F}^{\smooth}_L$,  see \cref{prop:SmoothCliffordActionOnF}. The associated  von Neumann algebra is $\mathcal{B}(\mathcal{F}_L)$. More interesting examples will be constructed in \cref{cor:SmoothHalfAction}; these give rise to type III$_1$-factors.

\begin{remark}
\label{re:oppositevNA}
If $N=(A,E)$ is a rigged von Neumann algebra, then $N^{\mathrm{opp}}\defeq (A^{\mathrm{opp}},E^{\sharp})$ is a rigged von Neumann algebra, see \cref{ex:OppositeRep}.
\end{remark}

\begin{definition}\label{def:vNaMorphism}
Let $N_{1} = (A_{1},E_{1})$ and $N_{2} = (A_{2},E_{2})$ be rigged von Neumann algebras.
A \emph{morphism} of rigged von Neumann algebras from $N_1$ to $N_2$ is a morphism $\phi: A_{1} \rightarrow A_{2}$ of rigged \cstar-algebras that extends to a normal $\ast$-homomorphism $N_{1}'' \rightarrow N_{2}''$.
\end{definition}

We recall that \quot{normal} means to be continuous in the $\sigma$-weak topologies, and we recall that these topologies are generated by semi-norms $p_{\rho}(T) \defeq | \mathrm{tr}( \rho T) |$, where $\rho$ is a trace-class operator. Since $A^{\| \cdot \|}\subset N''$ is $\sigma$-weakly dense, it is clear that the extension $N_1'' \to N_2''$ is unique, if it exists. The following lemma gives us a useful sufficient criterion for the extendability.

\begin{lemma}\label{lem:SpatialIsomorphisms}
 Let $N_{1} = (A_{1},E_{1})$ and $N_{2} = (A_{2},E_{2})$ be rigged von Neumann algebras.
 Suppose that $(\phi,\psi)$ is a unitary   intertwiner from $E_1$ to $E_2$ in the sense of \cref{def:RiggedRepresentation}.
 Then, $\phi$ is a  morphism of rigged von Neumann algebras.
\end{lemma}

\begin{proof}
 All that needs to be shown is that $\phi :A_{1} \rightarrow A_{2}$ extends to a normal $*$-homomorphism $N_{1}'' \rightarrow N_{2}''$.
 According to \cref{re:vnacl} we view $A_{1}$ as a subset of $\mc{B}(E_{1}^{\langle \cdot, \cdot \rangle })$ and $A_{2}$ as a subset of $\mc{B}(E_{2}^{\langle \cdot, \cdot \rangle})$.
 Now, we have, for all $a \in A_{1}$ and all $v \in E_{2}^{\langle \cdot, \cdot \rangle}$,
 \begin{equation*}
   \phi(a)v = \phi(a) \psi \psi^{*} v = \psi a\psi^{*}v.
 \end{equation*}
Here, $\psi\psi^{*}=\id$ since  $\psi$ is unitary and thus extends to a unitary operator $\psi:E_1^{\left \langle  \cdot,\cdot \right \rangle} \to E_2^{\left \langle  \cdot,\cdot \right \rangle}$ on the completions. 
It follows that $\phi(a) = \psi a \psi^{*}$ as elements of $\mc{B}(E_{2}^{\langle \cdot, \cdot \rangle})$, which implies that the map $C_{\psi}:\mc{B}(E_{1}^{\langle \cdot, \cdot \rangle}) \rightarrow \mc{B}(E_{2}^{\langle \cdot, \cdot \rangle}), a \mapsto \psi a \psi^{*}$ extends $\phi$.
 Next, we prove that $C_{\psi}$ is normal. If $\rho$ is a trace-class operator on $E_{2}^{\langle \cdot, \cdot \rangle}$, then $\psi^{*} \rho \psi$ is a trace-class operator on $E_{1}^{\langle \cdot, \cdot \rangle}$, and we have
 \begin{equation*}
  p_{\psi^{*} \rho \psi}(a) = p_{\rho}(C_{\psi}(a)),
 \end{equation*}
 for all $a \in \mc{B}(E_{1}^{\langle \cdot, \cdot \rangle})$.
 This implies that $C_{\psi}$ is $\sigma$-weakly continuous.
 \begin{comment}
  Because the pre-image of the subbasis open $\{ p_{\rho}(T) < \epsilon \}$ under $C_{\psi}$ is contained in the subbasis open $\{ p_{\psi^{*}\rho \psi}(a) < \epsilon \}$.
 \end{comment}
 Now, because $A_{1}$ is $\sigma$-weakly dense in $N_{1}''$, and $A_{2}$ is $\sigma$-weakly dense in $N_{2}''$ it follows that $C_{\psi}(N_{1}'') \subseteq N_{2}''$, and that, moreover, $C_{\psi}: N_{1}'' \rightarrow N_{2}''$ is a $*$-homomorphism.
\begin{comment}
  Even though the $\sigma$-weak topology is not the weak operator topology, it is true that a $*$-subalgebra $N \subseteq \mc{B}(H)$ is $\sigma$-weakly closed if and only if it is weak operatorly closed.
  (This is Theorem 2.4.11 on p.~72 of ``Operator Algebras and Quantum Statistical Mechanics 1'', by Bratteli and Robinson.)
 \end{comment}
\end{proof}

\begin{remark}
\cref{lem:SpatialIsomorphisms} motivates the following terminology: a \emph{spatial morphism} from a rigged von Neumann algebra $N_1$ to another rigged von Neumann algebra $N_2$ is a unitary  intertwiner $(\phi,\psi)$ between the underlying rigged modules.  \cref{lem:SpatialIsomorphisms}  implies then that $\phi$ is a morphism of rigged von Neumann algebras. 
Additionally, a spatial morphism $(\phi,\psi)$ will be called \emph{invertible}, or spatial \emph{iso}morphism if it is invertible by another spatial morphism. We note that this is the case if and only if $(\phi,\psi)$ is an invertible unitary  intertwiner. Spatial isomorphisms appear frequently in the context of \emph{bundles} of rigged von Neumann algebras, whose local trivializations will be spatial isomorphisms in each fibre, see \cref{def:rvnab,re:loctrivmorph}.
\end{remark}

\begin{comment}
Our guideline is: in \quot{spatial} things all rigged Hilbert space morphisms are unitary.
\end{comment}

Next is the discussion of modules and bimodules for rigged von Neumann algebras.

\begin{definition}
\label{def:smoothrepriggedvN}
Let $N=(A,E)$ be a rigged von Neumann algebra. A \emph{rigged von Neumann $N$-module} is  a rigged $A$-module  $F$ whose induced $\ast$-homomorphism $A^{\| \cdot \|} \rightarrow \mc{B}(F^{\langle \cdot, \cdot \rangle})$ extends to a normal $\ast$-homomorphism  $N'' \rightarrow \mc{B}(F^{\langle \cdot, \cdot \rangle})$.
\end{definition}
\begin{remark}
\label{re:rvNmodule}
Again, the extension $N'' \rightarrow \mc{B}(F^{\langle \cdot, \cdot \rangle})$ in \cref{def:smoothrepriggedvN} is unique, if it exists.
It  guarantees that, for $N$  a rigged von Neumann algebra and $F$  a rigged von Neumann $N$-module, the completion  $F^{\left \langle \cdot,\cdot  \right \rangle}$ is an  $N''$-module in the usual von Neumann-theoretical sense (see \cref{sec:StandardForms}).
Observe moreover that the defining representation of any rigged von Neumann algebras is automatically a rigged von Neumann module.
\end{remark}

We will need the following result, which will be used later in the proof of \cref{rem:vonNeumanBimoduleCompletion}.
\begin{lemma}\label{rem:Extending*Homomorphisms}
Let $N=(A,E)$ be a rigged von Neumann algebra, and let $F$ be a rigged von Neumann $N$-module with underlying representation $\rho:A \times F \to F$. Then, the normal $\ast$-homomorphism $N'' \to \mc{B}(F^{\langle \cdot, \cdot \rangle})$ factors through the von Neumann algebra  
$(\rho^{\vee}(A^{\| \cdot \| }))''\subset \mc{B}(F^{\langle \cdot, \cdot \rangle})$.
\end{lemma}

\begin{comment}
The corresponding diagram is:
\begin{equation*}
\xymatrix{A \ar@{^(->}[d] \\ A^{\|\cdot\|} \ar[dr] \ar[r]^{\rho^{\vee}} \ar@{^(->}[d] & \rho^{\vee}(A^{\| \cdot \| }) \ar@{^(->}[d]\ar@{^(->}[r] & \mc{B}(F^{\langle \cdot, \cdot \rangle}) \\ N'' \ar[r]  \ar@{^(->}[d] & (\rho^{\vee}(A^{\| \cdot \| }))''\ar@{^(->}[ur] \\  \mathcal{B}(E^{\left \langle  \cdot,\cdot \right \rangle}) }
\end{equation*}
\end{comment}

\begin{proof}
Let $a \in N''$ be arbitrary. We wish to prove that its image in $\mc{B}(F^{\langle \cdot, \cdot \rangle})$ in fact lies in $(\rho^{\vee}(A^{\| \cdot \| }))''$.
 So, let $a_{n}$ be a sequence in $A^{\| \cdot \|}$ that $\sigma$-weakly converges to $a$.
 Then, by continuity of $N'' \to \mc{B}(F^{\langle \cdot, \cdot \rangle})$, the image of the sequence converges to some element in $\mc{B}(F^{\langle \cdot, \cdot \rangle})$, but because all elements $a_{n}$ were in $A^{\| \cdot \|}$ to begin with, the limit must already lie in its completion $(\rho^{\vee}(A^{\| \cdot \| }))''$. 
\end{proof}

\begin{comment}
The following definition was not needed in the end:

\begin{definition}\label{def:riggedvonNeumannIntertwiner}
A \emph{(unitary) intertwiner} from a rigged von Neumann $N_1$-module $F_1$ to a rigged von Neumann $N_2$-module $F_2$ is a (unitary) intertwiner $(\phi,\psi)$ from $F_1$ to $F_2$  in the sense of \cref{def:RiggedRepresentation}, such that  $\phi:N_1 \to N_2$ is a morphism of rigged von Neumann algebras.
A \emph{spatial intertwiner} is a triple $(\phi,\varphi,\psi)$ in which $(\phi,\psi)$ is a unitary intertwiner  from $F_1$ to $F_2$, and  $(\phi,\varphi)$ is a spatial morphism from $N_1$ to $N_2$. 
\end{definition}

\begin{remark}\label{rem:IntertwinersExtend}
 Let $(\phi, \psi)$ be an intertwiner from a  rigged von Neumann $N_1$-module  $F_1$ to a rigged von Neumann $N_2$-module $F_2$.
 We denote by the same letters the extensions $\phi:N_{1}'' \rightarrow N_{2}''$ and $\psi: F_1^{\left \langle \cdot,\cdot  \right \rangle} \to F_2^{\left \langle \cdot,\cdot  \right \rangle}$. 
 It then follows from the usual continuity arguments that $\psi$ is an intertwiner along $\phi$ of ordinary von Neumann algebra modules.
  Indeed, let $\{ a_{i} \}$ be a net that ($\sigma$-weakly) converges to $a$.
  Then, the net $\psi( a_{i} \lact v)$ converges to $\psi(a \lact v)$.
  At the same time we have $\psi(a_{i} \lact v) = \phi(a_{i}) \lact \psi(v)$, which converges to $\bar{\phi}(a) \lact \psi(v)$, and we are done.
\end{remark}
\end{comment}

Finally, we come to bimodules for rigged von Neumann algebras.
One of the fundamental results of this works exhibits the Fock space $\mathcal{F}^{\smooth}$ as a rigged von Neumann bimodule, see \cref{prop:FockRiggedBimodule}. 

\begin{definition}
\label{def:riggedvonNeumannbimodule}
 If $N_1=(A_1,E_1)$ and $N_2=(A_2,E_2)$ are rigged von Neumann algebras, then a \emph{rigged von Neumann $N_1$-$N_2$-bimodule}  is a rigged $A_1$-$A_2$-bimodule $F$ that is a rigged von Neumann $N_1$-module and a rigged von Neumann $N_2^{\opp}$-module.
\begin{comment}
The corresponding diagram is:
\begin{equation*}
\xymatrix{A_1 \ar@{^(->}[d] &&&&A_2 \ar@{^(->}[d] \\ A^{\|\cdot\|}_1 \ar[rr]^-{\rho_1^{\vee}} \ar@{^(->}[d] && \mc{B}(F^{\langle \cdot, \cdot \rangle}) && A_2^{\|\cdot\|} \ar@{^(->}[d] \ar[ll]_{\rho_2^{\vee}} \\ N''_1 \ar@/_1pc/[urr] \ar@{^(->}[d] &&&& N_2''\ar@/^1pc/[ull] \ar@{^(->}[d]\\  \mathcal{B}(E_1^{\left \langle  \cdot,\cdot \right \rangle}) &&&& \mathcal{B}(E_2^{\left \langle  \cdot,\cdot \right \rangle})}
\end{equation*}
\end{comment}
A \emph{(unitary)} \emph{intertwiner} from a rigged von Neumann $N_1$-$N_2$-bimodule $F$ to a rigged von Neumann $\tilde N_1$-$\tilde N_2$-bimodule $\tilde F$ is a (unitary) intertwiner $(\phi^1,\phi^2,\psi)$ between the underlying rigged bimodules in the sense of \cref{def:rbim}, such that  $\phi^1:N_1 \to \tilde N_1$ and $\phi^2:N_2 \to \tilde N_2$ are morphisms of rigged von Neumann algebras.
A \emph{spatial intertwiner} from $F$ to $\tilde F$ is a quintuple $(\phi^1,\psi^1,\phi^2,\psi^2,\psi)$ in which $(\phi^1,\phi^2,\psi)$ is a unitary intertwiner from $F$ to $\tilde F$,  and $(\phi^1,\psi^1):N_1 \to \tilde N_1$ and $(\phi^2,\psi^2):N_2 \to \tilde N_2$ are spatial morphisms.
\end{definition}

Here, \cref{lem:SpatialIsomorphisms} implies again that spatial intertwiners are in particular (unitary) intertwiners.  
We assure by the following result that completion brings us  into the classical setting, see \ref{sec:StandardForms}.

\begin{lemma}
\label{rem:vonNeumanBimoduleCompletion}
If $F$ is a rigged von Neumann $N_1$-$N_2$-bimodule, then its completion $F^{\left \langle \cdot,\cdot  \right \rangle}$ is a  $N_1''$-$N_2''$-bimodule in the ordinary von Neumann-theoretical sense.
Likewise, any (unitary) intertwiner between rigged von Neumann bimodules induces a bounded (unitary) intertwiner in the ordinary sense.
\end{lemma}

\begin{proof}
 For the first statement, we only have to argue that the actions of the von Neumann algebras $N_1''$ and $N_2''$ on $F^{\left \langle  \cdot,\cdot \right \rangle}$  commute.
 Indeed, as the actions of the rigged \cstar-algebras commute, it is clear that the images $\rho_1^{\vee}(A_1^{\|\cdot\|})$ and $\rho_2^{\vee}(A_2^{\|\cdot\|})$ in $\mc{B}(F^{\langle \cdot, \cdot \rangle})$ commute, for instance,  $\rho_2^{\vee}(A_2^{\|\cdot\|})\subset \rho_1^{\vee}(A_1^{\|\cdot\|})'$.
 Taking commutants, we obtain $\rho_1^{\vee}(A_1^{\|\cdot\|})''\subset \rho_2^{\vee}(A_2^{\|\cdot\|})'=\rho_2^{\vee}(A_2^{\|\cdot\|})'''$.
 This shows that the von Neumann algebras $\rho_1^{\vee}(A_1^{\|\cdot\|})''$ and $\rho_2^{\vee}(A_2^{\|\cdot\|})''$ commute.
 Now, \cref{rem:Extending*Homomorphisms} proves the claim.
 The statement about intertwiners follows from the usual continuity arguments. \begin{comment}
Earlier, this was
\cref{rem:IntertwinersExtend}.
\end{comment}
\end{proof}

We remark that ordinary von Neumann algebras, bimodules and intertwiners form a bicategory, in which the composition of morphisms is given by the Connes fusion of bimodules \cite{Brouwer2003,ST04}. We have, unfortunately, not yet been able to lift Connes fusion to the setting of  rigged von Neumann bimodules, and thus, there is no corresponding bicategory of rigged von Neumann algebras. This is an important issue that we are going to address in future work.

\subsection{Locally trivial rigged bundles}

\label{sec:locallytrivialbundles}
\label{sec:riggedvonNeumann}

In this section we introduce locally trivial bundles of rigged von Neumann algebras and rigged von Neumann bimodules over Fr\'echet manifolds. The prerequisite notions of rigged Hilbert space bundles, rigged \cstar-algebra bundles, and rigged module bundles have been introduced in Section 2 of \cite{Kristel2020}, and are discussed there in more detail. Throughout this section, we let $\mathcal{M}$ be a Fr\'echet manifold.

\begin{definition}\label{def:HilbertBundle}
Let $E$ be a rigged Hilbert space. A \emph{rigged Hilbert space bundle} over $\mc{M}$ with typical fibre $E$ is a Fr\'echet vector bundle  $\pi:\mc{E} \to \mathcal{M}$  with typical fibre $E$ equipped with a map $g:\mc{E} \times_{\pi} \mc{E} \rightarrow \C$, such that $g$ is fibrewise an inner product, and local trivializations $\Phi: \mathcal{E}|_U \to E \times U$ of $\mathcal{E}$ can be chosen to be fibrewise isometric. A \emph{morphism of rigged Hilbert space bundles} is a morphism of the underlying Fr\'{e}chet vector bundles. A morphism is called \emph{locally bounded}/\emph{isometric} if it is locally bounded/isometric with respect to the inner products, and it is called \emph{unitary} if it is an isometric isomorphism of vector bundles.
\end{definition}

It is straightforward to see that the fibres of a rigged Hilbert space bundle $\mathcal{E}$ are rigged Hilbert spaces \cite[Rem.~2.1.10]{Kristel2020}, and that the fibrewise completion of $\mathcal{E}$ is a  locally trivial continuous Hilbert space bundle over $\mc{M}$ with typical fibre $E^{\left \langle \cdot,\cdot  \right \rangle}$ \cite[Lem.~2.1.13]{Kristel2020}.
Here, a locally trivial continuous Hilbert space bundle with fibre $H$ has continuous local transition functions $U \times H \rightarrow H$, or equivalently,  \emph{strongly} continuous maps $U \rightarrow \U(H)$.
Likewise,   on the level of morphisms, any locally bounded morphism of rigged Hilbert space bundles extends uniquely to a continuous morphism of the corresponding continuous Hilbert space bundles. The following lemma \cite[Prop. 2.1.15]{Kristel2020} shows that our notion of smooth representations (see \cref{sec:HilbertBundle:1}) fits well into the context of rigged Hilbert space bundles.

\begin{lemma}
\label{lem:AssociatedFrechetBundle}
Let $\mathcal{G}$ be a Fr\'echet Lie group, $\mathcal{P}$ be a Fr\'echet principal $\mathcal{G}$-bundle over $\mathcal{M}$, and let $\rho: \mathcal{G} \times E \to E$ be a smooth representation of $\mathcal{G}$ on a rigged Hilbert space $E$. Then, the  associated bundle $(\mathcal{P} \times E)/\mathcal{G}$ is a rigged Hilbert space bundle with typical fibre $E$ in  a unique way, such that every local trivialization $\Phi: \mathcal{P}|_U \to \mathcal{G} \times U$ of $\mathcal{P}$ induces a local trivialization $[p,v] \mapsto (\rho(g(p),v),\pi(p))$ of $(\mathcal{P} \times E)/\mathcal{G}$, where $p\in \mathcal{P}$, $v\in E$, and $\Phi(p)=(g(p),\pi(p))$.
\end{lemma}

The spinor bundle on loop space is a rigged Hilbert space bundle over the loop space $\mathcal{M}=LM$, and it is defined using \cref{lem:AssociatedFrechetBundle}, see \cref{def:SpinorBundle}.
We proceed similarly for rigged \cstar-algebra bundles.

\begin{definition}\label{def:cstarBundle}
Let $A$ be a rigged \cstar-algebra.
 A \emph{rigged \cstar-algebra bundle} over $\mc{M}$ with typical fibre $A$ is a Fr\'echet vector bundle $\pi: \mc{A} \rightarrow \mc{M}$, equipped with
 \begin{itemize}
 \item a map $\| \cdot \|: \mc{A} \rightarrow \R_{\geqslant 0}$, and
 \item fibre-preserving maps $m: \mc{A} \times_{\pi} \mc{A} \rightarrow \mc{A}$ and $*: \mc{A} \rightarrow \mc{A}$,
 \end{itemize}
 such that the following conditions hold for each $x \in \mc{M}$:
 \begin{itemize}
    \item The map $\| \cdot \|_{x} : \mc{A}_{x} \rightarrow \R_{\geqslant 0}$ is a norm.
    \item The maps $m_{x} : \mc{A}_{x} \times \mc{A}_{x} \rightarrow \mc{A}_{x}$ and $*_{x}: \mc{A}_{x} \rightarrow \mc{A}_{x}$ turn  $\mc{A}_{x}$ into a $*$-algebra.
    \item There exists a local trivialization around $x$ that is fibrewise an isometric $\ast$-homomorphism.
 \end{itemize}
 A \emph{morphism of rigged \cstar-algebra bundles} over $\mathcal{M}$ is a morphism $\phi: \mathcal{A}_1 \to \mathcal{A}_2$ of Fr\'echet vector bundles that is fibrewise a morphism of $\ast$-algebras and locally bounded with respect to the norms.
A morphism is called \emph{isometric} if it is isometric with respect to the norms.
\end{definition}

One can  show that each fibre of a rigged \cstar-algebra bundle $\mathcal{A}$ is a rigged \cstar-algebra, and that the opposite multiplication produces a rigged  \cstar-algebra bundle $\mathcal{A}^{\opp}$ with typical fibre $A^{\opp}$.
Further, the fibrewise norm completion gives a locally trivial continuous bundle of \cstar-algebras with typical fibre $A^{\|\cdot\|}$, and strongly continuous transition functions \cite[Lem.~2.2.6]{Kristel2020}. Likewise, any morphism of rigged \cstar-algebra bundles extends uniquely to a continuous morphism of continuous bundles of \cstar-algebras.
In order to avoid confusion, we remark that this notion of a continuous bundle of \cstar-algebras is not the same as the notion of a continuous field of \cstar-algebras, which is typically not locally trivial.

Rigged \cstar-algebra bundles can be obtained by associating a smooth representation to a principal bundle, as the following result \cite[Prop.~2.2.8]{Kristel2020} shows.

\begin{lemma}
\label{lem:cstarbundleass}
Let $\mathcal{P}$ be a Fr\'echet principal $\mc{G}$-bundle over $\mathcal{M}$, and let $\rho: \mathcal{G} \times A \to A$ be a smooth representation on a rigged \cstar-algebra $A$. Then, the associated bundle $(\mathcal{P} \times A)/\mathcal{G}$ is a rigged \cstar-algebra bundle with typical fibre $A$  in  a unique way, such that every local trivialization $\Phi: \mathcal{P}|_U \to \mathcal{G} \times U$ of $\mathcal{P}$ induces a local trivialization $[p,a] \mapsto (\rho(g(p),a),\pi(p))$ of $(\mathcal{P} \times A)/\mathcal{G}$, where $p\in \mathcal{P}$, $a\in A$, and $\Phi(p)=(g(p),\pi(p))$.
\end{lemma}

The Clifford bundle on loop space is a rigged \cstar-algebra bundle on $LM$, and it is defined using \cref{lem:cstarbundleass}, see \cref{def:CliffordBundle}.
Next, we discuss module bundles and bimodule bundles for rigged \cstar-algebra bundles.

\begin{definition}
\label{def:repcstarbundle}
Let $A$ be a rigged \cstar-algebra and  $E$ be a rigged $A$-module, with representation $\rho_0$, and let $\mathcal{A}$ be a rigged \cstar-algebra bundle over $\mathcal{M}$ with typical fibre $A$. A \emph{rigged $\mathcal{A}$-module bundle with typical fibre $E$} is a rigged Hilbert space bundle $\mathcal{E}$ with typical fibre $E$, and  a fibre-preserving map
 \begin{equation*}
  \rho: \mc{A} \times_{\mathcal{M}} \mc{E} \rightarrow \mc{E}
 \end{equation*}
 with the property that around every point in $\mathcal{M}$ there exist local trivializations $\Phi$ of $\mc{A}$ and $\Psi$ of $\mc{E}$  that fibrewise intertwine $\rho$ with $\rho_0$, i.e.~we have
$\Psi_x(\rho(a,v))=\rho_0(\Phi_x(a),\Psi_x(v))$
for all $x\in \mathcal{M}$ over which $\Phi$ and $\Psi$ are defined, and all $a\in \mathcal{A}_x$ and $v\in \mathcal{E}_x$.
A pair $(\Phi,\Psi)$ of local trivializations with this property   is called \emph{compatible}. 
\end{definition}

One can easily show that $\rho$ is automatically a morphism of Fr\'echet vector bundles.
Furthermore, for each $x \in \mc{M}$, the map $\rho_{x}$ turns $\mathcal{E}_x$ into a rigged $\mathcal{A}_x$-module; and every pair of compatible local trivializations $(\Phi,\Psi)$ around $x$ yields an invertible unitary intertwiner $(\Phi_x,\Psi_x)$ between the rigged $\mathcal{A}_x$-module $\mathcal{E}_x$ and the rigged $A$-module $E$. The definition of intertwiners between rigged module bundles is the natural one.

\begin{definition}
\label{def:intertwinerriggedmodulebundles}
A \emph{(unitary)} \emph{intertwiner} between a rigged $\mathcal{A}_1$-module bundle $\mathcal{E}_1$ and a rigged $\mathcal{A}_2$-module bundle $\mathcal{E}_2$ is a pair $(\Phi,\Psi)$ of a  morphism $\Phi:\mathcal{A}_1\to \mathcal{A}_2$ of rigged \cstar-algebra bundles, and a locally bounded (unitary)  morphism $\Psi:\mathcal{E}_1\to \mathcal{E}_2$ of rigged Hilbert space bundles, such that $(\Phi_x,\Psi_x)$ is a (unitary) intertwiner of rigged modules in the fibre over each point $x\in \mathcal{M}$.  
\end{definition}

\begin{comment}
There is one subtlety concerning induced module bundle structures, which we shall point out since it will occur later. If $\mathcal{E}$ is a rigged $\mathcal{B}$-module bundle, and $\Phi: \mathcal{A} \to \mathcal{B}$ is a morphism of rigged \cstar-algebra bundles, then $\Phi$ does not automatically induce on $\mathcal{E}$ the structure of a rigged $\mathcal{A}$-module bundle.
The reason is our strong requirement that module bundles have a typical fibre: if $\mathcal{E}$ has as its typical fibre a rigged $B$-module $E$, then in general no typical fibre for the induced rigged $\mathcal{A}$-module bundle is available. 
\end{comment}
The definition of a rigged \emph{bimodule} bundle now follows naturally; it is, however, important enough that we give it in full detail.

\begin{definition}
\label{def:cstarbimodulebundle}
Let $\mathcal{A}_1$ and $\mathcal{A}_2$ be rigged \cstar-algebra bundles over $\mathcal{M}$ with typical fibres $A_1$ and $A_2$, respectively, and let $E$ be a rigged $A_1$-$A_2$-bimodule.  
A \emph{rigged $\mathcal{A}_1$-$\mathcal{A}_2$-bimodule bundle $\mathcal{E}$ with typical fibre $E$} is a rigged Hilbert space bundle $\mathcal{E}$ over $\mathcal{M}$ that is both a rigged $\mathcal{A}_1$-module bundle and a rigged $\mathcal{A}_2^{\opp}$-module bundle, such that around every point in $\mathcal{M}$ there exist local trivializations  $\Phi^1$ of $\mc{A}_1$, $\Phi^2$ of $\mathcal{A}_2$,  and $\Psi$ of $\mc{E}$ with both $(\Phi^1,\Psi)$ and $(\Phi^2,\Psi)$ compatible. A triple $(\Phi^1,\Phi^2,\Psi)$ of local trivializations with this property is again called \emph{compatible}.  
\begin{comment}
The diagram
  \begin{equation*}
  \xymatrix{
  \mc{A}_U \times_{\mathcal{M}} \mc{E}_{U} \ar[r] \ar[d]_{\Phi_{\mathcal{A}} \times \Psi} & \mc{E}|_{U} \ar[d]^{\Psi} & \mathcal{E}_U \times_{\mathcal{M}} \mathcal{B}_U \ar[d]^{\Psi \times \Phi_{\mathcal{B}}} \ar[l] \\
  (A \times U)\times_U (E \times U) \ar[d]_{\pr_A \times \pr_E} & E \times U \ar[d]^{\pr_E} & (E \times U) \times (B \times U) \ar[d]\\A \times E \ar[r]_-{\rho_A} & E & E \times B \ar[l]^{\rho_B}\text{.} 
  }
 \end{equation*}
 is commutative.
\end{comment} 
\end{definition}

It is straightforward to see that  the fibres $\mathcal{E}_x$ are rigged $(\mathcal{A}_1)_x$-$(\mathcal{A}_2)_x$-bimodules,  for each $x\in \mathcal{M}$, and that compatible local trivializations around $x$ establish an invertible unitary intertwiner between $\mathcal{E}_x$ and $E$.
For completeness, we also note the bundle version of a bimodule intertwiner.

\begin{definition}
\label{def:cstarbimodintertwiner}
A \emph{(unitary) intertwiner} from a rigged $\mathcal{A}_1$-$\mathcal{A}_2$-bimodule bundle $\mathcal{E}$ to a rigged $\widetilde{\mathcal{A}}_1$-$\widetilde{\mathcal{A}}_2$-bimodule bundle $\widetilde{\mathcal{E}}$ is a triple $(\Phi^1,\Phi^2,\Psi)$ consisting of  morphisms $\Phi^1:\mathcal{A}_1\to \widetilde{\mathcal{A}}_1$ and $\Phi^2:\mathcal{A}_2\to \widetilde{\mathcal{A}}_2$ of rigged \cstar-algebra bundles and of a locally bounded (unitary) morphism $\Psi: \mathcal{E} \to \widetilde{\mathcal{E}}$ of rigged Hilbert space bundles, such that over each point $x\in \mathcal{M}$ the triple $(\Phi^1_x,\Phi^2_x,\Psi_x)$ is a (unitary) intertwiner of rigged bimodules.
\end{definition}

Later, we will exhibit the spinor bundle on loop space as a rigged module bundle for the Clifford bundle, see \cref{sec:fusion}. 
Next we come to the definition of rigged von Neumann algebra bundles and rigged von Neumann bimodule bundles. These definitions are a novelty introduced in this article; we are not aware of any other treatment of smooth bundles of von Neumann algebras, or von Neumann bimodules.

\begin{definition}
\label{def:rvnab}
Let $N = (A,E)$ be a rigged von Neumann algebra.
 A \emph{rigged von Neumann algebra bundle} $\mathcal{N}=(\mc{A},\mc{E})$ with typical fibre $N$  over  $\mathcal{M}$ is a rigged \cstar-algebra bundle $\mc{A}$ with typical fibre $A$ and a rigged $\mathcal{A}$-module bundle $\mathcal{E}$ with typical fibre $E$.  \end{definition}

Thus, the difference between a rigged $\mathcal{A}$-module bundle $\mathcal{E}$ and a rigged von Neumann algebra bundle $\mathcal{N}=(\mathcal{A},\mathcal{E})$ only lies in fact that the \emph{typical fibre} is  required to be a rigged von Neumann algebra.
\begin{comment}
Any rigged von Neumann algebra bundle is locally trivial in the sense that there exist compatible local trivializations $\Phi$ of $\mathcal{A}$ and $\Psi$ of $\mathcal{E}$, see \cref{def:repcstarbundle}. 
\end{comment}
Via compatible local trivializations (\cref{def:repcstarbundle}), this property extends to all fibres; thus, every fibre $\mathcal{N}_x=(\mathcal{A}_x,\mathcal{E}_x)$ is a rigged von Neumann algebra.  
\begin{comment}
To see this, choose compatible local trivializations $\Phi$ of $\mathcal{A}$ and $\Psi$ of $\mathcal{E}$ around $x$. These identify the \cstar-algebras $\mathcal{A}_x^{\|\cdot\|}$ and $A^{\|\cdot\|}$, the Hilbert spaces $\mathcal{E}_x^{g}$ and $E^{g}$, and intertwine the representations. This shows that the faithfulness of $\rho_0$ implies the one of $\rho_x$.
\end{comment}
Moreover, \cref{lem:SpatialIsomorphisms} guarantees that any choice of compatible local trivializations induces fibrewise spatial isomorphisms of rigged von Neumann algebras from $\mathcal{N}_x$ to $N$.

\begin{definition}
\label{def:morphismriggedvonneumann}
 A \emph{morphism} $\Phi: \mathcal{N}_1\to \mathcal{N}_2$ between rigged von Neumann algebra bundles $\mathcal{N}_1=(\mathcal{A}_1,\mathcal{E}_1)$ and $\mathcal{N}_2=(\mathcal{A}_2,\mathcal{E}_2)$ is a morphism $\Phi: \mathcal{A}_1\to \mathcal{A}_2$ of rigged \cstar-algebra bundles that is fibrewise a morphism of rigged von Neumann algebras. 
\end{definition}

\begin{remark}
\label{re:spatialmorphismbundles}
Again, it is useful to introduce the notion of a \emph{spatial morphism} from $\mathcal{N}_1=(\mathcal{A}_1,\mathcal{E}_1)$ to $\mathcal{N}_2=(\mathcal{A}_2,\mathcal{E}_2)$ as  a unitary intertwiner $(\Phi,\Psi)$ of rigged module bundles (\cref{def:intertwinerriggedmodulebundles}). In this situation, we have  in the fibre over each point $x\in \mathcal{M}$ a unitary intertwiner $(\Phi_x,\Psi_x)$ from $(\mathcal{E}_1)_x$ to $(\mathcal{E}_2)_x$, i.e., a spatial morphism from $(\mathcal{N}_1)_x$ to $(\mathcal{N}_2)_x$. By  \cref{lem:SpatialIsomorphisms},   $\Phi_x$ is a morphism of rigged von Neumann algebras, and thus, by definition, $\Phi$ is a morphism of rigged von Neumann algebra bundles. This shows that spatial morphisms of rigged von Neumann algebra bundles are, in particular, morphisms in the sense of \cref{def:morphismriggedvonneumann}. 
\end{remark}

\begin{remark}
\label{re:loctrivmorph}
Let $\mathcal{N}=(\mathcal{A},\mathcal{E})$ be a rigged von Neumann algebra bundle with typical fibre $N$. 
Then, there exist compatible local trivializations $(\Phi,\Psi)$ over open subsets $U \subset \mathcal{M}$ as  in \cref{def:repcstarbundle}, which are invertible unitary intertwiners between the rigged module bundles $\mathcal{N}|_U$ and $N \times U$, and thus, spatial isomorphisms of rigged von Neumann algebra bundles.  In particular, our rigged von Neumann algebra bundles are locally trivial in this strong sense of spatial isomorphisms.   
\end{remark}

The  rigged von Neumann algebra bundles that appear in this paper  are certain  Clifford algebra bundles over loop spaces, and they appear in \cref{sec:spinstructuresandspinorbundles,sec:algebrabundlepathspace}.  For now, it remains to define rigged von Neumann bimodule bundles.
 
\begin{definition}\label{def:vonNeumannBimoduleBundle}
Let $N_{1} = (A_{1},E_{1})$ and $N_{2} = (A_{2},E_{2})$ be rigged von Neumann algebras, and let $F$ be a rigged von Neumann $N_{1}$-$N_{2}$-bimodule.
 Further, let $\mc{N}_{1} = (\mc{A}_{1},\mc{E}_{1})$ and $\mc{N}_{2} = (\mc{A}_{2},\mc{E}_{2})$ be rigged von Neumann algebra bundles over $\mathcal{M}$ with typical fibres $N_1$ and $N_2$, respectively. 
 A \emph{rigged von Neumann $\mc{N}_{1}$-$\mc{N}_{2}$-bimodule bundle} with typical fibre $F$ is a rigged  $\mathcal{A}_1$-$\mathcal{A}_2$-bimodule bundle $\mathcal{F}$ with typical fibre $F$, such that around each point $x\in \mathcal{M}$ there exist compatible local trivializations $(\Phi^1,\Psi^1)$ of $\mathcal{N}_1$ and $(\Phi^2,\Psi^2)$ of $\mathcal{N}_2$, and a local trivialization $\Psi$ of $\mathcal{F}$ such that $(\Phi^1,\Phi^2,\Psi)$ is compatible in the sense of \cref{def:cstarbimodulebundle}. 

\end{definition}

The following lemma is to ensure that \cref{def:vonNeumannBimoduleBundle} gives the correct structure in each fibre.

 \begin{lemma}
 \label{re:fibresofvonneumannbimodule}
 If $\mathcal{F}$ is a rigged von Neumann $\mathcal{N}_1$-$\mathcal{N}_2$-bimodule bundle with typical fibre $F$, then the fibre $\mathcal{F}_x$ is a rigged von Neumann $(\mathcal{N}_1)_x$-$(\mathcal{N}_2)_x$-bimodule, for each $x\in \mathcal{M}$. 
\end{lemma}

\begin{proof}
 Because $\mc{F}$ is a rigged $\mc{A}_{1}$-$\mc{A}_{2}$-bimodule bundle, we have that $\mc{F}_{x}$ is a rigged $(\mc{A}_{1})_{x}$-$(\mc{A}_{2})_{x}$-bimodule.
 Let $(\rho_{1})_{x}^{\vee}: (\mc{A}_{1})^{\| \cdot \|}_{x} \rightarrow \mc{B}(\mc{F}^{\langle \cdot, \cdot \rangle}_{x})$ be the corresponding $*$-homomorphism.
 The statement that $\mc{F}_{x}$ is a rigged $(\mc{N}_{1})_{x}$-module is then equivalent to the statement that $\rho_1^{\vee}: (\mathcal{A}_1)_x^{\|\cdot\|} \to \mathcal{B}(\mathcal{F}_x^{\left \langle  \cdot,\cdot \right \rangle})$ extends to $(\mathcal{N}_1)_x''$.
 But, since the typical fibre $F$ is a rigged von Neumann bimodule, its corresponding maps $A_1^{\|\cdot\|} \to \mathcal{B}(F^{\left \langle \cdot,\cdot  \right \rangle})$ extend to $N_1''$.
Next, consider local trivializations $(\Phi^1,\Psi^1)$ and $\Psi$ around $x$ as in \cref{def:vonNeumannBimoduleBundle}. Then, we have the extensions $\Phi^1_x:(\mathcal{N}_1)_x'' \to N_1''$, and the conjugation by the unitary operator $\Psi_x$, which yields a normal $\ast$-homomorphism $\mathcal{B}(F^{\left \langle \cdot,\cdot  \right \rangle}) \to \mathcal{B}(\mathcal{F}_x^{\left \langle \cdot,\cdot  \right \rangle})$.
 Using the compatibility of $\Phi^{1}$ with $\Psi$ it then follows that the map
\begin{equation*}
(\mathcal{N}_1)_x'' \to N_1'' \to \mathcal{B}(F^{\left \langle \cdot,\cdot  \right \rangle}) \to \mathcal{B}(\mathcal{F}_x^{\left \langle \cdot,\cdot  \right \rangle})
\end{equation*}  
is an extension of $\rho_1^{\vee}$, which is moreover the composition of normal $*$-homomorphisms, and thus a normal $*$-homomorphism itself.
A similar argument for the $*$-homomorphism $\rho_{2}^{\vee}: (\mc{A}_{2}^{\mathrm{opp}})_{x}^{\| \cdot \|} \rightarrow \mc{B}(\mc{F}_{x}^{\langle \cdot, \cdot \rangle })$ proves that $\mc{F}_{x}$ is a rigged von Neumann $(\mc{N}_{2})_{x}^{\mathrm{opp}}$-module.
\end{proof}

In the next section, we will  glue rigged von Neumann bimodules bundles; therefore, we also need to introduce (spatial) intertwiners between them.
\begin{definition}
\label{def:morphvNbimodbun}
A \emph{(unitary)} \emph{intertwiner} from a rigged von Neumann $\mathcal{N}_1$-$\mathcal{N}_2$-bimodule bundle $\mathcal{F}$ to a rigged von Neumann $\widetilde{\mathcal{N}}_1$-$\widetilde{\mathcal{N}}_2$-bimodule bundle $\widetilde{\mathcal{F}}$ is a (unitary) intertwiner $(\Phi^1,\Phi^2,\Psi)$ between the underlying rigged bimodule bundles in the sense of \cref{def:cstarbimodintertwiner}, such that $\Phi^1$ and $\Phi^2$ are morphisms of rigged von Neumann algebra bundles. A \emph{spatial intertwiner} from $\mathcal{F}$ to $\widetilde{\mathcal{F}}$ is a quintuple $(\Phi^1,\Psi^1,\Phi^2,\Psi^2,\Psi)$ in which $(\Phi^1,\Phi^2,\Psi)$ is a unitary intertwiner of rigged von Neumann bimodules bundles, and  $(\Phi^1,\Psi^1):\mathcal{N}_1 \to \widetilde{\mathcal{N}}_1$ and $(\Phi^2,\Psi^2):\mathcal{N}_2\to\widetilde{\mathcal{N}}_2$ are spatial morphisms.
\end{definition}

The definitions guarantees full compatibility with the fibrewise notions: any (spatial)  intertwiner induces in the fibre over each point $x\in \mathcal{M}$ a (spatial) intertwiner  between rigged von Neumann bimodules
in the sense of \cref{def:riggedvonNeumannbimodule}. Moreover, if $\mathcal{F}$ is a rigged von Neumann $\mathcal{N}_1$-$\mathcal{N}_2$-bimodule bundle with typical fibre $F$, any choice of local trivializations $(\Phi^1,\Psi^1)$, $(\Phi^2,\Psi^2)$, and $\Psi$ over an open subset $U \subset \mathcal{M}$ as in \cref{def:vonNeumannBimoduleBundle}, assemble into an invertible spatial intertwiner $(\Phi^1,\Psi^1,\Phi^2,\Psi^2,\Psi)$ from $\mathcal{F}|_U$ to $F \times U$. In particular, rigged von Neumann bimodule bundles are locally trivial in this strong sense of invertible spatial intertwiners.

We would like to make clear and summarize the following important feature of our definitions. Rigged von Neumann algebra bundles and bimodule bundles over $\mathcal{M}$ have in the fibre over each point $x\in \mathcal{M}$ rigged von Neumann algebras and bimodules, as introduced in \cref{sec:HilbertBundle:1} (see \cref{re:fibresofvonneumannbimodule}). Likewise, (spatial) morphisms between rigged von Neumann algebra bundles, and (spatial) intertwiners between rigged von Neumann bimodules restrict in each fibre to (spatial) morphisms of rigged von Neumann algebras and (spatial) intertwiners of rigged von Neumann bimodules, as introduced in \cref{sec:HilbertBundle:1}. In turn, we have shown in \cref{re:vnacl,rem:vonNeumanBimoduleCompletion} that these fibrewise structures induce, respectively, ordinary von Neumann algebras, ordinary (spatial) morphisms,  ordinary bimodules between von Neumann algebras, and ordinary bounded intertwiners between these, in the classical sense.
 Thus, one can pass, at any time and in any fibre, to this classical setting.

\subsection{Rigged bundles over diffeological spaces}
\label{sec:HilbertBundle:3}

For the discussion of fusion on loop space, we need to treat bundles over \emph{diffeological spaces}. We recall briefly that a diffeology on a set $X$ consists of a set of maps $c:U \to X$ called \quot{plots}, where  $U \subset \R^{k}$ is open and $k\in \mathbb{N}_0$ can be arbitrary, subject to a number of axioms, see \cite{iglesias1}  for details.
A map $f:X \to Y$ between diffeological spaces is called \emph{smooth}, if its composition with any plot of $X$ results in a plot of $Y$. We let $\Diff$ denote the category of diffeological spaces. Any smooth manifold $M$ or Fr\'echet manifold  $M$ becomes a diffeological space by saying that every smooth map $c: U \to M$, for every open subset $U \subset \R^{k}$ and any $k$, is a plot. 
Maps between smooth manifolds or Fr\'echet manifolds are then smooth in the classical sense if and only if they are smooth in the diffeological sense. In other words, the category $\Frech$ of  Fr\'echet manifolds fully faithfully embeds into $\Diff$.

Next we describe a general procedure to extend a presheaf $\mathcal{F}$ of categories on $\Frech$ to one on $\Diff$, and then apply this to various presheaves of bundles defined in the previous section.   We briefly recall that a \emph{presheaf $\mathcal{F}$ of categories on a category $\mathcal{C}$ }is a weak functor $\mathcal{F}:\mathcal{C}^{\opp} \to \Cat$ to the bicategory of categories, functors, and natural transformations. That is, $\mathcal{F}$ assigns to each object $X$ of $\mathcal{C}$ a category $\mathcal{F}(X)$, to each morphism $f:X \to Y$ in $\mathcal{C}$ a functor $f^{*}: \mathcal{F}(Y) \to \mathcal{F}(X)$, and to each pair $(f,g)$ of composable morphisms $f:X \to Y$ and $g:Y \to Z$ a natural equivalence $\eta_{f,g}: (g \circ f)^{*} \Rightarrow f^{*} \circ g^{*}$, such that the natural equivalences respect the associativity of the composition. 

\begin{remark}
\label{re:examplesofpresheaves}
In \cref{sec:locallytrivialbundles} we have encountered the following presheaves of categories over $\Frech$:
\begin{itemize}

\item
For a Fr\'echet Lie group $\mathcal{G}$, the presheaf  of principal $\mathcal{G}$-bundles and bundle morphisms.

\item 
For a  rigged Hilbert space $E$, the presheaf  of rigged Hilbert space bundles with typical fibre $E$, and with isometric morphisms (\cref{def:HilbertBundle}). 

\item
For a  rigged \cstar-algebra $A$, the presheaf  of rigged \cstar-algebra bundles with typical fibre $A$, and  with isometric morphisms (\cref{def:cstarBundle}). 

\item 
For rigged \cstar-algebras $A_1$ and $A_2$, and a rigged $A_1$-$A_2$-bimodule $E$, the presheaf  of triples $(\mathcal{A}_1,\mathcal{A}_2,\mathcal{E})$ of  rigged \cstar-algebra bundles $\mathcal{A}_1$ and $\mathcal{A}_2$ with typical fibres $A_1$ and $A_2$, respectively, and a rigged $\mathcal{A}_1$-$\mathcal{A}_2$-bimodule bundle $\mathcal{E}$ with typical fibre $E$ as defined in \cref{def:cstarbimodulebundle}, with unitary intertwiners as defined in \cref{def:cstarbimodintertwiner} as morphisms. 

\item
For a  rigged von Neumann algebra $N$, the  presheaf  of rigged von Neumann algebra bundles with typical fibre $N$, and with spatial morphisms  (\cref{def:rvnab,re:spatialmorphismbundles}).

\item 
For rigged von Neumann algebras $N_1$ and $N_2$, and a rigged von Neumann $N_1$-$N_2$-bimodule $F$, the presheaf  of triples $(\mathcal{N}_1,\mathcal{N}_2,\mathcal{E})$ of  rigged von Neumann algebra bundles $\mathcal{N}_1$ and $\mathcal{N}_2$ with typical fibres $N_1$ and $N_2$, respectively, and a rigged von Neumann $\mathcal{N}_1$-$\mathcal{N}_2$-bimodule bundle $\mathcal{F}$ with typical fibre $F$ as defined in \cref{def:vonNeumannBimoduleBundle}, with spatial intertwiners as defined in \cref{def:morphvNbimodbun} as morphisms. 
        
\end{itemize}
\end{remark}

Our aim is to assign to any presheaf $\mathcal{F}$  of categories over $\Frech$ a presheaf $\mathcal{F}^{\Diff}$ of categories over $\Diff$.
There is an abstract canonical procedure how to do this, which we explain later in \cref{re:presheaves}. In the following definition we spell out directly the result of this procedure, emphasising the
concrete perspective.

\begin{definition}
\label{def:F(X)}
Let $\mathcal{F}$ be a presheaf of categories over $\Frech$, and let $X$ be a diffeological space. We define a category $\mathcal{F}^{\Diff}(X)$ in the following way: 
\begin{enumerate}[(a)]

\item 
An object of the category $\mathcal{F}^{\Diff}(X)$ is a pair $\mathcal{E}=((\mathcal{E}_c),(\phi_{c_1,c_2,f}))$ consisting of
\begin{itemize}

\item 
a   family $(\mathcal{E}_c)$, indexed by the  plots $c:U \to X$ of $X$, with 
 $\mathcal{E}_c$ an object of $\mathcal{F}(U)$.
 
 \item
 a family $(\phi_{c_1,c_2,f})$, indexed by triples $(c_1,c_2,f)$ consisting of two plots $c_1:U_1 \to X$ and $c_2:U_2 \to X$ and of a smooth map $f: U_1 \to U_2$ such that $c_2 \circ f=c_1$, with $\phi_{c_1,c_2,f}$ a morphisms from $\mathcal{E}_{c_1}$ to $f^{*}\mathcal{E}_{c_2}$ in the category $\mathcal{F}(U_1)$. 
\end{itemize}

This structure is subject to the condition that whenever $(c_1,c_2,f_{12})$ and $(c_2,c_3,f_{23})$ are triples as above, the diagram
\begin{equation*}
\xymatrix@C=5em{\mathcal{E}_{c_1}\ar[r]^-{\phi_{f_{23} \circ f_{12}}} \ar[d]_{\phi_{f_{12}}} & (f_{23}\circ f_{12})^{*}\mathcal{E}_{c_3} \ar[d]^{\eta_{f_{12},f_{23}}} \\ f_{12}^{*}\mathcal{E}_{c_2} \ar[r]_-{f_{12}^{*}\phi_{f_{23}}} & f_{12}^{*}f_{23}^{*}\mathcal{E}_{c_3}}
\end{equation*}   
of morphisms in $\mathcal{F}(U_1)$ is commutative.

\item
A morphism from an object $\mathcal{E}=((\mathcal{E}_c),(\phi_{c_1,c_2,f}))$ to an object $\mathcal{E}'=((\mathcal{E}'_c),(\phi'_{c_1,c_2,f}))$ in  $\mathcal{F}^{\Diff}(X)$ is a family $\psi=(\psi_c)$, indexed by the plots $c:U \to X$ of $X$, with morphisms $\psi_c: \mathcal{E}_c \to \mathcal{E}_c'$ in $\mathcal{F}(U)$, such that the  diagram 
\begin{equation*}
\xymatrix@C=4em{\mathcal{E}_{c_1} \ar[r]^-{\psi_{c_1}} \ar[d]_{\phi_{f}} & \mathcal{E}'_{c_1}   \ar[d]^{\phi'_f} \\ f^{*}\mathcal{E}_{c_2} \ar[r]_{f^{*}\psi_{c_2}} & f^{*}\mathcal{E}_{c_1}}
\end{equation*}
is commutative for all triples $(c_1,c_2,f)$.
The composition of morphisms is plot-wise.

\end{enumerate}
\end{definition}

It is straightforward to complete the assignment $X \mapsto \mathcal{F}^{\Diff}(X)$ given in \cref{def:F(X)} to a presheaf of categories on $\Diff$. This is how we extend presheaves of categories from $\Frech$ to $\Diff$. 
If $\mathcal{M}$ is a Fr\'echet manifold, which we may consider as a diffeological space, then there is a faithful functor
\begin{equation}
\label{eq:Yembedding}
\mathcal{F}(\mathcal{M})\to \mathcal{F}^{\Diff}(\mathcal{M})\text{,}
\end{equation}
which takes an object $\mathcal{E}$ of $\mathcal{F}(\mathcal{M})$ to the pair $((\mathcal{E}_c),(\phi_{c_1,c_2,f}))$, where $\mathcal{E}_c \defeq  c^{*}\mathcal{E}$
is just the pullback in $\mathcal{F}$ (recall that any plot $c$ is a smooth map between Fr\'echet manifolds here) and  $\phi_{c_1,c_2,f} : c_1^{*}\mathcal{E} \to f^{*}c_2^{*}\mathcal{E}$ is  the  isomorphism $\eta_{f,c_2}$ provided by $\mathcal{F}$. A morphism $\psi:\mathcal{E} \to \mathcal{E}'$ is sent to the family with $\psi_c \defeq  c^{*}\psi$; it is obvious that this preserves composition and is faithful. 
 
\begin{remark}
\label{re:presheaves}
The extension of presheaves of categories from $\Frech$ to $\Diff$ that we defined in \cref{def:F(X)} fits  into a more conceptual perspective, which we want to outline in this remark.   To this end, we infer that any diffeological space $X$ can be viewed as a presheaf $\underline{X}$ (of sets) on the category $\Open$ of open subsets of cartesian spaces,  see \cite{baez6}.
\begin{comment}
More  precisely, the objects of the category $\Open$ are open subsets $U \subset\R^{k}$, for all $k \in \mathbb{N}_0$, and the morphisms $f:U_1 \to U_2$ are all smooth maps.
\end{comment}
Namely, the presheaf $\underline{X}$  assigns to an object $U$ in $\Open$ the set $\underline{X}(U)$ of all plots $c:U \to X$ with domain $U$.
\begin{comment}
It assigns to a morphism $f:U_1 \to U_2$ in $\Open$ the map $\underline{X}(f):\underline{X}(U_2) \to \underline{X}(U_1)$ that sends a plot $c:U_2 \to X$ to the plot $f \circ c: U_1 \to X$. 
\end{comment}
\begin{comment}
The axioms of diffeology imply that this is even a sheaf. 
\end{comment}
\begin{comment}
Moreover, a smooth map $g: X \to Y$ between diffeological spaces induces a morphism $\underline{g}:\underline{X} \to \underline{Y}$ of sheaves, simply defined by $\underline{g}(U)(c)\defeq  g\circ c$. 
\end{comment}
In the following we regard $\underline{X}$ even as a presheaf of \emph{categories} on $\Open$, and we do this simply by regarding the sets $\underline{X}(U)$ as categories with only identity morphisms. Now, let $\mathcal{F}$ be a presheaf of categories on $\Frech$, which we may restrict to $\Open \subset \Frech$.
Presheaves of categories on $\Open$ form a bicategory $\mathrm{PSh}(\Open)$, the bicategory of weak functors $\Open^{\opp} \to \Cat$. Now, we consider the functor represented by $\mathcal{F}|_{\Open}$ and restricted along $\Diff \to \mathrm{PSh}(\Open)$,  
\begin{equation*}
\mathrm{Hom}_{\mathrm{PSh}(\Open)}(\,\underline{\,\cdot\,}\,,\mathcal{F}|_{\Open}):\Diff^{\opp} \to \Cat\text{.}
\end{equation*}
Explicitly, it assigns to a diffeological space $X$ the category $\mathrm{Hom}_{\mathrm{PSh}(\Open)}(\underline{X},\mathcal{F}|_{\Open})$ of morphisms between the objects $\underline{X}$ and $\mathcal{F}|_{\Open}$ of  $\mathrm{PSh}(\Open)$. 
\begin{comment}
Likewise, if $g:X \to Y$ is a smooth map between diffeological spaces, then \begin{equation*}
\mathcal{F}^{\Diff}(g):\mathrm{Hom}_{\mathrm{PSh}(\Open)}(\underline{Y},\mathcal{F}) \to \mathrm{Hom}_{\mathrm{PSh}(\Open)}(\underline{X},\mathcal{F})
\end{equation*}
is pre-composition with the sheaf morphism $\underline{g}:\underline{X} \to \underline{Y}$. 
\end{comment}
It is  a straightforward exercise to show that 
\begin{equation*}
\mathrm{Hom}_{\mathrm{PSh}(\Open)}(\underline{X},\mathcal{F}|_{\Open})=\mathcal{F}^{\Diff}(X)\text{;}
\end{equation*}
this embeds our \cref{def:F(X)} into a proper topos-theoretical framework. The bicategorical version of the Yoneda lemma, see e.g. \cite{Johnson2021}, implies now that $\mathcal{F}^{\Diff}(U) \cong \mathcal{F}(U)$ for every object $U$ in $\Open$.  
\begin{comment}
The bicategorical version of the Yoneda lemma is the following. Let $\mathcal{C}$ be a category, $U\in \mathcal{C}$ an object, and $\mathcal{F}:\mathcal{C}^{\opp} \to Cat$ a presheaf of categories on $\mathcal{C}$. Then, there is an equivalence of categories:
\begin{equation*}
\Hom_{\mathrm{PSh}(\mathcal{C})}(\Hom_{\mathcal{C}}(-,U),\mathcal{F}) \cong \mathcal{F}(U)\text{.}
\end{equation*}
This equivalence sends $\alpha: \Hom_{\mathcal{C}}(-,U) \to \mathcal{F}$ to $\alpha(U)(id_U)$. 

In our case, $\mathcal{C}=\Open$, and $\Hom_{\mathcal{C}}(-,U)=C^{\infty}(-,U)=\underline{U}$; thus, $\mathcal{F}^{\Diff}(U) \cong \mathcal{F}(U)$ for every $U\in \Open$, with the equivalence realized by $(\mathcal{E}_c,\phi_{c_1,c_2,f}) \mapsto \mathcal{E}_{id_U}$.
\end{comment}
This exhibits $\mathcal{F}^{\Diff}$ as the (left) Kan extension of $\mathcal{F}|_{\Open}$ along $\Open^{\opp} \to \Diff^{\opp}$, which is precisely the standard way to extend presheaves from dense subsites to sites \cite{Johnstone2002}. This is our main justification for \cref{def:F(X)}.  
\begin{comment}
Indeed, consider the diagram
\begin{equation*}
\xymatrix{\Open^{\opp} \ar[rr]^{\mathcal{F}|_{\Open}} \ar[dr] && \Cat \\ & \Diff^{\opp} \ar[ur]_{\mathcal{F}^{\Diff}}\text{,}}
\end{equation*}
which is filled by the Yoneda lemma with a natural equivalence. 
\end{comment}
\begin{comment}
It also clarifies the relation between the categories $\mathcal{F}(\mathcal{M})$ and $\mathcal{F}^{\Diff}(\mathcal{M})$ established in \cref{eq:Yembedding} for a Fr\'echet manifold $\mathcal{M}$, though some steps are not really clear. It is known that left Kan extension of presheaves from dense subsites to sites establishes an equivalence between \emph{sheaves} \cite[Theorem C2.2.3).]{Johnstone2002}.
If it is true that $\Open\subset \Frech$ is a dense subsite, then we have equivalences $\mathrm{Sh}(\Frech)\cong \mathrm{Sh}(\Open)\cong \mathrm{Sh}(\Diff)$ under which $\mathcal{F}$ goes to $\mathcal{F}^{\Diff}$. This implies that $\mathcal{F}(\mathcal{M})\cong\mathcal{F}^{\Diff}(\mathcal{M})$ for all Fr\'echet manifolds $\mathcal{M}$. 
\end{comment}
\end{remark}

We apply \cref{def:F(X)} to the presheaves listed in \cref{re:examplesofpresheaves}, obtaining neat definitions of  rigged Hilbert space bundles, rigged \cstar-algebra bundles, rigged bimodule bundles, rigged von Neumann algebra bundles, and rigged von Neumann bimodule bundles over diffeological spaces. 
Almost all examples of such bundles that appear in this article are obtained via the functor \cref{eq:Yembedding}, in the following way: consider a diffeological space $X$, a Fr\'echet manifold $\mathcal{M}$, and a smooth map $f: X \to \mathcal{M}$. Then, the composite
\begin{equation*}
\xymatrix{\mathcal{F}(\mathcal{M}) \ar[r]^-{\cref{eq:Yembedding}} &  \mathcal{F}^{\Diff}(\mathcal{M}) \ar[r]^-{f^{*}} & \mathcal{F}^{\Diff}(X)}\text{,}
\end{equation*}
produces objects in $\mathcal{F}^{\Diff}(X)$ from objects in $\mathcal{F}(\mathcal{M})$. 
\begin{remark}
It might  be worth to point out that bundles over a diffeological space $X$ have  fibres over points $x\in X$, just like ordinary bundles. Indeed, the axioms of diffeology imply that $c_x: \R^0 \to X: 0 \mapsto x$ is always a plot. Thus, if $\mathcal{N}$ is, say, a rigged von Neumann algebra bundle over $X$, then  $\mathcal{N}_x \defeq  \mathcal{N}_{c_x}$ is a rigged von Neumann algebra bundle over the point $\R^{0}$, i.e., a rigged von Neumann algebra, the fibre of $\mathcal{N}$ at $x$. 
\begin{comment}
If $c:U \to X$ is some other plot, and there is a point  $u\in U$ with $c(u)=x$, then $f:\R^0 \to U:0\mapsto u$ is a smooth map with $c \circ f=c_x$, thus coming with a \emph{canonical}  isomorphism $\phi_f: \mathcal{N}_x \to (\mathcal{N}_c)_u$.
\end{comment}
If $\mathcal{N}=f^{*}\mathcal{N}'$ is obtained by pullback of a rigged von Neumann algebra bundle $\mathcal{N}'$ over $\mathcal{M}$ along a smooth map $f:X \to \mathcal{M}$, then $\mathcal{N}_x = \mathcal{N}'_{f(x)}$ is just the ordinary fibre of $\mathcal{N}'$ over $f(x)\in \mathcal{M}$.  
\end{remark}

\section{The free fermions on the circle}\label{sec:FreeFermions}

\label{sec:freefermions}

In this section we consider a bimodule for certain von Neumann algebras, known as the free fermions on the circle. It plays the role of the typical fibre of the spinor bundle of the loop space.
In the first three subsections we recall and then extend our earlier work \cite{Kristel2019,Kristel2020} about the free fermions. In particular,  we upgrade Clifford algebras and Fock spaces into an appropriate rigged von Neumann-theoretical setting (\cref{prop:FockRiggedBimodule}), and throughout explore  the effects of  splitting circles into two halves.  In \cref{sec:FusionImpThroughFock} we define a novel Connes fusion     product of certain implementers on Fock space; its relation to loop fusion proved later in \cref{thm:FusionProductsAgree} is a cornerstone
of our construction of a fusion product on the spinor bundle.      
     
\subsection{Lagrangians and Fock spaces}
\label{sec:CliffFockAndSpin}
\label{sec:CliffFockAndImp}

We recall some aspects of infinite-dimensional Clifford algebras and their representations on Fock spaces, for which the book \cite{PR95} is an excellent reference.
The results we require on implementers are more spread out across the literature, see for instance \cite{Ara85, Ott95, Ne09, Kristel2019}.
In  \cref{sec:CliffFockAndSpin,sec:SmoothRepresentations} we consider in some generality a complex Hilbert space $V$ equipped with a real structure $\alpha$, i.e., an anti-unitary involution $\alpha : V \rightarrow V$.
From \cref{sec:ReflectionFreeFermions} on, we restrict to a specific example, described at the beginning of \cref{sec:ReflectionFreeFermions}. 

Given a unital \cstar-algebra $A$, we say that a map $f:V \rightarrow A$ is a \emph{Clifford map} if the following equations are satisfied, for all $v,w \in V$:
\begin{align}\label{eq:CliffordMap}
 f(v)f(w) + f(w)f(v) &= 2 \langle v, \alpha(w) \rangle \mathds{1}, \quad f(v)^{*} = f(\alpha(v)).
\end{align}
The Clifford \cstar-algebra $\Cl(V)$ is the unique (up to unique isomorphism) unital \cstar-algebra equipped with a Clifford map $\iota: V \rightarrow \Cl(V)$, such that for each unital \cstar-algebra $A$ and each Clifford map $f: V \rightarrow A$, there exists a unique unital \cstar-algebra homomorphism $\Cl(f): \Cl(V) \rightarrow A$ with the property that $\Cl(f) \circ \iota = f$.
An explicit construction of $\Cl(V)$ is given in \cite[Section 1.2]{PR95}.

We define the orthogonal group of $V$, denoted $\O(V)$, to consist of those unitary transformations of $V$ that commute with the real structure.
If $g \in \O(V)$, then $\iota g: V \rightarrow \Cl(V)$ is a Clifford map.
We write $\theta_{g} \defeq \Cl( \iota g): \Cl(V) \rightarrow \Cl(V)$ for its extension to $\Cl(V)$.
The map $\theta_{g}$ is called the \emph{Bogoliubov automorphism} associated to $g$.
The map $\theta: g \mapsto \theta_{g}$ is  a continuous homomorphism from $\O(V)$ to $\Aut(\Cl(V))$, where $\O(V)$ is equipped with the operator norm topology, and $\Aut(\Cl(V))$ is equipped with the strong operator topology, see, e.g. \cite[Proposition 4.35]{Ambler2012}.

A \emph{Lagrangian} in $V$ is a subspace $L \subset V$ such that $V$ splits as the orthogonal direct sum $V = L \oplus \alpha(L)$.
If $L$ is a Lagrangian, then the \emph{Fock space} $\mc{F}_{L}$ is the Hilbert completion of the exterior algebra, $\Lambda L$, of $L$.
We identify $\alpha(L)$ with the dual of $L$ by identifying $w \in \alpha(L)$ with the linear map $L \ni v \mapsto \langle v, \alpha(w) \rangle$.
If $v \in L$ and $w \in \alpha(L) \simeq L^{*}$, then we write $c(v): \mc{F}_{L} \rightarrow \mc{F}_{L}$ for left multiplication with $v$, and $a(w) : \mc{F}_{L} \rightarrow \mc{F}_{L}$ for contraction with $w$.
The maps $c(v)$ and $a(v)$ are bounded operators on $\mc{F}_{L}$, and the map 
\begin{align*}
 \rho_{L}:V = L \oplus \alpha(L) &\rightarrow \mc{B}(\mc{F}_{L}), \quad (v,w) \mapsto \sqrt{2}(c(v) + a(w))
\end{align*}
is a Clifford map. 
This means that $\Cl(\rho_{L}): \Cl(V) \rightarrow \mc{B}(\mc{F}_{L})$ is a unital \cstar-algebra homomorphism; i.e., a representation of $\Cl(V)$ on $\mc{F}_{L}$.
This representation is irreducible \cite[Theorem 2.4.2]{PR95} and faithful; hence, we may identify $\Cl(V)$ with its image in $\mc{B}(\mc{F}_{L})$.
Whenever convenient, we adopt the notation $a \lact v = \Cl(\rho_{L})(a)(v)$ for $a \in \Cl(V)$ and $v \in \mc{F}_{L}$.

If $g \in \O(V)$, then we say that $g$ is \emph{implementable}, if there exists  $U \in \U(\mc{F}_{L})$ with the property that
\begin{equation}\label{eq:Implementer}
 \theta_{g}(a) = UaU^{*}
\end{equation}
for all $a \in \Cl(V)$;
the operator $U$ is said to \emph{implement} $g$. The problem to decide which $g\in \O(V)$ are implementable is called the \quot{implementability problem}.
 It is completely solved: an element $g \in \O(V)$ is implementable if and only if the operator $P_{L} g P_{L}^{\perp}$ is Hilbert-Schmidt, where $P_L$ is the orthogonal projection to $L$, see \cite[Theorem 3.3.5]{PR95} or \cite[Theorem 6.3]{Ara85}.
We write $\O_{\res}(V)$ for the set consisting of those $g \in \O(V)$ which are implementable; the set $\O_{\res}(V)$ is in fact a subgroup of $\O(V)$.

The group $\O(V)$ can be equipped with the structure of Banach Lie group in the standard way, with underlying topology the operator norm topology. The Lie algebra of $\O(V)$ is
\begin{equation*}
 \lie{o}(V) = \{ X \in \mc{B}(V) \mid [X,\alpha] = 0, X^{*} = -X \}.
\end{equation*}
The subgroup $\O_L(V)$ can also be equipped with the structure of a Banach Lie group,  whose underlying topology is given by the norm $\| g\|_{\mc{J}} = \| g\| + \| P_{L} g P_{L}^{\perp} \|_{2}$, where $\|g\|$ is the operator norm of $g$, and where $\| \cdot \|_{2}$ is the Hilbert-Schmidt norm, \cite[Section 3.4]{Kristel2019}. The inclusion $\O_L(V) \to \O(V)$ is smooth.
The Lie algebra of $\O_{\res}(V)$ is 
\begin{equation*}
 \lie{o}_{\res}(V) = \{ X \in \lie{o}(V) \mid \| P_{L} X P_{L}^{\perp} \|_{2} < \infty \}.
\end{equation*}

The \emph{group of implementers}, $\Imp_{\res}(V)$, is defined to be the subgroup of $\U(\mc{F}_{L})$ consisting of those operators $U \in \U(\mc{F}_{L})$, for which there exists a $g \in \O_{\res}(V)$ such that  \cref{eq:Implementer} holds.
If $U \in \Imp_{\res}(V)$, then the element $g \in \O_{\res}(V)$ that it implements is determined uniquely, and we obtain a group homomorphism $q:\Imp_{\res}(V) \rightarrow \O_{\res}(V)$.
Using the irreducibility of the representation of $\Cl(V)$ on $\mc{F}_{L}$ together with Schur's Lemma, we see that, for each $g \in \O_{\res}(V)$, the fibre $q^{-1}\{g\}$ is a $\U(1)$-torsor.
We equip the group $\Imp_{\res}(V)$ with the structure of Banach Lie group as in \cite[Theorem 3.15]{Kristel2019}, see also \cite{Wurzbacher2001}.
We then have that the exact sequence
\begin{equation*}
 \U(1) \rightarrow \Imp_{\res}(V) \xrightarrow{q} \O_{\res}(V) ,
\end{equation*}
is a central extension of Banach Lie groups.
It is important to note that the topology underlying the Banach Lie group structure on $\Imp_{\res}(V)$ is not the operator norm topology, and that the inclusion map $\Imp_{\res}(V) \rightarrow \U(\mc{F}_{L})$ is not continuous, let alone smooth.

Let us assume that we are given a further orthogonal decomposition $V=V_{+} \oplus V_{-}$ of $V$ into two Hilbert spaces $V_{\pm}$ which are preserved under $\alpha$.
In the sequel, such a splitting will implement the splitting of a loop into two paths.
Given such a splitting, we are interested in the operators on $V$ which preserve this splitting.
We write $P_{\pm}$ for the projection onto $V_{\pm}$ and define the following Banach spaces:
\begin{align*}
 \lie{o}^{\theta}(V) &\defeq \{ X \in \lie{o}(V) \mid P_{\pm}X P_{\mp} = 0 \}, &
 \lie{o}^{\theta}_{\res}(V) &\defeq \{ X \in \lie{o}_{\res}(V) \mid P_{\pm}X P_{\mp} = 0 \}.
\end{align*}
The spaces $\lie{o}^{\theta}(V)$ and $\lie{o}_{\res}^{\theta}(V)$ are  Lie subalgebras of $\lie{o}(V)$ and $\lie{o}_{\res}(V)$ respectively.
Next, we define the metric groups
\begin{align*}
 \O^{\theta}(V) &\defeq \{ g \in \O(V) \mid P_{\pm}gP_{\mp} = 0 \}, &
 \O^{\theta}_{\res}(V) &\defeq \{ g \in \O_{\res}(V) \mid P_{\pm}gP_{\mp} = 0 \}.
\end{align*}
Using standard techniques one may then show that the exponential maps $\lie{o}^{\theta}(V) \rightarrow \O^{\theta}(V)$ and $\lie{o}_{\res}^{\theta}(V) \rightarrow \O_{\res}^{\theta}(V)$ are local homeomorphisms, which allows us to equip $\O^{\theta}(V)$ and $\O_{\res}^{\theta}(V)$ with the structure of Banach Lie groups, with Lie algebras $\lie{o}^{\theta}(V)$ and $\lie{o}^{\theta}_{\res}(V)$ respectively.
It is then clear from the construction that $\O^{\theta}(V)$ and $\O_{\res}^{\theta}(V)$ are closed submanifolds of $\O(V)$ and $\O_{\res}(V)$ respectively.
If $g \in \O^{\theta}(V)$, then we set $g_{\pm} \defeq P_{\pm}gP_{\pm} \in \O(V_{\pm})$.
We observe that the maps $g \mapsto g_{\pm}$ are smooth group homomorphisms.
\begin{comment}
This is because the maps $\O^{\theta}(V) \to \O(V_{\pm})$ are continuous, and every continuous homomorphism of Banach Lie groups is smooth. 
\end{comment}
Finally, we define $\Imp_{\res}^{\theta}(V) \defeq \Imp_{\res}(V)|_{\O_{\res}^{\theta}(V)}$.

Next, we consider the Clifford algebras $\Cl(V_{\pm})$.
We observe that extending by zero gives isometries $V_{\pm} \rightarrow V$ which moreover intertwine the real structures, and thus induce isometric $*$-homomorphisms $\iota_{\pm}: \Cl(V_{\pm}) \rightarrow \Cl(V)$.
The algebra product 
\begin{equation*}
\Cl(V_{-}) \times \Cl(V_{+}) \subset \Cl(V) \times \Cl(V) \to \Cl(V)
\end{equation*}
induces a unital isomorphism $\Cl(V_{-}) \otimes \Cl(V_{+}) \cong \Cl(V)$ of \cstar-algebras. (Here, $\otimes$ stands for any choice of tensor product of \cstar-algebras. The choice is immaterial, because the Clifford \cstar-algebra is uniformly hyperfinite, and hence in particular nuclear.) Under these identifications, we have $\iota_{-}(a_{-}) = a_{-} \otimes \mathds{1}$ and $\iota_{+}(a_{+}) = \mathds{1} \otimes a_{+}$, for $a_{\pm} \in \Cl(V_{\pm})$. The following result expresses that the Bogoliubov automorphisms are compatible with the splitting; this will be used later in the proofs of  \cref{lem:bogminus,lem:HalfCliffmult}.

\begin{lemma}
\label{lem:Bogolsplit}
 For all $g_{-} \oplus g_{+} \in \O^{\theta}(V)$ and all $a_{\pm} \in \Cl(V_{\pm})$ we have
 \begin{equation*}
  \theta_{g_{-} \oplus g_{+}}(a_{-} \otimes a_{+}) = \theta_{g_{-}}(a_{-}) \otimes \theta_{g_{+}}(a_{+}),
 \end{equation*}
 where $\theta_{g_{\pm}}$ are the Bogoliubov automorphisms of the Clifford algebras $\Cl(V_{\pm})$.
\end{lemma}
\begin{proof}
 This follows from the fact that $(g_{-} \oplus g_{+}) \circ \iota_{\pm} = \iota_{\pm} \circ g_{\pm}$ for all $g_{-} \oplus g_{+} \in \O^{\theta}(V)$.
\end{proof}

\subsection{Smooth Fock spaces}
\label{sec:SmoothRepresentations}

In \cref{sec:fusion,sec:SpinorBundleOnLoopSpaceI} we construct bundles of  rigged Hilbert spaces and rigged \cstar-algebras over Fr\'echet manifolds. This requires a detailed study of smoothness properties of the  representations obtained in \cref{sec:CliffFockAndSpin}; this is the goal of this section. We continue working with a complex Hilbert space $V$ with a real structure $\alpha$, additionally equipped with an orthogonal decomposition into two Hilbert spaces $V_{\pm}$ which are preserved under $\alpha$.
Our first goal is to equip the Fock space $\mathcal{F}_L$ and the Clifford \cstar-algebra $\Cl(V)$ with the structure of a rigged Hilbert space and a rigged \cstar-algebra, respectively. We have done this already in our earlier paper \cite{Kristel2020}; however, there we have not taken the splitting $V=V_{-}\oplus V_{+}$ into account. For the purpose of this article, the splitting is essential, and it leads to a finer rigging (we compare them in \cref{rem:compareriggings} below). Therefore, we describe the important steps again.

As a  subgroup of $\U(\mc{F}_{L})$, the group $\Imp^{\theta}_{\res}(V)$ comes equipped with a unitary representation on $\mc{F}_{L}$.
As is typical for infinite dimensional representations, the action map 
$\Imp_{\res}^{\theta}(V) \times \mc{F}_{L} \rightarrow \mc{F}_{L}$
is not smooth.
The \emph{subspace of smooth vectors} in $\mc{F}_{L}$ is defined as usual to be
\begin{equation*}
 \mc{F}^{\smooth}_{L} \defeq \{ v \in \mc{F}_{L} \mid \Imp^{\theta}_{\res}(V) \rightarrow \mc{F}_{L}, U \mapsto Uv \text{ is smooth} \}.
\end{equation*}
The following result follows directly from \cite[Proposition 3.17]{Kristel2019}, see also \cite[Section 10.1]{Ne09}.
\begin{lemma}
 The set of smooth vectors $\mc{F}_{L}^{\smooth}$ contains the exterior algebra $\Lambda L$, and is hence a dense subspace of $\mc{F}_{L}$.
\end{lemma}

By definition of smooth vectors, the Lie algebra $\lie{imp}^{\theta}(V)$ of $\Imp_L^{\theta}(V)$ acts infinitesimally on $\mathcal{F}_L^{\smooth}$; i.e., for $X \in \lie{imp}^{\theta}(V)$ and  $v \in \mc{F}^{\smooth}_{L}$, we may define
\begin{equation*}
 Xv \defeq \der{}{t} \bigg|_{t=0} \exp(t X)(v).
\end{equation*}
Let $\mc{P}(\mc{F}_{L})$ be the set of all continuous semi-norms on $\mc{F}_{L}$.
We define a topology on $\mc{F}_{L}^{\smooth}$ by the following family of semi-norms, \cite[Section 4]{neebdiffvect}:
\begin{equation*}
 p_{n}(v) = \sup \{ p( X_{1} \dots X_{n} v) \mid X_{i} \in \lie{imp}^{\theta}(V), \|X_{i}\| \leqslant 1 \}, \quad p \in \mc{P}(\mc{F}_{L}), n \in \mathbb{N}_{0}.
\end{equation*}
The following result is proved completely analogously to \cite[Proposition 3.2.4]{Kristel2020}.
\begin{proposition}
\label{prop:smoothFockSpace}
The space of smooth vectors $\mathcal{F}^{s}_L$ is a rigged Hilbert space, and it carries a smooth representation of the Banach Lie group $\Imp_{\res}^{\theta}(V)$.
\end{proposition}

Just like the action map $\Imp^{\theta}_{\res}(V) \times \mc{F}_{L} \rightarrow \mc{F}_{L}$ is not smooth, the  map $\O^{\theta}(V) \times \Cl(V) \rightarrow \Cl(V)$ for the action by Bogoliubov automorphisms is not smooth either, an issue that we handle in a similar way.
We write $\Clsm$ for the subspace of smooth vectors in $\Cl(V)$, i.e.~
\begin{equation*}
 \Cl(V)^{\smooth} \defeq \{ a \in \Cl(V) \mid \O^{\theta}(V) \rightarrow \Cl(V), g \mapsto \theta_{g}(a) \text{ is smooth} \}.
\end{equation*}
\begin{lemma}
 The set of smooth vectors $\Clsm$ contains the algebraic Clifford algebra of $V$, and is hence dense in the \cstar-Clifford algebra $\Cl(V)$.
\end{lemma}

The Lie algebra $\lie{o}_{res}^{\theta}(V)$ acts on $\Clsm$, which allows us to proceed as follows.
Let $\mc{R}(\Cl(V))$ be the set of continuous semi-norms on $\Cl(V)$.
The topology on $\Clsm$ is then defined by the family of semi-norms
\begin{equation*}
 r_{n}(a) = \sup \{ r( Y_{1} \dots Y_{n} a) \mid Y_{i} \in \lie{o}^{\theta}(V), \|Y_{i}\| \leqslant 1 \}, \quad r \in \mc{R}(\Cl(V)), n \in \mathbb{N}_{0}.
\end{equation*}
In analogy with \cite[Proposition 3.2.7]{Kristel2020} we then have the following proposition.
\begin{proposition}
\label{prop:CliffordFrechetAlgebra}
The algebra of smooth vectors $\Clsm$ is a rigged \cstar-algebra, and it carries a smooth representation of the Banach Lie group $\mathrm{O}^{\theta}(V)$.
\end{proposition}   

Finally, we adapt  to this situation, and obtain the following result.

\begin{proposition}\label{prop:SmoothCliffordActionOnF}
The Fock space representation $\Cl(V) \times \mc{F}_{L} \rightarrow \mc{F}_{L}$ restricts to a map $\Cl(V)^{\smooth} \times \mathcal{F}_L^{\smooth} \to \mathcal{F}_L^{\smooth}$ and exhibits  $\mathcal{F}_L^{\smooth}$ as a rigged $\Cl(V)^{\smooth}$-module. Moreover, $\ClvN(V)^\smooth\defeq (\Cl(V)^{\smooth},\mathcal{F}_L^{\smooth})$ is a rigged von Neumann algebra.
\end{proposition}

\begin{proof}
Showing that the Fock space representation restricts and that the map $\rho:\Cl(V)^{\smooth} \times \mathcal{F}_L^{\smooth} \to \mathcal{F}_L^{\smooth}$ is continuous is completely analogous to the proof we gave in \cite[Proposition 3.2.8]{Kristel2020} for the slightly coarser riggings. By \cite[Remark 2.2.9]{Kristel2020} this implies that   $\mathcal{F}_L^{\smooth}$ as a rigged $\Cl(V)^{\smooth}$-module. In order to show that $\ClvN(V)^\smooth$ is a rigged von Neumann algebra, we only have to notice that the \cstar-representation induced by $\rho$ is the Fock space representation, which is faithful.  
\end{proof}

We remark that the rigged von Neumann algebra $\ClvN(V)^\smooth$ determines an ordinary von Neumann algebra $\ClvN(V)\defeq (\ClvN(V)^\smooth)''$ as the completion of $\Cl(V)$ acting on $\mathcal{F}_L$ (\cref{re:vnacl}). It is well-known that $\ClvN(V)=\mathcal{B}(\mathcal{F}_L)$.

\begin{remark}
\label{rem:compareriggings}
In our earlier paper \cite[Section 3.2]{Kristel2020} we considered different Fr\'echet spaces as the riggings on $\mathcal{F}_L$ and $\Cl(V)$, obtained without assuming a splitting of $V$:
 \begin{align*}
  \mc{F}^{\infty}_{L} &\defeq \{ v \in \mc{F}_{L} \mid \Imp_{\res}(V) \rightarrow \mc{F}_{L}, U \mapsto Uv \text{ is smooth} \}\\
  \Cl(V)^{\infty} &\defeq \{ a \in \Cl(V) \mid \O(V) \rightarrow \Cl(V), g \mapsto \theta_{g}(a) \text{ is smooth} \}
 \end{align*}
Because of the fact that $\Imp_{\res}^{\theta}(V)$ is contained in $\Imp_{\res}(V)$ we have that $\mc{F}_L^{\infty} \subset \mc{F}_{\res}^{\smooth}$.
 This means that we improve the rigging by passing to the subgroup $\Imp_L^{\theta}(V)$.
 The analogous statement is true for $\Cl(V)^{\infty}$ and $\Clsm$.
 The riggings $\mathcal{F}_L^{\infty}$ and $\Cl(V)^{\infty}$ have been appropriate for the construction of the spinor bundle on loop space and its Clifford action. For the construction of the fusion product we need the finer riggings $\mathcal{F}_L^{\smooth}$ and $\Cl(V)^{\smooth}$; in particular, we need these in \cref{lem:SmoothHalfInclusion}, see \cref{rem:finerriggingessential}. \end{remark}

We now move beyond the analogy with \cite[Section 3.2]{Kristel2020} and lift the separate Clifford \cstar-algebras $\Cl(V_{\pm})$ to the setting of rigged \cstar-algebras, by defining
\begin{equation*}
 \Cl(V_{\pm})^{\smooth} \defeq \{ a \in \Cl(V_{\pm}) \mid \O(V_{\pm}) \rightarrow \Cl(V_{\pm}),\; g \mapsto \theta_{g}(a) \text{ is smooth} \}.
\end{equation*}
Just like $\Cl(V)^{\smooth}$ these are rigged \cstar-algebras.

\begin{lemma}\label{lem:SmoothHalfInclusion}
 The inclusion maps $\iota_{\pm}$ restrict to isometric morphisms $\iota_{\pm}:\Cl(V_{\pm})^{\smooth} \rightarrow \Cl(V)^{\smooth}$ of rigged \cstar-algebras.
\end{lemma}
\begin{proof}
 Let $a \in \Cl(V_{-})$ be arbitrary.
 We then see from \cref{lem:Bogolsplit} that the map
 \begin{equation}
 \label{eq:mapsmooth}
  \O^{\theta}(V) \rightarrow \Cl(V)\;, \;\;g \mapsto \theta_{g}(a \otimes \mathds{1})
 \end{equation}
is the composition of the  Lie group homomorphism  $\O^{\theta}(V) \to \O(V_{-}),\;g \mapsto g_{-}$ with the action map  $\O(V_{-}) \to \Cl(V_{-}),\; g_{-}\mapsto \theta_{g_{-}}(a)$, followed by the smooth map $\iota_{-}:\Cl(V_{-})\to \Cl(V)$.
This proves that if $a \in \Cl(V_{-})^{\smooth}$, then \cref{eq:mapsmooth} is smooth, and we have $a \otimes \mathds{1} \in \Cl(V)^{\smooth}$; and thus $\iota_{-}$ restricts to a map $\iota_{-}:\Cl(V_{-})^{\smooth} \rightarrow \Cl(V)^{\smooth}$.
 What remains to be shown is that this restriction $\iota_{-}: \Cl(V_{-})^{\smooth} \rightarrow \Cl(V)^{\smooth}$ is continuous w.r.t. the Fr\'echet space structures.
  Because $\iota_{-}$ is linear, it suffices to show continuity at $0$.
 First, observe that if $r \in \mc{R}(\Cl(V))$, then $r \circ \iota_{-} \in \mc{R}(\Cl(V_{-}))$.
 Moreover, we claim that
 \begin{equation}
 \label{eq:diffOtheta}
  r_{n}(a \otimes \mathds{1}) = (r \circ \iota_{-})_{n}(a),
 \end{equation}
 indeed
 \begin{align*}
  r_{n}(a \otimes \mathds{1}) &= \sup \{ r(Y_{1} \dots Y_{n}(a\otimes \mathds{1})) \mid Y_{i} \in \lie{o}^{\theta}(V), \| Y_{i} \| \leqslant 1 \} \\
  &= \sup \{ r(Y_{1}|_{V_{-}} \dots Y_{n}|_{V_{-}}(a)\otimes \mathds{1}) \mid Y_{i} \in \lie{o}^{\theta}(V), \| Y_{i} \| \leqslant 1 \} \\
  &=\sup \{ r \circ \iota_{-}(Y_{-,1} \dots Y_{-,n}(a)) \mid Y_{-,i} \in \lie{o}(V_{-}), \| Y_{i} \| \leqslant 1 \} \\
  &= (r \circ \iota_{-})_{n}(a).
 \end{align*}
 From Equation \cref{eq:diffOtheta} it follows that the preimage of the subbasis open neighbourhood of zero given by $\{ x \in \Cl(V) \mid r_{n}(x) < \eps \} \subset \Cl(V)$ is the subbasis open neighbourhood of zero given by $\{ a \in \Cl(V_{-}) \mid (r \circ \iota_{-})_{n}(a) < \eps \}$.
The discussion of $\Cl(V_{+})$ is analogous.
\end{proof}

\begin{remark}
\label{rem:finerriggingessential}
 For \cref{lem:SmoothHalfInclusion} it is essential that we work with $\O^{\theta}(V)$ and not with $\O(V)$, because $\O(V)$ does not come equipped with a map to $\O(V_{-})$, and moreover, because Equation \cref{eq:diffOtheta} only holds because we use $\O^{\theta}(V)$. Then, in turn, the Fock space representation of $\Cl(V)$ on $\mathcal{F}_L$ does not restrict to a representation of $\Cl(V)^{\smooth}$ on $\mathcal{F}_L^{\infty}$; this is why we have to use the bigger rigging $\mathcal{F}_L^{\smooth}$ on $\mathcal{F}_L$ a well.
\end{remark}

We obtain the following result about the rigged \cstar-algebras $\Cl(V_{\pm})^{\smooth}$. 

\begin{proposition}\label{cor:SmoothHalfAction}
The isometric morphisms $\iota_{\pm}$ induce  on  $\mathcal{F}_L^{\smooth}$ the structure of a rigged $\Cl(V_{\pm})^{\smooth}$-module. Moreover,   $N_{\pm}(V) \defeq (\Cl(V_{\pm})^{\smooth},\mc{F}_{L}^{\smooth})$ are rigged von Neumann algebras, and the pairs $(\iota_{\pm},\id)$ are spatial morphisms $N_{\pm}(V) \to \ClvN(V)^\smooth$.
\end{proposition}

\begin{proof}
The first claim follows from \cref{lem:SmoothHalfInclusion} and  \cref{ex:OppositeRep}\cref{re:inducedcstarrep}. The second claim  (see \cref{def:riggedvonneumannbundle}) follows since the representation of $\Cl(V_{\pm})^{\smooth}$ on $\mc{F}_{L}^{\smooth}$ extends to the representation of $\Cl(V_{\pm})$ on $\mc{F}_{L}$, which is faithful.
\end{proof}

The rigged von Neumann algebra $N_{\pm}(V) = (\Cl(V_{\pm})^{\smooth},\mathcal{F}_{L}^{\smooth})$ determines a von Neumann algebra $N_{\pm}(V)''$ in the classical sense, as the von Neumann closure of the \cstar-algebra $\Cl(V_{\pm})$ acting on $\mathcal{F}_{L}$, see \cref{re:vnacl}. It is well known that these von Neumann algebras are III$_{1}$-factors, \cite[Example 4.3.2]{ST04}, \cite[Section 16]{wassermann98}.

\subsection{Free fermions as a rigged bimodule}\label{sec:ReflectionFreeFermions}

We now fix a concrete Hilbert space $V$, a real structure $\alpha$, a Lagrangian $L$, and a splitting $V=V_{-}\oplus V_{+}$, which will remain the same for the rest of the paper.
Let $\mathbb{S} \rightarrow S^{1}$ be the odd spinor bundle on the circle, i.e., the one associated  to the odd (i.e., the connected) spin structure on the circle.
We set $V \defeq L^{2}(S^{1}, \mathbb{S} \otimes \C^{d})$, where $d$ is a natural number (interpreted as the spacetime dimension).
Pointwise complex conjugation gives a real structure  $\alpha: V \rightarrow V$.
In \cite[Section 2]{Kristel2019} it is explained how the space of smooth 2$\pi$-antiperiodic functions on the real line can be identified with the dense subspace $\Gamma(S^{1},\mathbb{S} \otimes \C^{d}) \subset V$ of smooth sections.
Under this identification, $V$ has an orthonormal basis $\{ \xi_{n,j} \}_{n \in \mathbb{N}, j =1,...,d}$ by setting
\begin{equation*}
 \xi_{n,j}(t) = e^{-i \left(n+ \frac{1}{2} \right) t} e_{j},
\end{equation*}
where $\{ e_{j} \}_{j=1,...,d}$ is the standard basis of $\C^{d}$.
It is further shown that $\alpha(\xi_{n,j}) = \xi_{-n-1,j}$ for all $n$ and all $j$.
It follows that the closed linear span
\begin{equation*}
 L \defeq \operatorname{span} \{ \xi_{n,j} \mid n \geqslant 0, \, j=1,...,d \}
\end{equation*}
is a Lagrangian in $V$.
The corresponding Fock space $\mathcal{F}\defeq \mc{F}_L$ is called the \emph{free fermions}. Let us write $I_{+} \subset S^{1}$ for the open upper semicircle and $I_{-} \subset S^{1}$ for the open lower semicircle.
\begin{comment}
Explicitly,
\begin{align*}
I_{+} \defeq  \{ e^{i \ph} \in S^{1} \mid \ph \in (0, \pi) \}, \quad
I_{-} \defeq  \{ e^{i \ph} \in S^{1} \mid \ph \in (\pi, 2\pi) \}.
\end{align*}
\end{comment}
If $f \in V$, then we write $\operatorname{Supp}(f)$ for the support of $f$. We consider the subspaces
\begin{equation*}
        V_{\pm} \defeq  \{ f\in V \mid \operatorname{Supp}(f) \subseteq I_{\pm} \};
\end{equation*}
they yield a decomposition $V = V_{-} \oplus V_{+}$, and  $\alpha$ restricts to real structures on $V_{\pm}$.
This puts us in the setting of \cref{sec:SmoothRepresentations}, and we thus obtain rigged von Neumann algebras $N_{\pm} \defeq  N_{\pm}(V) \defeq  (\Cl(V_{\pm})^{\smooth},\mc{F}^{\smooth})$.

\begin{comment}
\begin{remark}
Unfortunately, we do not know how to describe Fr\'echet spaces $\mathcal{F}^{\smooth}$ and $\mathrm{Cl}(V_{\pm})^{\smooth}$ explicitly in terms of spinors on the circle.
\end{remark}
\end{comment}

We denote by $\tau$ the complex-linear extension of the map $\xi_{n,j}\mapsto \xi_{-n-1,j}$.
The map $\tau$ orthogonal, exchanges $L$ with $\alpha(L)$, and exchanges $V_{+}$ with $V_{-}$.
We remark that while $\tau\in\O(V)$, it is not implementable.
\begin{comment}
The map $\tau$ is an isometric isomorphism that preserves the real structure.
 Indeed, we have $\tau(\xi_{n,j})=-\xi_{-(n+1),j}$.
Thus, $\tau\in \O(V)$. The associated Bogoliubov automorphism
 $\theta_\tau: \Cl(V) \rightarrow \Cl(V)$
 exchanges the subalgebras $\Cl(V_{+})$ and $\Cl(V_{-})$.
Since $\tau$ exchanges $L$ with $\alpha(L)$, one can check that $[\tau,\mc{J}]=2\tau \mc{J}$, which is not Hilbert-Schmidt. Thus, $\tau \notin \O_{\res}(V)$ and hence $\theta_{\tau}$ is not implementable.
\end{comment}
Since $\alpha$ interchanges $L$ with $\alpha(L)$ as well, the anti-unitary isomorphism $\alpha \tau$ preserves both $L$ and $\alpha(L)$. In particular, it induces an anti-unitary operator $\Lambda_{\alpha \tau}: \mc{F} \rightarrow \mc{F}$.

If $g \in \O(V)$, then we write $\tau(g) \defeq  \tau \circ g \circ \tau \in \O(V)$.
With this notation we have, for all $f \in V$ and all $g \in \O(V)$, that $\tau(g(f)) = \tau(g) (\tau(f))$.
Because $\tau$ interchanges $L$ with $\alpha(L)$ we have that conjugation by $\tau$ yields an isometric (hence smooth) group homomorphism from $\O_{\res}(V)$ into $\O_{\res}(V)$.
\begin{comment}
Smoothness be seen as follows. Conjugation by $\tau$ is an isometry, and hence continuous. Then because it's a group homomorphism and $\O_{\res}(V)$ is a Lie group it follows that conjugation by $\tau$ is smooth.
\end{comment}
Finally, because $\tau$ exchanges $V_{+}$ with $V_{-}$ we have that  $\tau$ preserves $\O_{\res}^{\theta}(V)$, and moreover exchanges $\O(V_{-})$ with $\O(V_{+})$.

\begin{remark}
 In the following discussion of the Clifford algebras $\Cl(V_{\pm})$ we will focus on $\Cl(V_{-})$ instead of $\Cl(V_{+})$, and in the remainder of this work we will continue to do so.
The rigged von Neumann algebra $N_{-}=(\Cl(V_{-})^{\smooth},\mathcal{F}^{\smooth})$ will be denoted by just $N$ in the following.  Nothing is lost by this choice, the theory for $\Cl(V_{+})$ is completely parallel.
\end{remark}

We recall that the Fock space $\mc{F}$ is equipped with the left action of $\Cl(V_{-})$ induced by the inclusion $\iota_{-}:\Cl(V_{-}) \to \Cl(V)$.
Next, we equip it with a compatible right action of $\Cl(V_{-})$.
We note that  $\mathcal{F}$ is naturally graded, as it is the completion of an exterior algebra. 
We let $\operatorname{k}: \mc{F} \rightarrow \mc{F}$ be the \quot{Klein transformation}, that is, $\operatorname{k}$ acts on $\mc{F}$ as the  identity on the even part, and as multiplication by $i$ on the odd part. Note that $\operatorname{k}$ is  unitary. Then we define the operator
\begin{equation}
\label{eq:Jdef}
J \defeq \operatorname{k}^{-1} \Lambda_{\alpha \tau}
\end{equation}
which is an anti-unitary operator on $\mathcal{F}$ with $J^2=1$.
Now, we define a right action as
\begin{align}
\label{eq:rightaction}
\mathcal{F} \times \Cl(V_{-})  \rightarrow \mathcal{F}, \quad
 (v,a) \mapsto Ja^{*}Jv.
\end{align}
We shall then write $v \ract a$ for the right action of $a$ on $v$. 

\begin{lemma}
Right and left actions of $\Cl(V_{-})$ on $\mathcal{F}$ commute.
\end{lemma}

\begin{proof}
It is convenient to use Tomita-Takesaki theory to see this.
The vector $\Omega \defeq 1 \in \Lambda^{0}L \subset \mc{F}$ is cyclic and separating for the von Neumann algebra $\Cl(V_{-})''\subset \mathcal{B}(\mathcal{F})$, see \cite[Section 4.2]{Kristel2019}.
In Tomita-Takesaki theory, one considers the triple $(\Cl(V_{-})'',\mathcal{F},\Omega)$ and associates to it a so-called modular conjugation operator.
In our case, this is precisely the operator $J$, see \cite[Section 4.2]{Kristel2019} and  other references listed there.
A main result of Tomita-Takesaki theory (recalled as \cref{th:tota}) is that $a \mapsto Ja^{*}J$ is an anti-isomorphism of von Neumann algebras from $\Cl(V_{-})''$ onto its commutant.
This shows that the action of $a\in \Cl(V_{-})$ on $\mathcal{F}$ commutes with the one of $Ja^{*}J$.  
\end{proof}

Next, we pass to the rigged setting.
In \cref{prop:SmoothCliffordActionOnF,cor:SmoothHalfAction} we have seen that $\mathcal{F}^{\smooth}$ is a rigged $\Cl(V_{-})^{\smooth}$-module under the left action. The analog is true for the right action of \eqref{eq:rightaction}, as the following result shows.

\begin{lemma}
\label{lem:Friggedcstarbundle}
The right action $(v,a) \mapsto v\ract a$ restricts to an action of $\Cl(V_{-})^{\smooth}$ on $\mathcal{F}^{\smooth}$, and exhibits $\mathcal{F}^{s}$ as a rigged $(\Cl(V_{-})^{\smooth})^{\opp}$-module. In particular, $\mathcal{F}^{\smooth}$ is a rigged $\Cl(V_{-})^{\smooth}$-$\Cl(V_{-})^{\smooth}$-bimodule.
\end{lemma}

\begin{proof}
The proof is mostly standard, and analogous to that of \cite[Proposition 3.2.8]{Kristel2020}.
To carry out the necessary computations one must make use of the formulas
\begin{equation}\label{eq:JCommutesWithTheta}
 J\theta_{\tau(g)}(a)J = \theta_{g}(JaJ), \quad \quad (a \in \Cl(V), \; g \in \O(V)),
\end{equation}
which is  \cite[Lemma 4.8]{Kristel2019}, and
\begin{equation*}
 X(v \ract a) = X(v) \ract a + v \ract \tau(X)(a), \quad \quad (X \in \lie{o}_{\res}(V)),
\end{equation*}
where $\tau(X) = \tau \circ X \circ \tau \in \lie{o}_{\res}(V)$, which follows from Equation \cref{eq:JCommutesWithTheta}.
From this, one obtains a non-commutative binomial expansion, c.f.~\cite[Equation 12]{Kristel2020}, at which point the proof is completely analogous to that of \cref{prop:SmoothCliffordActionOnF}.
\end{proof}

We recall from \cref{cor:SmoothHalfAction} that $\Cl(V_{-})^{\smooth}$ is not just a rigged \cstar-algebra but also forms a rigged von Neumann algebra $N^{\smooth}=N^{\smooth}_{-}=(\Cl(V_{-}^{\smooth}),\mathcal{F}^{\smooth})$, whose defining representation is the left action.

\begin{proposition}\label{prop:FockRiggedBimodule}
$\mathcal{F}^{\smooth}$ is a rigged von Neumann $N^{\smooth}$-$N^{\smooth}$-bimodule.
\end{proposition}

\begin{proof}
It remains to show that the right action, $(\Cl(V_{-})^{\smooth})^{\opp} \times \mathcal{F}^{\smooth} \to \mathcal{F}^{\smooth}$, makes $\mathcal{F}^{\smooth}$ a rigged von Neumann $(N^{\smooth})^{\mathrm{opp}}$-module, where $(N^{\smooth})^{\opp} = ((\Cl(V_{-})^{\smooth})^{\mathrm{opp}},(\mc{F}^{\smooth})^{\sharp})$, see \cref{def:smoothrepriggedvN}.
 For this, all that needs to be shown is that the induced map
 \begin{equation*}
  \rho^{\vee}: ((\Cl(V_{-})^{\smooth})^{\mathrm{opp}})^{\| \cdot \|} \rightarrow \mc{B}(\mc{F}), a \mapsto Ja^{*}J
 \end{equation*}
 extends to a normal $*$-homomorphism
 \begin{equation*}
  (((\Cl(V_{-})^{\smooth})^{\mathrm{opp}})^{\| \cdot \|})'' \rightarrow \mc{B}(\mc{F}).
 \end{equation*}
 Now, let $\rho_{0}$ be the defining representation of $(N^{\smooth})^{\mathrm{opp}}$,  let $c: \mc{F}^{\sharp} \rightarrow \mc{F}$ be the canonical anti-linear isomorphism, and let $\sigma: \mc{B}(\mc{F}^{\sharp}) \rightarrow \mc{B}(\mc{F})$ be the $*$-homomorphism
 \begin{equation*}
  \sigma: T \mapsto Jc T c^{-1}J\text{,}
 \end{equation*}
 then one checks that the following diagram commutes
 \begin{equation*}
  \xymatrix{
  ((\Cl(V_{-})^{\smooth})^{\mathrm{opp}})^{\| \cdot \|} \ar[r]^(.65){\rho_{0}^{\vee}} \ar[rd]_{\rho^{\vee}} & \mc{B}(\mc{F}^{\sharp}) \ar[d]^{\sigma} \\
  & \mc{B}(\mc{F})
  }
 \end{equation*}
 \begin{comment}
  Indeed, if $a \in \Cl(V_{-})^{\mathrm{opp}}$, then $\rho_{0}^{\vee}(a)$ is the operator $\xi \mapsto \xi \circ a$.
  Now, if $v \in \mc{F}$, then we have
  \begin{equation*}
   c \rho_{0}^{\vee}(a) c^{-1} v = a^{*}v,
  \end{equation*}
  whence $Jc\rho_{0}(a)c^{-1}Jv = Ja^{*}Jv$,
  and we are done.
 \end{comment}
 Thus, $\sigma$ extends $\rho^{\vee}$.
 Moreover, since $Jc:\mc{F}^{\sharp} \rightarrow \mc{F}$ is an isometric isomorphism, it follows that $\sigma$ is a normal $*$-homomorphism.
\end{proof}

\label{sec:FockSpaceAsStandardForm}

We recall that the rigged von Neumann algebra $N^{\smooth} =N^{\smooth}_{-}= (\Cl(V_{-})^{\smooth},\mathcal{F}^{\smooth})$ determines a von Neumann algebra $N:=(N^{\smooth})''$, see \cref{re:vnacl}. Moreover, the statement that  $\mathcal{F}^{\smooth}$ is a rigged von Neumann $N^{\smooth}$-$N^{\smooth}$-bimodule implies that its completion, the Fock space $\mathcal{F}$, is an $N$-$N$-bimodule in the classical sense, see \cref{rem:vonNeumanBimoduleCompletion}. 

As we recall in \cref{sec:StandardForms}, a cyclic and separating vector for a representation of a von Neumann algebra equips that representation with the structure of a so-called \emph{standard form} of the von Neumann algebra. 
Using this in our case for $\Omega\in \mathcal{F}$ (see \cite[Section 4.2]{Kristel2019}), we note the following result.

\begin{proposition}
\label{prop:standardform}
$\mathcal{F}$ is a standard form of $N$.
\end{proposition}

\begin{remark}
\label{re:canonicalstandardform}
\begin{enumerate}[(a)]

\item 
From the uniqueness of standard forms
(see \cref{thm:StandardFormIsUnique}) it follows  that $\mc{F}$ is isomorphic to the canonical standard form $L^{2}_{\Omega}(N)$ constructed from the faithful and normal state $N \rightarrow \C, a \mapsto \langle a \lact \Omega, \Omega \rangle$. We recall that  $L^{2}_{\Omega}(N)$ is defined as the completion of $N$ with respect to the bilinear form $(a,b) \mapsto \langle a \lact \Omega, b\lact \Omega \rangle$. 
An explicit isomorphism $u: \mc{F} \to L^{2}_{\Omega} (N)$ is given by the extension of the densely defined map $\mc{F} \to N,  a \lact \Omega\mapsto a$, see \cref{lem:PutInStandardForm}.

\item
A standard form  of a von Neumann algebra  is in a natural way a bimodule for this von Neumann algebra, see \cref{re:bimodulefromstandardform}. The left action is the given representation, while the right action is given by the formula $v \otimes a \mapsto Ja^{*}Jv$, where $J$ is the modular conjugation for the triple  $(\Cl(V_{-})'',\mathcal{F},\Omega)$. In our case, this is precisely the right action we have defined in \cref{eq:rightaction}. In other words, the $N$-$N$-bimodule structure on $\mathcal{F}$ defined above coincides with the bimodule structure obtained from the theory of standard forms; moreover, the isomorphism $u:\mc{F} \to L^{2}_{\Omega} (N)$ is an intertwiner. 

\end{enumerate}
\end{remark}

A standard form of a von Neumann algebra, viewed as a bimodule, is neutral with respect to Connes fusion, as we recall in \cref{lem:L2IsConnesNeutral,co:standardformneutral}. More precisely, there are left and right unitors $\lambda_K: \mathcal{F} \boxtimes K \to K$ and $\rho_K: K \boxtimes \mathcal{F} \to K$ for any $N$-$N$-bimodule $K$, and for $K=\mathcal{F}$ we have coincidence $\lambda_{\mathcal{F}}=\rho_{\mathcal{F}}$. We will denote this unitary intertwiner       
by
\begin{equation*}
 \chi: \mc{F} \boxtimes \mc{F} \rightarrow \mc{F}.
\end{equation*}
We shall need a more explicit description of  $\chi$ to prove one of our  key results, \cref{thm:FusionProductsAgree}.
To make sense of this explicit description, we first recall the basic definition of the Connes fusion product, see \cref{sec:VNABimodules}, and in particular \cref{def:ConnesFusionProduct}, for more details.
We start by writing $\mc{D}(\mc{F},\Omega) \defeq \Hom_{-,N}(L^{2}_{\Omega}(N), \mc{F})$ for the space of bounded linear maps that intertwine the right $N$ actions.
We define a map 
\begin{equation*}
p_{\Omega}: \Hom_{-,N}(L^{2}_{\Omega}(N),L^{2}_{\Omega}(N)) \rightarrow N
\end{equation*}
by requiring
$p_{\Omega}(x) \lact v = x(v)$
for all $v \in L^{2}_{\Omega}(N)$.
Next, we consider the space $\mc{D}(\mc{F},\Omega) \otimes \mc{F}$ equipped with the degenerate inner product
\begin{equation*}
 \langle x \otimes v, y \otimes w \rangle_{\Omega} \defeq \langle v, p_{\Omega}( x^{*} y) \lact w \rangle.
\end{equation*}
The Connes fusion $\mc{F}\boxtimes \mc{F}$ is now  the Hilbert completion of $\mc{D}(\mc{F},\Omega) \otimes \mc{F}/ \ker \langle \cdot, \cdot \rangle_{\Omega}$.
On the space $\mc{D}(\mc{F}, \Omega) \otimes \mc{F}$ the map $\chi:\mathcal{F} \boxtimes \mathcal{F} \to \mathcal{F}$ is given by (see \cref{lem:L2IsConnesNeutral,co:standardformneutral})
\begin{equation*}
 x \otimes v \mapsto p_{\Omega}(u\circ x) \lact v,
\end{equation*}
where $u$ is the invertible intertwiner $u: \mc{F} \to L^{2}_{\Omega} (N)$ of \cref{re:canonicalstandardform}.

\begin{comment}
We recall that we may view $\mc{D}(\mc{F},\Omega)$ as a subspace of $\mc{F}$, through the map $x \mapsto x(\mathds{1})$.
\begin{lemma}
\label{lem:connesfusionwithOmega}
 For every $a \in N$, the vector $a \lact \Omega$ is contained in the subspace $\mc{D}(\mc{F},\Omega)$.
 Moreover, we have
 \begin{equation*}
  \chi((a \lact \Omega) \boxtimes_{\phi} v) = a \lact v,
 \end{equation*}
 for all $v \in \mc{F}$.
\end{lemma}
\begin{proof}
 Let $a \in N$ be arbitrary, then we define a map $x_{a}: N \mapsto \mc{F}$ by $x_{a}(b) = ab \lact \Omega$.
 It is clear that this map is bounded by $\| a \|$ when $N$ is equipped with the norm $\| b \|_{\Omega} = \| b \lact \Omega \|$, and thus yields an element $x_{a} \in \mc{D}(\mc{F},\Omega)$.
 We have $x_{a}(\mathds{1}) = a \lact \Omega$, and thus $a \lact \Omega \in \mc{D}(\mc{F},\Omega)$.
 Now, we compute, for $v \in \mc{F}$,
 \begin{equation*}
  \chi(a \lact \Omega \boxtimes_{\phi}v) = \chi(x_{a} \boxtimes_{\phi} v) = \chi(a \lact x_{\Omega} \boxtimes_{\phi}v) = a\lact \chi(x_{\Omega} \boxtimes_{\phi} v) = a p_{\Omega}(u \circ x_{\Omega}) \lact v.
 \end{equation*}
 It is then easy to verify that $u \circ x_{\Omega} = \mathrm{Id}_{L^{2}_{\Omega}(N)}$, and thus $p_{\Omega}(u \circ x_{\Omega}) = \mathds{1}$, which completes the proof.
\end{proof}
\end{comment}

\begin{remark}
 It is interesting to note that the map $\chi: \mc{D}(\mc{F},\Omega) \otimes \mc{F} \rightarrow \mc{F}$ is surjective, which implies that the space $\mc{D}(\mc{F},\Omega) \otimes \mc{F}/ \ker \langle \cdot, \cdot \rangle_{\Omega}$, is already complete.
\begin{comment}P: If this space isn't complete, then $\chi$ can not be injective. (After all, where would the new points be mapped to?) But $\chi$ is a bijection (on the completion), so it has to be injective.
\end{comment}

\end{remark}

We remark that Connes fusion is coherently associative (\cref{lem:ConnesFusionAssociative}), which allows us to omit bracketing of multiple Connes fusions.  
The following lemma  follows then directly from \cref{co:standardformneutral}.

\begin{lemma}\label{lem:FockFusionAssociative}
 The isomorphism $\chi$ is associative in the sense that the following diagram commutes:
 \begin{equation*}
 \xymatrix@C4em{\mathcal{F} \boxtimes \mathcal{F} \boxtimes \mathcal{F} \ar[r]^-{\id \boxtimes \chi} \ar[d]_{\chi\boxtimes \id} & \mathcal{F} \boxtimes \mathcal{F}\ar[d]^{\chi} \\ \mathcal{F} \boxtimes \mathcal{F} \ar[r]_-{\chi} & \mathcal{F}\text{.}}
 \end{equation*}
\end{lemma}
\begin{comment}
In fact, the following diagram commutes:
\begin{equation*}
 \xymatrix{(\mathcal{F} \boxtimes \mathcal{F}) \boxtimes \mathcal{F} \ar[dr]_{\chi \boxtimes \id_{\mathcal{F}}} \ar[rr] && \mathcal{F} \boxtimes (\mathcal{F} \boxtimes \mathcal{F}) \ar[dl]^{\id_{\mathcal{F}}\boxtimes \chi} \\ & \mathcal{F} \boxtimes \mathcal{F}}
\end{equation*}
This is equivalent to the claim of the lemma: the formula for the diagram in \cref{lem:FockFusionAssociative} is
\begin{equation*}
 \chi \circ (\chi \boxtimes \mathds{1}) = \chi \circ (\mathds{1} \boxtimes \chi) \circ \mathrm{associator}
\end{equation*}
Multiplying with $\chi^{-1}$ from the left on both sides gives the diagram with only three nodes.
\end{comment}

\subsection{Connes fusion of implementers}\label{sec:FusionImpThroughFock}

We return to the orthogonal operator $\alpha\tau:V \to V$ discussed in \cref{sec:ReflectionFreeFermions}, which is not implementable in the classical sense, but will turn out to be implementable by an \emph{anti-unitary} operator.
The following summarizes our results about $\alpha\tau$ from our earlier work \cite[Lemmas 4.4 and 4.6, Propositions 4.9 and 4.11]{Kristel2019}.

\begin{proposition}
\label{lem:kJImplementsKappa}
\label{lem:PropertiesOfKappa}
        The map $\alpha \tau: V \rightarrow V$ extends uniquely to a complex anti-linear \cstar-automorphism $\kappa : \Cl(V) \rightarrow \Cl(V)$ of the Clifford algebra.
Moreover, the anti-unitary operator $\Lambda_{\alpha \tau}: \mc{F} \rightarrow \mc{F}$ \quot{implements}  $\kappa$, i.e., for all $a \in \Cl(V)$ we have
        \begin{equation*}
                \kappa(a) = \Lambda_{\alpha \tau} a \Lambda_{\alpha \tau}
        \end{equation*}
as elements of $\mathcal{B}(\mathcal{F})$. In particular,  $\kappa: \Cl(V) \rightarrow \Cl(V)$ extends uniquely to an anti-unitary automorphism of  $\mc{B}(\mc{F})$.
From there, it restricts to a (smooth) Lie group homomorphism $\kappa: \Imp_{\res}(V) \rightarrow \Imp_{\res}(V)$ covering $\tau: \O_{\res}(V) \rightarrow \O_{\res}(V)$.

\end{proposition}

\begin{remark}
The group $\O_{\res}(V)$ has two connected components, thus $\Imp_{\res}(V)$ has two connected components as well.
All the elements of the identity component of $\Imp_{\res}(V)$ act on $\mc{F}$ by \emph{even} operators and all the remaining elements act by \emph{odd} operators, \cite[Theorem 6.7]{Ara85}.
The following discussion is restricted to even implementers. The reason for this restriction is, basically, that Connes fusion is only available for ungraded bimodules.
As $\tau:\O_L(V) \to \O_L(V)$ and $\kappa$ are Lie group homomorphisms, they restrict to identity components.
\end{remark}

\begin{definition} \label{def:compfusion}
Two even implementers $U,U'\in \Imp_{\res}^{\theta}(V)$ are called \emph{fusable}, if the elements $g_{-} \oplus g_{+}$ and $g'_{-} \oplus g_{+}'$ of $\O_L^{\theta}(V)$ implemented by $U$ and $U'$, respectively, satisfy $g'_{-} = \tau g_{+} \tau$.
\end{definition}

In the remainder of this section we define a product of  fusable implementers using Connes fusion. Let $U \in \Imp_{\res}^{\theta}(V)$ be even and implement an element $g = g_{-} \oplus g_{+} \in \O_L^{\theta}(V)$.
\begin{lemma}
\label{lem:bogminus}
        The Bogoliubov automorphism $\theta_{g_{\pm}}: \Cl(V_{\pm}) \rightarrow \Cl(V_{\pm})$ extends uniquely to an automorphism of $N_{\pm}$, induced by $a \mapsto UaU^{*}$.
\end{lemma}
\begin{proof}
       We discuss $\theta_{g_{-}}$, the proof for $\theta_{g_{+}}$ is analogous. Because $\Cl(V_{-})$ is dense in $N=N_{-}$, it is sufficient to prove existence. We know that $\theta_{g}: \Cl(V) \rightarrow \Cl(V)$ restricts to $\theta_{g_{-}}$ (by \cref{lem:Bogolsplit}) and at the same time extends to an automorphism of $\Cl(V)'' = \mc{B}(\mc{F})$, namely $a \mapsto UaU^{*}$.
       It is thus sufficient to prove that conjugation by $U$ preserves $N\subset \Cl(V)''$. Let $c \in \Cl(V_{-})$ and let $b \in \Cl(V_{-})'$, then, because $UcU^{*} = \theta_{g_{-}}c \in \Cl(V_{-})$, we have
       \begin{equation*}
        U^{*}bUc = U^{*}bUcU^{*}U = U^{*}UcU^{*}bU = cU^{*}bU,
       \end{equation*}
       and hence $U^{*}bU \in \Cl(V_{-})'$. Now let $a \in N$, then we have
       \begin{equation*}
        UaU^{*}b = UaU^{*}bUU^{*} = bU aU^{*},
       \end{equation*}
       and hence $UaU^{*} \in N$.
\end{proof}
We emphasize that the resulting automorphism $\theta_{g_{-}}: N \rightarrow N$ depends on neither $g_{+}$ nor $U$.
It follows from \cref{lem:PropertiesOfKappa,lem:bogminus} that $\theta_{\tau(g_{+})}: \Cl(V_{-}) \rightarrow \Cl(V_{-})$ extends uniquely to an automorphism of $N$, induced by $a \mapsto \kappa(U) a \kappa(U)^{*}$.

\begin{lemma}\label{lem:IntertwinersAndImplementers}
 Let $U: \mc{F} \rightarrow \mc{F}$ be an even unitary map.
 Then, the following are equivalent:
 \begin{enumerate}[(a)]
  \item $U$ implements $g_{-} \oplus g_{+} \in \O^{\theta}(V)_{0}$.
  \item The triple $(\theta_{g_{-}}, \theta_{\tau(g_{+})}, U)$ is a bimodule intertwiner.
 \end{enumerate}
\end{lemma}
\begin{proof}
Suppose that (a) holds.
It then follows from \cref{lem:bogminus} that we have
\begin{equation*}
        U(a \lact v) = \theta_{g_{-}}(a) \lact Uv,
       \quad (a \in N, v \in \mc{F}).
\end{equation*}
Because  $U$ is  even, we have that $U$ commutes with the Klein operator $\operatorname{k}$, and thus, by \cref{lem:kJImplementsKappa} and the definition of the operator $J$ in Equation \cref{eq:Jdef}, we have  $\kappa(U) = JUJ$.
Using this, we compute, again for $a \in N$ and $v \in \mc{F}$
\begin{align*}
 (Uv) \ract \theta_{\tau(g_{+})}(a) &= J \theta_{\tau(g_{+})}(a^{*}) J \lact U v 
 = J \kappa(U) a^{*} \lact  \kappa(U)^{*}J Uv 
 = UJ a^{*} \lact J v 
 = U(v \ract a).
\end{align*}
In other words, the triple $(\theta_{g_{-}},\theta_{\tau (g_{+}) },U)$ is a bimodule intertwiner.

Now, assume (b).
Running the above arguments in reverse, we obtain that, for all $a_{-} \in \Cl(V_{-})$, and all $a_{+} \in \Cl(V_{+})$ we have
\begin{align*}
 U a_{-} U^{*} = \theta_{g_{-}}(a_{-}), \qquad Ua_{+}U^{*} = \theta_{g_{+}}(a_{+}),
\end{align*}
and thus we have  that for all elements $a$ of the algebraic tensor product of $\Cl(V_{-})$ with $\Cl(V_{+})$ the equation
\begin{equation*}
 U a U^{*} = \theta_{g_{-} \oplus g_{+}}(a)
\end{equation*}
holds.
Because the algebraic tensor product is dense in the Clifford \cstar-algebra $\Cl(V_{-} \oplus V_{+}) = \Cl(V)$ this equation holds for all $a \in \Cl(V)$, and this completes the proof.
\end{proof}

%\todo[inline]{$\mathcal{A}_{+}$ here? \textcolor{green}{A}: $\mc{A}_{-}$ is correct, because we conjugate by $J$ twice we stay in $\mc{A}_{-}$, which is what we want because we have turned $\mc{F}$ into an $\mc{A}_{-}$-$\mc{A}_{-}$ bimodule.
%This is confusing. why does it say ,,...extends to an automorphism of $\Cl(V_{+})''$.``? It seems that $\theta_{g_{+}}$ is more difficult than $\theta_{g_{-}}$, and that it would be better to include proper claim and proof in the lemma.}

Now, let $U' \in \Imp_{\res}^{\theta}(V)$ be even, and suppose that it implements $g'= g_{-}' \oplus g_{+}' \in \O_L^{\theta}(V)$ with $g'_{-}=\tau g_{+}\tau$, so that $U$ and $U'$ are fusable in the sense of \cref{def:compfusion}.
We now have three $\ast$-automorphisms of $N$, namely $\theta_{g_{-}}$, $\theta_{\tau(g_{+})}=\theta_{g_{-}'}$, and $\theta_{\tau(g'_{+})}$, and
additionally, we have the bimodule intertwiners $(\theta_{g_{-}},\theta_{\tau(g_{+})},U)$ and $(\theta_{g_{-}'},\theta_{\tau(g_{+}')},U')$  according to \cref{lem:IntertwinersAndImplementers}.
Since Connes fusion is a functor (\cref{lem:ConnesFusionOfMaps}) we obtain a bimodule intertwiner
$(\theta_{g_{-}},\theta_{\tau(g_{+}')},U \boxtimes U'): \mathcal{F} \boxtimes \mathcal{F} \to \mathcal{F}\boxtimes \mathcal{F}$. 

\begin{definition}
\label{eq:ImpFusionDiagram}
The \emph{Connes fusion of fusable implementers  $U$ and $U'$} is the unitary  $\hat\mu(U,U')\in \U(\mathcal{F})$ defined as the composite
\begin{equation*}
        \xymatrix@C=3em{\mc{F} \ar[r]^-{\chi^{-1}} & \mc{F} \boxtimes \mc{F} \ar[r]^{U \boxtimes U'} & \mc{F} \boxtimes \mc{F} \ar[r]^-{\chi} &  \mc{F}.}
\end{equation*}
\end{definition}

By \cref{lem:IntertwinersAndImplementers}, we obtain that the Connes fusion of implementers results into an even implementer $\hat\mu(U,U')\in \Imp_L^{\theta}(V)$, which implements $g_{-}\oplus g'_{+}\in \O_L^{\theta}(V)$.  
\begin{remark}
Note that, in particular, this shows that if $g_{-} \oplus g_{+} \in \O_L^{\theta}(V)$ and $\tau g_{+} \tau \oplus g_{+}' \in \O_L^{\theta}(V)$ are in the identity component of $\O_{\res}(V)$, then $g_{-} \oplus g_{+}' \in \O_L^{\theta}(V)$. \end{remark}

The following result follows directly from the functoriality of Connes fusion (see \cref{lem:ConnesFusionOfMaps}).

\begin{proposition}\label{lem:multiplicativitymuhat}
         Connes fusion $\hat{\mu}$ of implementers  is multiplicative in the following sense. Let $U,U',V,V' \in \Imp_L^{\theta}(V)$ be even implementers, of which $U$ and $U'$ are fusable, and $V$ and $V'$ are fusable. Then, $UV$ and $U'V'$ are fusable, and we have 
        \begin{equation*}
        \hat{\mu}(U,U') \hat{\mu}(V,V') = \hat{\mu}(UV,U'V').
        \end{equation*}
\end{proposition}

In \cref{sec:operatormodel} we further study the Connes fusion of implementers.
In particular, we relate it to the fusion product on the universal central extension of $L\Spin(d)$, see \cref{thm:FusionProductsAgree}. This relation will be crucial for the proof of our main result, \cref{th:fusionfibrewise}.

\section{Fusive spin structures on loop space}

In this section, we recall and summarize the notions of spin structures on loop space,  fusion products, and string structures, and we relate the Connes fusion of implementers to loop fusion (\cref{thm:FusionProductsAgree}).     

\label{sec:fusion}

\subsection{The spinor bundle on loop space}

\label{sec:spinstructuresandspinorbundles}

At the basis of the construction of the spinor bundle on loop space is the notion of a spin structure on loop space according to Killingback \cite{killingback1}.
We suppose $M$ is a spin manifold of dimension $d$, and consider its spin structure as a principal $\Spin(d)$-bundle $\Spin(M)$ over $M$ that lifts the structure group of the oriented orthonormal frame bundle $\SO(M)$ along the central extension  $\Z_2 \to\Spin(d) \to \SO(d)$.
\begin{comment}
That is, there is a bundle map
$\Spin(M) \to \SO(M)$
that intertwines the group actions along the projection $\Spin(d) \to \SO(d)$. \end{comment}

By definition of the Fr\'echet manifold structure on $LM$, a tangent vector at a loop $\gamma\in LM$  is a smooth section of $TM$ along $\gamma$. 
This is the motivation to consider $L\Spin(M)$, the Fr\'echet manifold of smooth free loops in the total space of $\Spin(M)$, as (a version of) the frame bundle of $LM$. We note that  $L\Spin(M)$ is  a Fr\'echet principal $L\Spin(d)$-bundle over $LM$, see \cite[Proposition 1.8]{SW}, where $L\Spin(d)$ is the loop group of $\Spin(d)$.
The following definition is due to Killingback \cite{killingback1}.

\begin{definition}
\label{def:spinstructureLM}
A \emph{spin structure on $LM$} is a lift of  the structure group of  the frame bundle $L\Spin(M)$ of $LM$ along the basic central extension
\begin{equation*}
 \U(1) \rightarrow \BCE \rightarrow L\Spin(d).
\end{equation*}
\end{definition}

Thus, a spin structure on $LM$ is a principal $\BCE$-bundle $\SStruct$ over $LM$ together with a bundle map
$\sigma: \SStruct \to L\Spin(M)$
that intertwines the group actions along the projection $\smash{\widetilde{L\Spin}}(d) \to L\Spin(d)$. 
We remark that the map $\sigma:\SStruct \to L\Spin(M)$ has itself the structure of a principal $\U(1)$-bundle, where the $\U(1)$-action is induced along the inclusion $\U(1)\subset \smash{\widetilde{L\Spin}}(d)$. Concerning the  existence of spin structures on loop spaces, the following result was proved by McLaughlin \cite{mclaughlin1}. 

\begin{proposition}
\label{prop:classspinLM}
The loop space $LM$ of a spin manifold $M$ admits a spin structure if $\frac{1}{2}p_1(M)=0$.
\end{proposition}

The group $L\SO(d)$ acts in $V=L^{2}(S^{1}, \mathbb{S} \otimes \C^{d})$ by pointwise multiplication; under the pointwise projection map $L\Spin(d) \rightarrow L\SO(d)$, this defines an action of $L\Spin(d)$ in $V$.
It is clear that $L\Spin(d)$ acts on $V$ through $\O^{\theta}(V)$, and a standard result that it acts through $\O_{\res}(V)$, see \cite[Lemma 3.21]{Kristel2019} or \cite[Proposition 6.3.1]{PS86}. Moreover, the resulting map 
\begin{equation*}
\omega:L\Spin(d) \rightarrow \O^{\theta}_{\res}(V)
\end{equation*}
is smooth, \cite[Lemma 3.22]{Kristel2019}.
The following result is \cite[Theorem 3.26]{Kristel2019}.
\begin{theorem}
\label{basicce}
    If $1<d\neq 4$, then the pullback of the central extension $\U(1) \to\Imp_{\res}^{\theta}(V) \xrightarrow{q} \O_{\res}^{\theta}(V)$ along $\omega:L\Spin(d) \to \O_L^{\theta}(V)$ is the basic central extension of $L\Spin(d)$.
\end{theorem}

From now on we will use this fixed model for the basic central extension.  Most importantly, it comes with a smooth map $\tilde \omega: \smash{\BCE} \to \Imp_L^{\theta}(V)$, from which we obtain a smooth representation of the basic central extension on the Fock space $\mathcal{F}^{\smooth}$. The spinor bundle on loop space is obtained using the associated bundle construction of \cref{lem:AssociatedFrechetBundle}.

\begin{definition}
\label{def:SpinorBundle}
Let $M$ be a spin manifold equipped with a spin structure $\SStruct$ on its loop space.  The \emph{spinor bundle on loop space} is the associated rigged Hilbert space bundle
 \begin{equation*}
    \mc{F}^{\smooth}(LM) \defeq \left( \SStruct \times \mc{F}^{\smooth} \right) / \BCE\text{.}
 \end{equation*}
\end{definition}

The Clifford bundle is obtained in a similar way using \cref{lem:cstarbundleass} and the smooth representation of $L\Spin(d)$ on $\Cl(V)^{\smooth}$, induced via $\omega:L\Spin(d) \to \O(V)$  from the action of $\O(V)$ on $\Cl(V)^{\smooth}$ by Bogoliubov automorphisms.
\begin{definition}\label{def:CliffordBundle}
 The \emph{Clifford bundle on loop space} is the associated rigged \cstar-algebra bundle
 \begin{equation*}
    \Cl^{\smooth}(LM) \defeq (L \Spin(M) \times \Clsm)/ L\Spin(d).
 \end{equation*}
\end{definition}

\begin{remark}
\label{re:cliffordbundlewithoutspin}
We remark that the Clifford algebra bundle does not make use of a spin structure on $LM$, and can even be defined without assuming a spin structure on $M$, since $\omega$ factors through $L\SO(d)$. 
\end{remark}

\begin{comment}
Indeed, we have a canonical isomorphism 
\begin{equation*} 
 (L\SO(M) \times \Clsm)/L\SO(d) \cong (L\Spin(M) \times \Clsm)/ L\Spin(d)
\end{equation*}
of associated bundles, 
so that the Clifford bundle on loop space can be defined equally well in the case that $M$ is not spin but only oriented. For the purpose of this paper, we always assume that $M$ is a spin manifold; then, the  expression in \cref{def:CliffordBundle} makes the following sections lighter on notation.
\end{comment}

The Clifford bundle $\Cl^{\smooth}(LM)$ acts on the spinor bundle $\mc{F}^{\smooth}(LM)$ by \quot{Clifford multiplication}. To make this precise, we consider the rigged $\Cl(V)^{\smooth}$-module $\mathcal{F}^{\smooth}$ of \cref{prop:SmoothCliffordActionOnF}, with representation $(a,v) \mapsto a\lact v$.
We have proved in \cite[Proposition 2.2.19]{Kristel2020} that this representation extends from the typical fibres to all fibres, resulting in the following statement.

\begin{proposition}
\label{prop:cliffmult}
There is a unique bundle map
\begin{equation*}
 \Cl^{\smooth}(LM) \times_{LM} \mc{F}^{\smooth}(LM) \rightarrow \mc{F}^{\smooth}(LM)
\end{equation*}
such that
\begin{equation*}
 ([\ph,a],[\tilde \varphi,v]) \mapsto [\tilde \varphi, a \lact v],
\end{equation*}
for all $\tilde \varphi \in \widetilde{L\Spin}(M)$ lifting $\ph \in L\Spin(M)$, and all $a \in \Clsm$ and $v \in \mc{F}^{\smooth}$. Moreover, this map equips the spinor bundle $\mc{F}^{\smooth}(LM)$ with the structure of a rigged $\Cl^{\smooth}(LM)$-module bundle with typical fibre $\mathcal{F}^{\smooth}$. 
\end{proposition}

We write $\mc{F}(LM)$ for the continuous Hilbert space bundle whose fibres are the completions of the fibres of $\mc{F}^{\smooth}(LM)$ (see \cref{sec:locallytrivialbundles}).
Similarly, we write $\Cl(LM)$ for the continuous \cstar-algebra bundle whose fibres are the completions of the fibres of $\Cl^{\smooth}(LM)$.
The smooth representation $\rho$ of $\Cl^{\smooth}(LM)$ on $\mc{F}^{\smooth}(LM)$ then extends to a continuous representation of $\Cl(LM)$ of $\mc{F}(LM)$.

As the typical fibre of the rigged $\Cl^{\smooth}(LM)$-module bundle  $\mc{F}^{\smooth}(LM)$ is in fact a rigged von Neumann algebra, $\ClvN(V)^\smooth=(\Cl(V)^{\smooth},\mathcal{F}^{s})$ (\cref{prop:SmoothCliffordActionOnF}), we  note immediately the following consequence of \cref{prop:cliffmult}.  

\begin{corollary}
The pair $\ClvN(LM)^\smooth\defeq (\Cl^{\smooth}(LM),\mathcal{F}^{\smooth}(LM))$ is a rigged von Neumann algebra bundle over $LM$ with typical fibre $\ClvN(V)^\smooth$.
\end{corollary}

\begin{remark}
 Let us comment on the relation between the  spinor bundle on loop space defined in our earlier work using coarser riggings $\mathcal{F}^{\infty}$ of Fock spaces and $\Cl(V)^{\infty}$ of Clifford algebras, \cite[Section 4]{Kristel2020}.
 As remarked in \cref{rem:compareriggings}, we have $\mc{F}^{\infty} \subset \mc{F}^{\smooth}$, in other words, there is an isometric morphism of rigged Hilbert spaces $\mc{F}^{\infty} \rightarrow \mc{F}^{\smooth}$, which induces the identity on the completion $\mathcal{F}$; and similarly for the rigged Clifford algebras.
 Comparing \cref{def:SpinorBundle,def:CliffordBundle} with \cite[Definitions 4.4 and 4.3]{Kristel2020} respectively, we observe that the above-mentioned isometric morphisms induce isometric morphisms on the level of the spinor bundle and the Clifford algebra bundle on loop space:
 \begin{align*}
  \mc{F}^{\infty}(LM) &\rightarrow \mc{F}^{\smooth}(LM) & \Cl^{\infty}(LM) \rightarrow \Cl^{\smooth}(LM).
 \end{align*}
 Moreover, this pair of isometric morphisms in fact is an isometric  intertwiner as in \cref{def:intertwinerriggedmodulebundles}. On the completions, these isometric intertwiners induce the identity maps.
\end{remark}

\subsection{Fusion on loop space}

\label{sec:genfusionproducts}
\label{sec:StringsGroupsFusion}

We review first the general notion of a fusion product for principal $\U(1)$-bundles over the loop space $LM$ of a smooth manifold $M$, following \cite{waldorf10}. In the subsequent subsections we apply it to two situations: central extensions of loop groups and spin structures on loop spaces.
We write $PM$ for the set of smooth paths in $M$ with sitting instants, i.e.,
\begin{equation}
\label{eq:pathsinG}
PM \defeq  \{\beta:[0,\pi] \to M \;|\; \beta\text{ is smooth and constant around $0$ and $\pi$} \}\text{.}
\end{equation} 
We use sitting instants so that we are able to concatenate arbitrary paths with a common end point: the usual path concatenation $\beta_2 \star \beta_1$ is again a smooth path whenever $\beta_1(\pi)=\beta_2(0)$. 
Unfortunately, with sitting instants, $PM$ is not any kind of manifold; however, we may regard it as a diffeological space, see \cref{sec:HilbertBundle:3} for a quick review.
The plots of $PM$ are all maps $c:U \to PM$ whose adjoint map $c^{\vee}: U \times [0,\pi] \to M$, with $c^{\vee}(u,t) \defeq c(u)(t)$ is smooth.  We remark that path concatenation $\star$ and path reversal $\beta \mapsto \bar\beta$ are smooth maps.
The evaluation map $\mathrm{ev}: PM \to M \times M, \beta \mapsto (\beta(0),\beta(\pi))$ is a smooth map, and since diffeological spaces admit arbitrary fibre products, the iterated fibre products $PM^{[k]} \defeq  PM \times_{M \times M} ... \times_{M \times M} PM$ are again diffeological spaces: their plots are  simply tuples $(c_1,...,c_k)$ of plots of $PM$, such that $\mathrm{ev} \circ c_1 = ... = \mathrm{ev} \circ c_k$.
In the following, we will use the smooth map
\begin{align*}
PM^{[2]} &\to LM\;,\; (\beta_1,\beta_2) \mapsto \beta_1 \cup \beta_2 \defeq  \bar\beta_2 \star \beta_1
\end{align*}
that combines two paths with a common initial point and a common end point to a loop.
If $(\beta_1,\beta_2,\beta_3)\in PM^{[3]}$, we will regard the loop $\beta_1\cup \beta_3$ as the \quot{fusion} of the loops $\beta_1 \cup \beta_2$ and $\beta_2 \cup \beta_3$. A fusion product on a principal $\U(1)$-bundle over $LM$ is now a lift of this fusion operation to the total space:

\begin{definition}
\label{def:fusion product}
Let $\pi:\mathcal{L} \to LM$ be a Fr\'echet principal $\U(1)$-bundle over $LM$.
A \emph{fusion product} on $\mathcal{L}$  assigns to each element $(\beta_1,\beta_2,\beta_3) \in PM^{[3]}$  a  $\U(1)$-bilinear map
\begin{equation*}
\lambda_{\beta_1,\beta_2,\beta_3}: \mathcal{L}_{\beta_1 \cup \beta_2}  \times \mathcal{L}_{\beta_2\cup\beta_3} \to \mathcal{L}_{\beta_1 \cup \beta_3}\text{,}
\end{equation*} 
such that the following two conditions are satisfied: 
\begin{enumerate}[(i)]

\item 
Associativity: for all $(\beta_1,\beta_2,\beta_3,\beta_4) \in PM^{[4]}$ and all $q_{ij} \in \mathcal{L}_{\beta_i \cup \beta_j}$, 
\begin{equation*}
\lambda_{\beta_1,\beta_3,\beta_4}(\lambda_{\beta_1,\beta_2,\beta_3}(q_{12} , q_{23}) , q_{34})= \lambda_{\beta_1,\beta_2,\beta_4}(q_{12} , \lambda_{\beta_2,\beta_3,\beta_4}(q_{23} , q_{34}) )\text{.}
\end{equation*}

\item
Smoothness: for every plot $(c_1,c_2,c_3):U \to PM^{[3]}$ and all smooth maps $c_{12},c_{23}:U \to \mathcal{L}$ such that $(\pi\circ c_{ij})(x)=c_i(x) \cup c_j(x)$ for all $x\in U$ and $ij\in\{12,23\}$, 
the map
\begin{equation*}
U \to \mathcal{L} : x \mapsto \lambda_{c_1(x),c_2(x),c_3(x)}(c_{12}(x),c_{23}(x))
\end{equation*}
is smooth.

\end{enumerate}
\end{definition}

Early versions of fusion products have been studied in \cite{brylinski1} and in \cite{stolz3}. For a more complete treatment of this topic we refer to \cite{waldorf10}. The smoothness condition (ii) has several equivalent formulations; for instance, the ones given in  \cite{waldorfa,waldorf10,Kristel2019}. \begin{comment}
One can require that the maps $\lambda_{\beta_1,\beta_2,\beta_3}$ form a smooth  homomorphism $\cup_{12}^{*}\mathcal{L} \otimes \cup_{23}^{*}\mathcal{L} \to \cup_{13}^{*}\mathcal{L}$ of principal $\U(1)$-bundles over the diffeological space $PM^{[3]}$ (see \cref{sec:HilbertBundle:3}). 
\end{comment}
Fusion products are a characteristic  feature of structure in the image of transgression. The main result of \cite{waldorf10} is that principal $\U(1)$-bundles over $LM$ equipped with a fusion product and a thin homotopy equivariant structure form a category that is equivalent to the category of bundle gerbes over $M$, with the equivalence established by transgression and regression functors. In particular, every principal $\U(1)$-bundle over $LM$ obtained by transgression of a bundle gerbe over $M$ comes equipped with a fusion product.

\subsection{Fusion in the basic central extension}

The first occurrence of a fusion product in this paper is when the smooth manifold $M$ is a Lie group $G$, and the principal $\U(1)$-bundle $\mathcal{L}$ is a central extension
$\U(1) \to \mathcal{L} \to LG$
of the loop group. The diffeological space $PG$ of paths with sitting instants becomes now a diffeological group, i.e., (pointwise) multiplication and (pointwise) inversion are smooth maps; further, path concatenation and path reversal are smooth group homomorphisms. We require compatibility between fusion products and these group structures in the following sense.

\begin{definition}
A fusion product $\mu$ on a central extension $\U(1)\to\mathcal{L}\to LG$ is called \emph{multiplicative}, if it is a group homomorphism; i.e., \begin{equation*}
\mu_{\beta_1,\beta_2,\beta_3}(q_{12} , q_{23}) \cdot \mu_{\beta_1',\beta_2',\beta_3'}(q_{12}' , q_{23}') = \mu_{\beta_1\beta_1',\beta_2\beta_2',\beta_3\beta_3'}(q^{}_{12}q_{12}' , q^{}_{23}q_{23}')
\end{equation*}
for all $(\beta_1,\beta_2,\beta_3),(\beta_1',\beta_2',\beta_3') \in PG^{[3]}$,   $q_{ij}\in \mathcal{L}_{\beta_i\cup \beta_j}$, and $q'_{ij}\in \mathcal{L}_{\beta_i' \cup \beta_j'}$. 
\end{definition}

The basic central extension of any compact simple Lie group can be obtained by transgression, and thus comes equipped with a multiplicative fusion product, which is unique up to isomorphism  \cite{Waldorfc}. In this paper, we need a precise formula for this fusion product $\mu$ on the basic central extension of $L\Spin(d)$. Within the operator-theoretic model for $\widetilde{L\Spin}(d)$ described above, such a formula has been obtained  in  \cite{Kristel2019}, and we shall recall this. 
\label{sec:operatormodel}
A key ingredient to this work will be a  relation between that fusion product $\mu$ and the Connes fusion of implementers defined in \cref{sec:FusionImpThroughFock}.

We first note that the map $\omega: L\Spin(d) \to \O_L^{\theta}(V)$ that defines our operator-theoretic model via pullback, is compatible with all relations between loops and paths in $\Spin(d)$.
To see this, let $p_{\pm}:\O^{\theta}_L(V) \to \O(V_{\pm})$ denote the projection. Further, let $\Delta: P \Spin(d) \rightarrow L\Spin(d), \beta \mapsto \beta \cup \beta$ be the doubling map.
We define  smooth maps $\omega_{\pm}: P\Spin(d) \to \O(V_{\pm})$
by $\omega_{\pm}(\beta) \defeq p_{\pm}(\omega(\Delta(\beta)))$.
Due to the pointwise definition of the action of $L\Spin(d)$ on $V$, it is then clear that the following diagram is commutative:
\begin{equation}
\label{eq:diagramomega}
\xymatrix@R=2em{P\Spin(d) \ar[dd]_{\omega_-} & P\Spin(d)^{[2]}\ar[l]_{\pr_1} \ar[r]^{\pr_2} \ar[d]^{\cup} & P\Spin(d) \ar[dd]^{\omega_{+}} \\ & L\Spin(d) \ar[d]^{\omega} \\ \O(V_{-}) & \O_L^{\theta}(V)\ar[r]_{p_{+}}\ar[l]^{p_{-}} & \O(V_{+})}
\end{equation}

Next, we recall from \cite[Definition 5.5]{Kristel2019} that a \emph{fusion factorization}  $\rho: P\Spin(d) \rightarrow \BCE$ is a smooth group homomorphism with the property that the following diagram commutes:
        \begin{equation*}
        \begin{tikzpicture}
                \node (B) at (2,1) {$\BCE$};
                \node (C) at (-.5,0) {$P\Spin(d)$};
                \node (D) at (2,0) {$L\Spin(d)$.};
                \path[->,font=\scriptsize]
                (C) edge node[above]{$\rho$} (B)
                (B) edge (D)
                (C) edge node[below]{$\Delta$} (D);
        \end{tikzpicture}
        \end{equation*}
Such a fusion factorization was constructed in \cite[Section 5.3]{Kristel2019};
it is uniquely characterized by the property that $\rho(\beta)$ satisfies $\tilde{\omega}(\rho(\beta))J=J\tilde{\omega}(\rho(\beta))$ and $P_{\Omega}=P_{\tilde{\omega}(\rho(\beta))\Omega}$, where $\tilde\omega:\BCE \to \Imp_L^{\theta}(V)$ is the map from \cref{sec:spinstructuresandspinorbundles}, $J$ is the modular conjugation (see Eq. \cref{eq:Jdef}) and $P_{\xi}$ is the closed self-dual cone corresponding to $\xi$ (see \cref{sec:StandardForms}).
The fusion product $\mu$ on $\BCE$ is then defined by 
\begin{equation*}
 \mu_{\beta_1,\beta_2,\beta_3}(g_{1} , g_{2}) \defeq g_{1} \rho(\beta_{2})^{-1} g_{2}\text{,}
\end{equation*}
where $(\beta_1,\beta_2,\beta_3) \in P \Spin(d)^{[3]}$,  $g_{1}  \in \BCE_{\beta_1\cup\beta_2}$, and $g_2\in  \BCE_{\beta_2\cup\beta_3}$, see \cite[Eq. (13)]{Kristel2019}.

The following  key result now tells us that the fusion product $\mu$ and the Connes fusion map  $\hat{\mu}$ of \cref{sec:FusionImpThroughFock} coincide under the group homomorphism $\tilde\omega$. 
We believe that this result is the first place where  a relation between loop space fusion and Connes fusion is established. 

\begin{theorem}\label{thm:FusionProductsAgree}
        Let $\beta=(\beta_1,\beta_2,\beta_3) \in P \Spin(d)^{[3]}$,          
          $g_{1}  \in \BCE_{\beta_1\cup\beta_2}$ and $g_2\in  \BCE_{\beta_2\cup\beta_3}$.
Then, the implementers $ \tilde \omega(g_1)$ and $ \tilde \omega(g_2)$ are  fusable in the sense of \cref{def:compfusion}. Moreover, $\tilde\omega$ exchanges fusion on the basic central extension with the Connes fusion of implementers, i.e., we have        
\begin{equation*}
\tilde \omega(\mu_{\beta_1,\beta_2,\beta_3}(g_{1} , g_{2})) = \hat{\mu}(\tilde \omega(g_{1}),\tilde \omega(g_{2})).
\end{equation*}
\end{theorem}

\begin{proof}
Since  the map $\omega:L\Spin(d) \rightarrow \O_{\res}(V)$ factors through $L\SO(d)$, its image is contained in the connected component of the identity of $\osplit$. Thus, all implementers in the image of $\tilde\omega$ are even, which is a prerequisite for being fusable (\cref{def:compfusion}). For simplicity, we set $U_i\defeq\tilde\omega(g_i)$ for $i=1,2$. 
Let $K \defeq \tilde\omega(\rho(\beta_{2}))^{-1}$, so that $\tilde\omega(\mu_{\beta_1,\beta_2,\beta_3}) (g_{1} , g_{2}) = U_{1}KU_{2}$.
        We compute, using the multiplicativity of $\hat{\mu}$ (\cref{lem:multiplicativitymuhat}),
        \begin{equation*}
                \hat{\mu}(U_{1},U_{2}) = \hat{\mu}(U_{1}KK^{-1},U_{2}KK^{-1}) = \hat{\mu}(U_{1}K, U_{2}K) \hat{\mu}(K^{-1},K^{-1}).
        \end{equation*}
        We claim that
        \begin{align*}
                \hat{\mu}(U_{1}K, U_{2}K)= U_{1}KU_{2}K 
        \quad\text{ and }  \quad
                \hat{\mu}(K^{-1},K^{-1}) = K^{-1};
        \end{align*}
        this proves the theorem. 
        The element $K$ has the property that $U_{1}K$ implements an operator of the form $g_{-} \oplus \mathds{1} \in \O^{\theta}_L(V)$, see \cite[Eq.~(13)]{Kristel2019}, which implies that $U_{1}K \in (\Cl(V_{+})_{0}'')'$, see \cref{lem:IntertwinersAndImplementers}.
        Then,  the unitary $\overline{\nu}_{2}$ used in the construction of $U_{1}K\boxtimes U_{2}K$ (see \cref{lem:ConnesFusionOfMaps}) is  the identity.
        Let $f: L^{2}_{\Omega}(\Cl(V_{-})'') \rightarrow \mc{F}$ be the inverse of the isomorphism $u$ defined in \cref{re:canonicalstandardform}, given by $a \mapsto a\lact \Omega$.
        Let $v \in \mc{F}$, then to determine $\hat{\mu}(U_{1}K,U_{2}K)$ we compute
        \begin{align*}
            \chi( (U_{1}K \boxtimes U_{2}K) \chi^{-1} v) &= \chi (U_{1}K f \otimes U_{2}Kv) \\
            &= \chi (U_{1}K f \otimes v \ract JK^{*}U_{2}^{*}J) \\
            &= U_{1}K v \ract JK^{*}U_{2}^{*}J \\
            &= U_{1}K U_{2}K v;
         \end{align*}
         this proves the first equation.
         Because of the fact that $K^{-1}$ has the property that $JK^{-1} = K^{-1}J$ and $P_{K^{-1}\Omega} = P_{\Omega}$, \cite[Lemma 5.16]{Kristel2019}, we see that the corresponding map $L^{2}_{\Omega}(\Cl(V_{-})'') \rightarrow L^{2}_{\Omega}(\Cl(V_{-})'')$ is $f^{-1}K^{-1}f$. We compute:
         \begin{align*}
            (K^{-1}\boxtimes K^{-1})(f \otimes v) &= K^{-1} f f^{-1} K f \otimes K^{-1}v = f \otimes K^{-1} v;
         \end{align*}
         this proves the second equation.
\end{proof}

\subsection{Fusive spin structures  and string structures}

We return to the discussion of spin structures on the loop space of a spin manifold $M$. We recall that a spin structure $\SStruct$ on $LM$ can be viewed as a principal $\U(1)$-bundle over $L\Spin(M)$, the total space of the frame bundle of $LM$, and as such, may host fusion products. The following definition is \cite[Definition 3.6]{waldorfa}.

\begin{definition} 
\label{def:FusiveSpinStructure}
 A \emph{fusion product} on a spin structure $\SStruct$ on $LM$ is 
a fusion product $\lambda$ on the principal $\U(1)$-bundle $\SStruct$ over $L\Spin(M)$ that is equivariant with respect to the fusion product $\mu$ on $\BCE$ under the principal action, i.e., 
\begin{equation}\label{eq:FusionIsEquivariant}
  \lambda_{\beta_1\gamma_1,\beta_2\gamma_2,\beta_3\gamma_3}( \tilde\varphi_{12} \cdot g_{12} \otimes \tilde\varphi_{23} \cdot g_{23}) = \lambda_{\beta_1,\beta_2,\beta_3}(\tilde\varphi_{12} \otimes \tilde\varphi_{23}) \cdot \mu_{\gamma_1,\gamma_2,\gamma_3}(g_{12} \otimes g_{23})
 \end{equation}
 for all $(\beta_1,\beta_2,\beta_3)\in P\Spin(M)^{[3]}$, $(\gamma_1,\gamma_2,\gamma_3) \in P\Spin(d)^{[3]}$, $\tilde\varphi_{ij}\in \SStruct_{\beta_i\cup \beta_j}$, and $g_{ij}\in \widetilde{L\Spin(d)}_{\gamma_i \cup \gamma_j}$.
A spin structure on $LM$ with a fusion product is called a \emph{fusive spin structure}.
\end{definition}

Fusive spin structures bring us one step forward on the way from spin structures on $LM$ to string structures on $M$. They already fix the missing \quot{only if}-part of \cref{prop:classspinLM}, as shown in \cite[Theorem 1.4]{waldorfa}:
 $LM$ admits a fusive spin structure if and only if $\frac{1}{2}p_1(M)=0$. A full string structure yet contains more information, a so-called \emph{thin homotopy equivariant structure}, however, this information is not needed for the construction of a fusion product on the spinor bundle on loop space, which we perform in \cref{sec:SpinorBundleOnLoopSpaceI}. There, we only need a fusive spin structure.

\label{sec:stringstructures}

In the following we want to explain how a fusive spin structure on $LM$, the main ingredient of our construction in \cref{sec:SpinorBundleOnLoopSpaceI}, can be obtained from a (geometric) string structure on $M$.
 There are essentially four different -- but equivalent -- ways to say what a string structure on a spin manifold $M$ is. All four versions have in common that their existence is obstructed by $\frac{1}{2}p_1(M) \in H^4(M,\Z)$, and that -- under  appropriate notions of equivalence -- they form  a torsor over the group $H^3(M,\Z)$. 

\begin{enumerate}[(1)]

\item 
In purely topological terms, a string structure is a lift of the structure group  of the spin frame bundle $\Spin(M)$ to the 3-connected covering group of
$\Spin(d)$ \cite{ST04}. That covering group, the \emph{topological string group}, does not allow (finite-dimensional) Lie  group structures (it has cohomology in infinitely many degrees).

\item
In terms of higher-categorical structures, a string structure is a  lift of the structure group of the spin frame bundle $\Spin(M)$ along the central extension of  Lie 2-groups, 
\begin{equation*}
B\U(1) \to \String(d) \to \Spin(d)\text{,}
\end{equation*}
where $B\U(1)$ denotes the Lie 2-group with a single object, $\Spin(d)$ is considered as a Lie 2-group with only identity morphisms, and $\String(d)$ is, for instance, the \emph{String Lie 2-group} constructed in \cite{baez9} (this is strict, but infinite-dimensional) or in \cite{pries2} (this is finite-dimensional, but not strict), or in \cite{Waldorf} (again strict, and with diffeological spaces). Geometric realization establishes the relation with (1), see \cite{baez8,Nikolaus}.
\begin{comment}
The equivalence between  (1) and (2) follows from a result of Baez-Stevenson  that identifies homotopy classes of maps into the classifying space of the geometric realization of a Lie 2-groups, as in  (1), with the continuous non-abelian cohomology with values in that Lie 2-group. Then,  results of  show that continuous and smooth non-abelian cohomology of manifolds coincide, and show further that (smooth) non-abelian cohomology classifies (smooth) principal 2-bundles for Lie 2-groups, as in  (2), see \cite[Theorem 4.6]{Nikolaus}.
\end{comment}

\item
In terms of bundle gerbes, a string structure is a trivialization of the \emph{Chern-Simons 2-gerbe} over $M$ \cite{waldorf8}. We explain below some more details about this approach. Its equivalence with (2) was established in \cite[Theorem 7.9]{Nikolausa}.

\item
In terms of loop space geometry, a string structure  is a thin homotopy equivariant fusive spin structure on $LM$; this definition and its equivalence with (3) is in  \cite{waldorfb}; we will recall it below. 

\end{enumerate}

Now we want to describe some details about (3) in condensed form. A trivialization of the Chern-Simons 2-gerbe (i.e., a string structure) is a triple $(\mathcal{S},\mathcal{A},\sigma)$ consisting of the following structure. \begin{itemize}
\item 
A bundle gerbe $\mathcal{S}$ in the sense of Murray \cite{Murray1996} over  $\Spin(M)$.

\item
A bundle gerbe isomorphism 
\begin{equation*}
\mathcal{A}: \delta^{*}\mathcal{G}_{bas} \otimes \pr_2^{*}\mathcal{S} \to \pr_1^{*}\mathcal{S}
\end{equation*}
over the double fibre product $\Spin(M)^{[2]}$, where $\mathcal{G}_{bas}$ is the \emph{basic gerbe} over $\Spin(d)$ constructed by Meinrenken \cite{meinrenken1} and Gawedzki-Reis \cite{gawedzki1}, and $\delta(\varphi,\varphi')\in \Spin(d)$ is defined by $\varphi\delta(p,p')=\varphi'$ for frames $\varphi,\varphi'\in L\Spin(M)$ at the same point.

\item
A certain 2-isomorphism $\sigma$ over the triple fibre product, expressing a compatibility condition between $\mathcal{A}$ and the multiplicativity of $\mathcal{G}_{bas}$. For the details we refer to \cite{waldorf8}.

\end{itemize}
A major advantage of  this notion of a string structure is that it allows  a differential refinement  by \emph{string connections} \cite{waldorf8}. A string connection consists of a connection on the bundle gerbe $\mathcal{S}$, such that the bundle gerbe morphism $\mathcal{A}$ is connection-preserving (the basic gerbe $\mathcal{G}_{bas}$ has a canonical connection); again, there are additional compatibility conditions related to the multiplicativity of the connection on the basic gerbe. In the present context, string connections are useful for establishing the relation between  (3) and (4), which we explain in the following.

The key technique is a transgression functor
\begin{equation*}
\mathscr{T}: \mathrm{h}_1\grbcon{\relax}{M} \to \ufusbun {LM} 
\end{equation*}
from the homotopy 1-category of the bicategory of bundle gerbes with connection over $M$ to the category of principal $\U(1)$-bundles over $LM$ equipped with fusion products, in the sense of \cref{def:fusion product}. Versions of this functor  have been described in \cite{gawedzki3,brylinski1}, the complete  construction is in \cite{waldorf10}. 
As a prerequisite, we apply the transgression functor to the basic gerbe $\mathcal{G}_{bas}$ over $\Spin(d)$, and obtain a principal $\U(1)$-bundle $\mathscr{T}(\mathcal{G}_{bas})$ with fusion product over the loop group $L\Spin(d)$. Functoriality allows to transgress multiplicativity; thus, what we really obtain is a central extension with a multiplicative fusion product. In fact, it is the basic central extension \cite{waldorf5}, and we have proved in \cite[Theorem 6.4]{Kristel2019} that there is even a canonical, fusion-preserving isomorphism 
\begin{equation}
\label{eq:isoce}
\mathscr{T}(\mathcal{G}_{bas}) \cong \BCE
\end{equation}
to our operator-theoretic model for the basic central extension of \cref{basicce}, equipped with the fusion product $\mu$ of \cref{sec:operatormodel}. 
Next, we apply transgression  to a string structure $(\mathcal{S},\mathcal{A},\sigma)$ with a connection. The bundle gerbe $\mathcal{S}$ transgresses to a principal $\U(1)$-bundle $S\defeq \mathscr{T}(\mathcal{S})$ over $L\Spin(M)$ equipped with a fusion product. Taking the isomorphism \cref{eq:isoce} into account, the bundle gerbe morphism $\mathcal{A}$ transgresses to a fusion-preserving bundle morphism 
\begin{equation*}
\mathscr{T}(\mathcal{A}):\delta^{*}\smash{\widetilde{L\Spin}}(d) \otimes \pr_2^{*}S \to \pr_1^{*}S\text{,}
\end{equation*}
from which one can extract an $\BCE$-action on $S$ turning $S$ into a principal $\BCE$-bundle over $M$. One can then show that $S$ is a spin structure on $LM$, and further, that the fusion product on $S$ turns it into a fusive spin structure, for the details, see \cite{waldorfa}.

Summarizing, a   geometric string structure on $M$ (i.e., a string structure with a string connection) induces in a canonical way a fusive spin structure on $LM$. Thus, our construction of the Connes fusion product on the spinor bundle on loop space, which we describe in the subsequent \cref{sec:SpinorBundleOnLoopSpaceI}, applies, in particular, to spin manifolds equipped with a geometric string structure.

\section{Fusion in the spinor bundle on loop space}

\label{sec:SpinorBundleOnLoopSpaceI}

This section contains our main result: we first exhibit the spinor bundle $\mathcal{F}^{\smooth}(LM)$ on loop space as a rigged von Neumann bimodule bundle, and then equip it with a fusion product with respect to Connes fusion.
Throughout this section, $M$ is a spin manifold equipped with a  spin structure $\SStruct$ on its loop space.
In \cref{sec:Fusion} we will then require a \emph{fusive} spin structure.

\subsection{The von Neumann algebra bundle over path space}

\label{sec:algebrabundlepathspace}

In this section we construct a rigged von Neumann algebra bundle $\mathcal{N}^{\smooth}$ on the path space $PM$ of $M$, which, heuristically, has the property that the algebra $\mathcal{N}^{\smooth}_{\beta_{1}}$ sits in the Clifford algebra $\ClvN(LM)^\smooth_{\beta_{1} \cup \beta_{2}}$ as $\Cl(V_{-})^{\smooth}$ sits in $\Clsm$, for paths $(\beta_1,\beta_2)\in PM^{[2]}$. In order to construct $\mathcal{N}^{\smooth}$, we start by constructing the underlying rigged \cstar-algebra bundle $\mathcal{A}^{\smooth}$.  
\begin{comment}
We recall from \cref{sec:genfusionproducts} that paths have sitting instants, and that the space of paths is considered as a diffeological space, with plots $c:U \to PM$ all maps such that $c^{\vee}: U \times [0,1] \to M$ is smooth. Here, we rather want to work with the Fr\'echet manifold 
$C^{\infty}([0,1],M)$, 
which contains $PM$ as a subset such that the inclusion map $PM \to C^{\infty}([0,1],M)$ is a smooth map between diffeological spaces. 

Indeed, when we regard $C^{\infty}([0,1],M)$ as a diffeological space (with plots all smooth maps $c:U \to C^{\infty}([0,1],M)$), then this diffeology and the functional diffeology (with plot all maps such that $U \times[0,1] \to M$ is smooth) coincide (\cite{waldorf9}).
But $PM$ is is equipped with the subspace diffeology of $C^{\infty}([0,1],M)$ in the functional diffeology.
\end{comment}
We recall from \cref{sec:operatormodel} that we have a Fr\'{e}chet Lie group homomorphism $L\Spin(d) \rightarrow \O^{\theta}_L(V)$, which is induced by the pointwise multiplication of $L\SO(d)$ on $V$.
From \cref{sec:CliffFockAndSpin} we recall that we have a Fr\'{e}chet Lie group homomorphism $\O^{\theta}_L(V) \rightarrow \O(V_{-})$.
Finally, $\O(V_{-})$ acts smoothly on $\Cl(V_{-})^{\smooth}$, and we thus obtain an induced smooth representation  $L\Spin(d) \times \Cl(V_{-})^{\smooth} \rightarrow \Cl(V_{-})^{\smooth}$.
This allows us to define via \cref{lem:cstarbundleass} an associated rigged \cstar-algebra bundle
\begin{equation*}
 \Cl_{-}^{\smooth}(LM) \defeq L\Spin(M) \times_{L\Spin(d)} \Cl^{\smooth}(V_{-})
\end{equation*}
over $LM$ with typical fibre $\Cl^{\smooth}(V_{-})$. Next we consider the diffeological space $PM$ of paths in $M$, together with the doubling map
$\Delta : PM \rightarrow LM, \; \beta \mapsto \beta \cup \beta$, which
is well-defined and smooth (see \cref{sec:genfusionproducts}).
Then, we  define the rigged \cstar-algebra bundle 
\begin{equation*}
\mathcal{A}^{\smooth} \defeq  \Delta^{*}\Cl_{-}^{\smooth}(LM)
\end{equation*}
with typical fibre $\Cl^{\smooth}(V_{-})$ over  $PM$, in the sense of \cref{sec:HilbertBundle:3}.
\begin{comment}
As $PM$ is a diffeological space, by this we mean that $\mathcal{A}^{\smooth}$ is the collection $\mathcal{A}^{\smooth}=(\mathcal{A}^{\smooth}_c)$ of the rigged \cstar-algebra bundles $\mathcal{A}^{\smooth}_c \defeq c^{*}\Delta^{*}\Cl^{\smooth}_{-}(LM)$ over the domain $U$ of each plot $c:U \to PM$, together with isometric \cstar-algebra morphisms for the transition between plots,  see \cref{sec:HilbertBundle:3} for more details.
\end{comment}

\begin{remark}
Analogous to \cref{re:cliffordbundlewithoutspin}, our construction of the rigged \cstar-algebra bundle $\mathcal{A}^{\smooth}$ holds if $M$ is merely oriented, it neither needs a spin structure on $M$ nor on $LM$. The following discussion requires both, however.
\end{remark}

We have proved in \cref{lem:SmoothHalfInclusion} that the inclusion $\iota_{-}: \Cl(V_{-})^{\smooth} \rightarrow \Clsm,a\mapsto a \otimes \id$ is smooth, and since $\iota_{-}$ obviously intertwines the $L\Spin(d)$-actions, we obtain an induced isometric morphism
\begin{equation}
\label{eq:inclusionhalf}
\Cl_{-}^{\smooth}(LM) \rightarrow \Cl^{\smooth}(LM), [\ph, a] \mapsto [\ph, a \otimes \mathds{1}]
\end{equation}
of rigged \cstar-algebra bundles over $LM$. Further, we have considered the induced representation of $\Cl(V_{-})^{\smooth}$ on Fock space $\mathcal{F}^{\smooth}$, turning $\mathcal{F}^{\smooth}$ into a rigged $\Cl(V_{-})^{\smooth}$-module. On the level of rigged module \emph{bundles}, no general induction  procedure exists; however, in this case, the following result shows that it is yet possible.

\begin{proposition}
\label{lem:indrepminus}
\label{re:locatrivN}
The restriction of  Clifford multiplication $\Cl^{\smooth}(LM) \times \mathcal{F}^{\smooth}(LM) \to \mathcal{F}^{\smooth}(LM)$ along \cref{eq:inclusionhalf} equips the spinor bundle $\mathcal{F}^{\smooth}(LM)$ with the structure of a rigged  $\Cl^{\smooth}_{-}(LM)$-module bundle with typical fibre the rigged $\Cl(V_{-})^{\smooth}$-module $\mathcal{F}^{\smooth}$. 
\end{proposition}

\begin{proof}
In order to meet the assumptions of \cref{def:repcstarbundle}, we have to find \emph{compatible} local trivializations $\Psi$ of $\mathcal{F}^{\smooth}(LM)$ and $\Phi$ of $\Cl^{\smooth}_{-}(LM)$. As associated bundles, local trivializations can be induced from local trivializations of the underlying principal bundles (\cref{lem:AssociatedFrechetBundle,lem:cstarbundleass}); here, $\SStruct$ and $L\Spin(M)$. Let $\tilde\varphi: U \rightarrow \widetilde{L\Spin}(M)$ be any local section, and let $\ph : U\rightarrow L\Spin(M)$ be the local section obtained by composing $\tilde\varphi$ with the projection $\SStruct \rightarrow L\Spin(M)$.
The corresponding local trivializations of principal bundles induce the following local trivializations of associated bundles:
\begin{align*}
&\Psi:\mc{F}^{\smooth}(LM)|_{U} \rightarrow \mc{F}^{\smooth} \times U,&  \Psi([\tilde\varphi(x),v]) = (v,x)
 \\  
&\Phi:\Cl^{\smooth}_{-}(LM)|_U \rightarrow \Cl(V_{-})^{\smooth} \times U,&  \Phi([\varphi(x),a]) = (a,x)
\end{align*}
It is obvious that these exchange the rigged module structure on the fibres with the one on the typical fibre, and hence are compatible in the sense of \cref{def:repcstarbundle}.
\begin{comment}
We then see by inspection  that the following diagram commutes
\begin{equation*}
 \xymatrix{
 (\Cl^{\smooth}_{-}(LM) \times_{LM} \mc{F}^{\smooth}(LM))|_{U}\ar[d]_{u_{-}\times \nu} \ar[r] & (\Cl^{\smooth}(LM) \times_{LM} \mc{F}^{\smooth}(LM))|_{U} \ar[d]_{u \times \nu} \ar[r]^(.7){\rho} & \mc{F}^{\smooth}(LM) \ar[d]^{\nu}\\
 \Cl(V_{-})^{\smooth} \times \mc{F}^{\smooth}_{L} \times U \ar[r]_{\iota_{-}\times \mathds{1}\times \mathds{1}} & \Cl(V)^{\smooth} \times \mc{F}^{\smooth}_{L} \times U \ar[r] & \mc{F}^{\smooth}_{L} \times U \text{;}
 }
\end{equation*}
this proves the claim.
\end{comment}
\end{proof}

Again via pullback along the doubling map $\Delta$, we obtain the rigged $\mathcal{A}^{\smooth}$-module bundle  $\Delta^{*}\mathcal{F}^{\smooth}(LM)$, with typical fibre the rigged $\Cl(V_{-})^{\smooth}$-module $\mathcal{F}^{\smooth}$. Since the latter is a rigged von Neumann algebra, $N^{\smooth}=(\Cl(V_{-})^{\smooth},\mc{F}^{\smooth})$, we have  that 
\begin{equation*}
\mathcal{N}^{\smooth} \defeq (\mathcal{A}^{\smooth},\Delta^{*}\mc{F}^{\smooth}(LM))
\end{equation*}
is a \emph{rigged von Neumann algebra bundle} over $PM$ with typical fibre $N^{\smooth}$ in the sense of \cref{def:rvnab}.
This implies that for each $\beta \in PM$ the fibre $\mathcal{N}^{\smooth}_{\beta}=(\mathcal{A}^{\smooth}_{\beta},\mc{F}^{\smooth}(LM)_{\Delta(\beta)})$, is a rigged von Neumann algebra, and thus induces an ordinary von Neumann algebra  $N_{\beta}$, see \cref{re:vnacl}. Any choice of compatible local trivializations as constructed in \cref{lem:indrepminus} establishes  a normal $\ast$-isomorphism $u:\mathcal{N}_{\beta} \to N$ of von Neumann algebras, see \cref{re:loctrivmorph}. In particular, the von Neumann algebras $\mathcal{N}_{\beta}$ are type III$_1$-factors. 

\begin{comment}
\begin{remark}
For later use, we shall extract from \cref{lem:indrepminus}  how to obtain local trivializations of $\mathcal{N}^{\smooth}$, or rather of the plot-wise defined bundle $\mathcal{N}^{\smooth}_c=(\mathcal{A}^{\smooth}_{c},(\Delta^{*}\mathcal{F}^{\smooth}(LM))_c)$, for any plot $c:U \to PM$. Indeed, possibly after shrinking $U$, there exists  a  section $\tilde\varphi:U \to \SStruct$ along the smooth map $\Delta \circ c: U \to LM$, projecting to a section $\varphi:U \to L\Spin(M)$ along $\Delta\circ c$. Then, we have local trivializations $\Phi: \mathcal{A}^{\smooth}_c \to \Cl(V_{-})^{\smooth}\times U$ and $\Psi: (\Delta^{*}\mathcal{F}^{\smooth}(LM))_c \to \mathcal{F}^{\smooth} \times U$ with
\begin{equation*}
\Phi([\varphi(x),a])=(a,x)
\quad\text{,}\quad
\Psi([\tilde\varphi(x),v])=(v,x)\text{.}
\end{equation*}
\end{remark}
\end{comment}

\begin{comment}
The  relation between $\Cl^{\smooth}(LM)$ and $p_1^{*}A\otimes p_2^{*}A^{opp}$
is not clear. First of all, we don't know what the tensor product of rigged \cstar-algebras is (let alone of bundles of those).
 Then, we believe one can take the tensor product of Fr\'{e}chet algebras, but we don't know how $\Cl(V_{-})^{\smooth} \otimes \Cl(V_{+})^{\smooth}$ relates to $\Cl(V)^{\smooth}$, and a discussion on this would take us too far afield.
 Finally, on the level of \cstar-algebras we do know that $\Cl(V_{-}) \otimes \Cl(V_{+}) = \Cl(V)$, which still holds fibrewise for the bundles above, but we think that's not too interesting of a comment...
\end{comment}

\subsection{The spinor bundle as a bundle of bimodules}

\label{sec:bimodule}
\label{sec:BimoduleBundle}

The goal of this section is to exhibit the spinor bundle $\mathcal{F}^{\smooth}(LM)$ as a rigged von Neumann $\mathcal{N}^{\smooth}$-$\mathcal{N}^{\smooth}$-bimodule bundle. 
We start fibrewise, and shall, 
for each loop of the form $\gamma = \beta_{1} \cup \beta_{2}$, where $(\beta_1,\beta_2)\in PM^{[2]}$, equip the Fock spaces $\mc{F}^{\smooth}(LM)_{\gamma}$ with  representations of the rigged \cstar-algebras $\mathcal{A}^{\smooth}_{\beta_1}$ and $(\mathcal{A}^{\smooth}_{\beta_2})^{\opp}$.
We  recall that $\mathcal{F}^{\smooth}$ is a rigged  $\Cl(V_{-})^{\smooth}$-$\Cl(V_{-})^{\smooth}$-bimodule, with  representations denoted by $a_1 \lact v \ract a_2$ (\cref{lem:Friggedcstarbundle}). 
\begin{comment}
We recall that the left action is by inclusion, $a\lact v =(a \otimes \mathds{1} )\lact v $, and the right action is by $v \ract a = J(a^{*}\otimes \mathds{1})J \lact v$. 
\end{comment}

\begin{lemma}
\label{lem:HalfCliffmult}
 For each $(\beta_1,\beta_2) \in PM^{[2]}$ there exist unique maps
 \begin{align*}
  (\rho_{1})_{\beta_1,\beta_2}: \mathcal{A}^{\smooth}_{\beta_1} \times \mc{F}^{\smooth}(LM)_{\gamma} \rightarrow \mc{F}^{\smooth}(LM)_{\gamma} \quad\text{ and }\quad  (\rho_{2})_{\beta_1,\beta_2}: (\mathcal{A}^{\smooth})^{\opp}_{\beta_2} \times \mc{F}^{\smooth}(LM)_{\gamma} \rightarrow \mc{F}^{\smooth}(LM)_{\gamma}\text{,}
 \end{align*}
where $\gamma=\beta_1\cup\beta_2$,
 with the property that
the equations \begin{align*}
  (\rho_{1})_{\beta_1,\beta_2} ([\Delta(\ph_{1}),a],[\tilde\varphi,v]) = [\tilde\varphi, a \lact v]\quad\text{ and }\quad (\rho_{2})_{\beta_1,\beta_2} ([\Delta(\ph_{2}),a],[\tilde\varphi,v]) = [\tilde\varphi, v \ract a]
 \end{align*}
 hold for all $(\ph_1,\ph_2) \in P\Spin(M)^{[2]}$ lifting $(\beta_1,\beta_2)$, all  $\tilde\varphi \in \SStruct$ lifting $\ph_1\cup\ph_2$, all $a \in \Cl(V_{-})^{\smooth}$ and all $v \in \mc{F}^{\smooth}_{L}$. 
\end{lemma}

\begin{comment}
Let's make sure that everything makes sense. We have $\Delta(\varphi_i)\in L\Spin(M)$, so that $[\Delta(\varphi_i),a]\in \Cl_{-}^{\smooth}(LM)$ in the fibre over $\Delta(\beta_i)$. Thus, by definition, $[\Delta(\varphi_i),a]\in \mathcal{A}^{\smooth}_{\beta_i}$. Further, we have $\varphi_1\cup\varphi_2\in L\Spin(M)$.
\end{comment}

\begin{proof}

It is clear that the maps are determined uniquely by the given equations, provided that choices of $(\varphi_1,\varphi_2)$ and $\tilde\varphi$ exist for arbitrary $(\beta_1,\beta_2)$. To see this, we first note that by surjectivity of the bundle projection $L\Spin(M) \to LM$, there exists $\varphi\in L\Spin(M)$ projecting to $\gamma=\beta_1\cup\beta_2$. Since the loop $\gamma$ is constant around $0$ and $\pi$ (due to the sitting instants of the paths) a standard argument shows that $\varphi$ can be chosen with the same
constancy conditions; then, $\varphi_1(t) \defeq \varphi(t)$ and $\varphi_2(t)\defeq\varphi(2\pi-t)$ define the pair $(\varphi_1,\varphi_2)\in P\Spin(M)^{[2]}$ as required. 
\begin{comment}
We first show that the loop $\varphi(u)$ can be assumed to be locally constant at $\pi$. Since $\tau\defeq \beta_1(u) \cup \beta_2(u)$ is locally constant at $\pi$, there exists $\varepsilon>0$ such that $\varphi(u)(t)$ sits in the fibre over $\tau(\pi)$ for all $t\in I_{\varepsilon}$, with $I_{\varepsilon} \defeq (\pi-\varepsilon, \pi+\varepsilon)$. Thus, there exists a smooth map $g:I_{\varepsilon} \to \Spin(d)$ with $\varphi(u)(\pi)=\varphi(u)(t)g(t)$ for all $t\in I_{\varepsilon}$. Now choose a smooth map $\tilde g:S^1 \to \Spin(d)$ with the property that $\tilde g|_{J_{\varepsilon}}= g|_{J_{\varepsilon}}$, for $J_{\varepsilon}\defeq (\pi-\frac{1}{2}\varepsilon,\pi+\frac{1}{2}\varepsilon)$,  and $\tilde g|_{S^1 \setminus I_{\varepsilon}}=1$. Then,  $\varphi(u)\tilde g$ is a smooth map $S^1 \to \Spin(M)$ which is locally constant at $\pi$. This procedure can be performed at $0$ in the same way
\end{comment}
Now, $\varphi_1\cup\varphi_2=\varphi\in L\Spin(M)$, and since $\SStruct\to L\Spin(M)$ is surjective, a lift $\tilde\varphi$ exists.

For existence, we have to check that the given equations can be used as definitions, i.e., that they are independent of the involved choices. To see this, we suppose $(\varphi_1',\varphi_2')\in P\Spin(M)^{[2]}$ lifts the same pair $(\beta_1,\beta_2)$, and $\tilde\varphi'$ lifts $\varphi_1'\cup\varphi_2'$. Then, using the principal actions, we have $\varphi_i'=\varphi_ig_i$ for $g_i\in P\Spin(d)$ and $\tilde\varphi'=\tilde\varphi \tilde g$ for $\tilde g\in \widetilde{L\Spin}(d)$. By definition of $\mathcal{A}^{\smooth}$ and $\mathcal{F}^{\smooth}(LM)$ as associated bundles, we have $[\Delta(\ph_{1}'),a]=[\Delta(\varphi_1),\theta_{\omega_{-}(g_1)}(a)]$ and $[\tilde\varphi',v]=[\tilde\varphi ,Uv]$, with $\omega_{-}:P\Spin(d) \to \O(V_{-})$ defined in \cref{sec:operatormodel}, and $U\defeq \tilde\omega(\tilde g)\in \Imp_L^{\theta}(V)$. 
Thus, we get
\begin{align*}
(\rho_1)_{\beta_1,\beta_2}([\Delta(\ph_{1}'),a],[\tilde\varphi',v]) &= [\tilde\varphi, (\theta_{\omega_{-}(g_1)}(a) \otimes \mathds{1}) \lact Uv]
\\&= [\tilde\varphi, (\theta_{\omega(g_1\cup g_2)_{-}}(a) \otimes \theta_{\omega(g_1\cup g_2)_{+}}(\mathds{1})) \lact Uv]
\\&=[\tilde\varphi ,\theta_{\omega(g_1\cup g_2)}(a \otimes \mathds{1}) \lact Uv]
\\&=[\tilde\varphi ,U ((a \otimes \mathds{1})\lact v)]
\\&=[\tilde\varphi',a\lact v]\text{,}
\end{align*}
as intended.
In the second step we have used the commutativity of diagram \cref{eq:diagramomega}, together with the fact that Bogoliubov automorphisms are unital.
The third step  uses \cref{lem:Bogolsplit}, and the fourth step uses the fact that $U$ implements $\omega(g_1\cup g_2)$.
For $\rho_2$, we  have $[\Delta(\ph_{2}'),a] = [\Delta(\ph_{2}), \theta_{\omega_{-}(g_{1})}(a)]$, and compute
 \begin{align*}
(\rho_2)_{\beta_1,\beta_2}([\Delta(\ph_{2}'),a],[\tilde\varphi',v]) &= [\tilde\varphi, J (\theta_{\omega_{-}(g_{2})}(a) \otimes \mathds{1})^{*}J \lact Uv]
\\&=[\tilde\varphi ,J \theta_{\omega(g_2\cup g_1)}(a \otimes \mathds{1})^{*}J \lact Uv]
\\&=[\tilde\varphi , \theta_{\omega(g_1\cup g_2)}(J(a \otimes \mathds{1})^{*}J) \lact Uv]
\\&=[\tilde\varphi ,U (J(a \otimes \mathds{1})^{*}J\lact v)]
\\&=[\tilde\varphi', v \ract a]\text{,}
\end{align*}
where, in the third step, we have used Equation \cref{eq:JCommutesWithTheta} from \cite[Lemma 4.8]{Kristel2019} together with the obvious identity $\tau(\omega(g_{1}\cup g_{2})) = \omega(g_{2} \cup g_{1})$.
\end{proof}

It is clear from the formulas in  \Cref{lem:HalfCliffmult} that the maps  $(\rho_1)_{\beta_1,\beta_2}$ and $(\rho_2)_{\beta_1,\beta_2}$ are indeed left actions of $\mathcal{A}^{\smooth}_{\beta_1}$  and $(\mathcal{A}^{\smooth}_{\beta_2})^{\opp}$, respectively, and that they commute.
We remark that it is further true (this follows later from \cref{prop:spinorbundlecstarbimodule}) that $\mathcal{F}^{\smooth}(LM)_{\beta_1\cup\beta_2}$ is a rigged $\mathcal{A}^{\smooth}_{\beta_1}$-$\mathcal{A}^{\smooth}_{\beta_2}$-bimodule in the sense of \cref{def:rbim}.

Trying to assemble the maps $(\rho_1)_{\beta_1,\beta_2}$ and $(\rho_2)_{\beta_1,\beta_2}$ into bundle morphisms, we face the problem that the bundle $\mc{F}^{\smooth}(LM)$ lives over $LM$, while the algebra bundle $\mathcal{A}^{\smooth}$ lives over $PM$. Therefore, we work over the  space $PM^{[2]}$ of pairs of paths with common end points, and work with the pullbacks of bundles to this space, along the maps
\begin{equation*}
\xymatrix{
 LM & & PM^{[2]} \ar[ll]_{\cup} \ar@<-.7ex>[rr]_{p_{2}} \ar@<.7ex>[rr]^{p_{1}} & & PM\text{.}}
\end{equation*}
We recall from \cref{sec:HilbertBundle:3} that over the
diffeological space $PM^{[2]}$, bundles consist of plot-wise defined bundles, and bundle isomorphisms for the transition of plots. Let $c:U \to PM^{[2]}$ be a plot. We write $c_{i} \defeq \Delta \circ p_i \circ c$ for $i=1,2$ and $\tilde c\defeq \cup \circ c$. By definition of the involved diffeologies, the maps $c_1,c_2,\tilde c: U \to LM$ are smooth maps between (Fr\'echet) manifolds. The  rigged \cstar-algebra bundles $p_1^{*}\mathcal{A}^{\smooth}$ and $p_2^{*}\mathcal{A}^{\smooth}$, consist of the plot-wise defined rigged \cstar-algebra bundles $(p_1^{*}\mathcal{A}^{\smooth})_c=c_1^{*}\Cl^{\smooth}_-(LM)$ and $(p_2^{*}\mathcal{A}^{\smooth})_{c} = c_2^{*}\Cl^{\smooth}_{-}(LM)$ over $U$, respectively, and the rigged Hilbert space  bundle   $\cup^{*}\mathcal{F}^{\smooth}(LM)$ consists of the plot-wise defined rigged Hilbert space bundles  $(\cup^{*}\mathcal{F}^{\smooth}(LM))_c = \tilde c^{*}\mathcal{F}^{\smooth}(LM)$ over $U$. 
Moreover, if $c':U' \to PM^{[2]}$ is another  plot,  and $f: U \to U'$  is a smooth map such that $c' \circ f=c$, then the bundle isomorphisms $(p_i^{*}\mathcal{A}^{\smooth})_c \to f^{*}(p_i^{*}\mathcal{A}^{\smooth})_{c'}$ are the canonical ones obtained from the equality $c_i=c'_i \circ f$. Similarly, the bundle isomorphism $(\cup^{*}\mathcal{F}^{\smooth}(LM))_c \to f^{*}(\cup^{*}\mathcal{F}^{\smooth}(LM))_{c'}$ is the canonical one obtained from the equality $\tilde c= \tilde c' \circ f$.

Our goal is to show that the maps of
 \Cref{lem:HalfCliffmult} equip  the rigged Hilbert space bundle   $\cup^{*}\mathcal{F}^{\smooth}(LM))$ over $PM^{[2]}$ with the structure of a rigged $p_1^{*}\mathcal{A}^{\smooth}$-$p_2^{*}\mathcal{A}^{\smooth}$-bimodule. We will first discuss the situation on a fixed plot $c:U \to PM^{[2]}$, and to this end, consider $x\in U$ and $(\beta_1,\beta_2) \defeq c(x)$. Then, for the fibres over $x$ we find 
\begin{equation*}
((p_1^{*}\mathcal{A}^{\smooth})_{c})_x=c_1^{*}\Cl^{\smooth}_-(LM)_x=\Cl^{\smooth}_{-}(LM)_{\Delta(\beta_1)}=\mathcal{A}^{\smooth}_{\beta_1}
\end{equation*}
and similarly,
\begin{equation*}
((p_2^{*}\mathcal{A}^{\smooth})_c)_x=\mathcal{A}^{\smooth}_{\beta_2}
\quad\text{ and }\quad
(\cup^{*}\mathcal{F}^{\smooth}(LM)_c)_x= \mathcal{F}^{\smooth}(LM)_{\beta_1\cup\beta_2}\text{.}
\end{equation*}
Now we see that the maps   $(\rho_1)_{\beta_1,\beta_2}$ and $(\rho_2)_{\beta_1,\beta_2}$ of \Cref{lem:HalfCliffmult} assemble into fibre-preserving maps
\begin{align*}
\rho_{1,c}&:(p_1^{*}\mathcal{A}^{\smooth})_{c} \times_U \cup^{*}\mathcal{F}^{\smooth}(LM)_c \to \cup^{*}\mathcal{F}^{\smooth}(LM)_c,
\\
\rho_{2,c}&:(p_2^{*}(\mathcal{A}^{\smooth})^{\opp})_{c} \times_U \cup^{*}\mathcal{F}^{\smooth}(LM)_c \to \cup^{*}\mathcal{F}^{\smooth}(LM)_c\text{.} 
\end{align*}

\begin{lemma}
\label{lem:spirbmpw}
Let $c:U \to PM^{[2]}$ be a plot.
The maps $\rho_{1,c}$ and $\rho_{2,c}$  equip the rigged Hilbert space bundle   $(\cup^{*}\mathcal{F}^{\smooth}(LM))_c$ over $U$ with the structure of a rigged $(p_1^{*}\mathcal{A}^{\smooth})_{c}$-$(p_2^{*}\mathcal{A}^{\smooth})_{c}$-bimodule bundle with typical fibre  the rigged $\Cl(V_{-})^{\smooth}$-$\Cl(V_{-})^{\smooth}$-bimodule $\mc{F}^{\smooth}$. 
\end{lemma}

For the proof of \cref{lem:spirbmpw} we require the following  lemma, which delivers us appropriate local trivializations that at each point meet the conditions of \cref{lem:HalfCliffmult}.

\begin{lemma}
\label{lem:comptriv}
Each $x\in U$ has an open neighborhood $W \subset U$ admitting  smooth maps $\varphi_1,\varphi_2: W \to P\Spin(M)$ and a smooth map $\tilde\varphi:W \to \widetilde{L\Spin}(M)$, such that $(\varphi_1(u),\varphi_2(u))\in P\Spin(M)^{[2]}$ for all $u\in W$, and  the diagram
\begin{equation*}
\xymatrix@C=4em{ && \widetilde{L\Spin}(M) \ar[d] \\  & P\Spin(M)^{[2]} \ar[d] \ar[r]_-{\cup} & L\Spin(M) \ar[d] \\ W \ar@/^2.4pc/[uurr]^-{\tilde\varphi}  \ar@/^0.6pc/[ru]_-{(\varphi_1,\varphi_2)} \ar[r]_-{c} & PM^{[2]} \ar[r]_-{\cup} & LM}
\end{equation*}
is commutative.
\end{lemma}

\begin{proof}
The map $c:U \to LM$ is smooth, and since $L\Spin(M) \to M$ is a locally trivial bundle, there exists $W$ on which $c$ lifts to a smooth map $\varphi: W \to L\Spin(M)$. We claim that -- probably by going to a smaller open subset -- $\varphi$ can be chosen such that $\varphi(u)\in L\Spin(M)$ is locally constant at $0$ and at $\pi$, for all $u\in W$. Then, $\varphi_1(u)(t)\defeq\varphi(u)(t)$ and $\varphi_2(u)(t) \defeq \varphi(u)(2\pi-t)$ define smooth maps $\varphi_1,\varphi_2:W \to P\Spin(M)$ as required. We have already seen in the proof of \cref{lem:HalfCliffmult} that the claim is true at each point $u\in W$ separately. Doing this simultaneously over $W$ may require to restrict to a compact neighborhood, where the interval on which $\varphi(u)$ is constant achieves a minimum length, and then to restrict further to an open subset thereof.    
\begin{comment}
It can be performed smoothly over $W$ by first going to a compact neighborhood of $W$, where the various $\varepsilon$ achieve a minimum, and then by continuing with an open subset contained in the compact one. 
\end{comment}
Now, the map $W \to L\Spin(M):u \mapsto \varphi_1(u)\cup\varphi_2(u)$ is smooth, and since $\widetilde{L\Spin}(M) \to L\Spin(M)$ is a locally trivial bundle, a section $\tilde\varphi$ exists, possibly after a further shrinking of $W$.
\end{proof}

An immediate consequence of the definitions of the bundles $(p_i^{*}\mathcal{A}^{\smooth})_c$ and $(\cup^{*}\mathcal{F}^{\smooth}(LM))_c$ as (pullbacks of) associated bundles, the sections into the corresponding principal bundles obtained from \cref{lem:comptriv} induce   local trivializations 
 \begin{align}
\nonumber
  &\Phi^{1}: (p_{1}^{*}\mathcal{A}^{\smooth})_c|_W \rightarrow \Cl(V_{-})^{\smooth} \times W &&\Phi^{1}( [\Delta(\ph_{1}(x)),a] )= (a,x), \\\label{lem:loctrivcomp}
    &\Phi^{2}: (p_{2}^{*}\mathcal{A}^{\smooth})_c|_W \rightarrow \Cl(V_{-})^{\smooth} \times W && \Phi^{2} ([\Delta(\ph_{2}(x)),a] )= (a,x),\\\nonumber
\qquad &\Psi: (\cup^{*}\mc{F}^{\smooth}(LM))_c|_W \rightarrow \mc{F}^{\smooth} \times W && \Psi( [\tilde\varphi(x),v]) = (v,x)\text{,}
\end{align}
see \cref{lem:AssociatedFrechetBundle,lem:cstarbundleass}.

\begin{proof}[Proof of \cref{lem:spirbmpw}]
According to \cref{def:cstarbimodulebundle}, the statement can be proved locally in a neighborhood of any point $x\in U$. We may assume an open neighborhood $V \subset U$ that admits smooth maps $\varphi_1,\varphi_2: W \to P\Spin(M)$ and $\tilde\varphi:W \to \widetilde{L\Spin}(M)$ satisfying the conditions  in \cref{lem:comptriv}. Then, we let $\Phi_1$, $\Phi_2$, and $\Psi$ be the local trivializations of \cref{lem:loctrivcomp}. Inspecting  the definition of $\rho_{1,c}$ and $\rho_{2,c}$ in \cref{lem:HalfCliffmult}, 
we find that
\begin{equation*}
\Psi(\rho_{1,c}(a_1,v)) = \Phi^1(a_1) \lact \Psi(v)
\quad\text{ and }\quad
\Psi(\rho_{2,c}(a_2,v)) = \Psi(v)\ract \Phi^2(a_2)
\end{equation*}
for all appropriate $a_1\in (p_{1}^{*}\mathcal{A}^{\smooth})_c$, $a_2\in (p_{2}^{*}\mathcal{A}^{\smooth})_c$ and $v\in (\cup^{*} \mathcal{F}^{\smooth}(LM))_c$. 
\begin{comment}
In other words, the diagram
 \begin{equation*}
  \xymatrix@C=5em{
  (p_1^{*}\mathcal{A}^{\smooth})_c \times_{U} (\cup^{*} \mathcal{F}^{\smooth}(LM))_c \ar[r]^-{\rho_{1,c}} \ar[d]_{\Phi_1 \times \Psi} & (\cup^{*} \mathcal{F}^{\smooth}(LM))_c \ar[d]^{\nu} & \ar[l]_-{\rho_{2,c}} (\cup^{*} \mathcal{F}^{\smooth}(LM))_c \times_{U} (p_2^{*}\mathcal{A}^{\smooth})_c \ar[d]^{\Psi \times \Phi_{2}} \\
  (\Cl(V_{-})^{\smooth} \times U) \times_U (\mc{F}^{\smooth}_{L} \times U) \ar[d] \ar[r] & \mc{F}^{\smooth}_{L} \times U \ar[d] & \ar[l] (\mc{F}^{\smooth}_{L} \times U) \times_U (\Cl(V_{-})^{\smooth} \times U) \ar[d]
\\ \Cl(V_{-})^{\smooth} \times \mathcal{F}^{\smooth}_L \ar[r] & \mathcal{F}^{\smooth}_L & \mathcal{F}^{\smooth}_L \times \Cl(V_{-})^{\smooth} \ar[l]  }
 \end{equation*}
which at the bottom has the $\Cl(V_{-})^{\smooth}$-$\Cl(V_{-})^{\smooth}$-bimodule structure of $\mathcal{F}^{\smooth}_L$, is commutative.
\end{comment}
This proves that $(\Phi^1,\Psi)$ and $(\Phi^2,\Psi)$ are compatible local trivializations. 
\end{proof} 

\cref{lem:spirbmpw} establishes the plot-wise definition of the rigged $p_1^{*}\mathcal{A}^{\smooth}$-$p_2^{*}\mathcal{A}^{\smooth}$-bimodule bundle  $\cup^{*}\mathcal{F}^{\smooth}(LM)$; now it remains to assure the correct gluing behaviour for the transition between plots. Thus, suppose $c:U \to PM^{[2]}$ and $c':U' \to PM^{[2]}$ are plots, and $f:U\to U'$ is a smooth map with $c' \circ f=c$. But the canonical bundle isomorphisms  $(p_i^{*}\mathcal{A}^{\smooth})_c \to f^{*}(p_i^{*}\mathcal{A}^{\smooth})_{c'}$ and $(\cup^{*}\mathcal{F}^{\smooth}(LM))_c \to f^{*}(\cup^{*}\mathcal{F}^{\smooth}(LM))_{c'}$ are under the fibre-wise identifications used in the definition of $\rho_{1,c}$ and $\rho_{2,c}$ over each point $x\in U$ the identity maps of $\Cl^{\smooth}_{-}(LM)_{\beta_i \cup\beta_i}$ and $\mathcal{F}^{\smooth}(LM)_{\beta_1\cup\beta_2}$, respectively, where $(\beta_1,\beta_2)\defeq  c(x)$. In particular, they are unitary intertwiner. 
Thus, we have the following result.

\begin{proposition}
\label{prop:spinorbundlecstarbimodule}
The plot-wise  bimodule structure of \cref{lem:spirbmpw} exhibits the spinor bundle $\cup^{*}\mathcal{F}^{\smooth}(LM)$ as a rigged $p_1^{*}\mathcal{A}^{\smooth}$-$p_2^{*}\mathcal{A}^{\smooth}$-bimodule bundle, with typical fibre the rigged $\Cl(V_{-})^{\smooth}$-$\Cl(V_{-})^{\smooth}$-bimodule $\mc{F}^{\smooth}$.
\end{proposition}

The bundle $\mathcal{A}^{\smooth}$ is a rigged \cstar-algebra bundle; however, in \cref{sec:algebrabundlepathspace} we discussed how to upgrade it to a rigged von Neumann algebra bundle $\mathcal{N}^{\smooth}=(\mathcal{A}^{\smooth}, \Delta^{*} \mc{F}^{\smooth}(LM))$ with typical fibre $N^{\smooth}$. Moreover, the typical fibre $\mathcal{F}^{\smooth}$ of the rigged bimodule bundle $\cup^{*}\mathcal{F}^{\smooth}(LM)$ is a rigged von Neumann $N^{\smooth}$-$N^{\smooth}$-bimodule. Our main result in this section is that these rigged von Neumann structures carry over to the spinor bundle.

\begin{theorem}
\label{lem:vonNeumannBimodule}
The spinor bundle on loop space $\cup^{*}\mc{F}^{\smooth}(LM)$ is a rigged von Neumann  $p_1^{*}\mathcal{N}^{\smooth}$-$p_2^{*}\mathcal{N}^{\smooth}$-bimodule bundle with typical fibre the rigged von Neumann  $N^{\smooth}$-$N^{\smooth}$-bimodule $\mc{F}^{\smooth}$.
\end{theorem}

For the proof, we note the following improvement of \cref{lem:comptriv}.  

\begin{lemma}
\label{lem:comptriv2}
Let $c:U \to PM^{[2]}$ be a plot. Then, each point $x\in U$ has an open neighborhood $W \subset U$ admitting smooth maps $\varphi_1,\varphi_2:W \to P\Spin(M)$ and $\tilde\varphi_1,\tilde\varphi_2,\tilde\varphi: W \to \SStruct$ such that $\varphi_1$, $\varphi_2$, and $\tilde \varphi$ satisfy the conditions of \cref{lem:comptriv}, and $\tilde\varphi_i$ lifts $\Delta \circ \varphi_i:W \to L\Spin(M)$, for $i=1,2$.
\end{lemma}

\begin{proof}
The additional maps $\tilde\varphi_i$ exist, possibly after a further shrinking of $W$, since $\SStruct \to L\Spin(M)$ is a surjective submersion.
\end{proof}

\begin{proof}[Proof of \cref{lem:vonNeumannBimodule}]
Using \cref{lem:comptriv2} We obtain the local trivializations $\Phi^1$ of $(p_{1}^{*}\mathcal{A}^{\smooth})_c$, $\Phi^2$ of $(p_{2}^{*}\mathcal{A}^{\smooth})_c$, and $\Psi$ of $(\cup^{*}\mc{F}^{\smooth}(LM))_c$ of \cref{lem:loctrivcomp}, and additionally from $\tilde\varphi_i$, since $\Delta\circ \varphi_i$ is a section along
\begin{equation*}
\xymatrix{U \ar[r]^-{c} & PM^{[2]} \ar[r]^-{p_i} & PM \ar[r]^{\Delta} & LM\text{,}}
\end{equation*}
 local trivializations 
\begin{equation*} 
 \Psi^i:(p_i^{*}\Delta^{*}\mathcal{F}^{\smooth}(LM))_c|_W \to \mathcal{F}^{\smooth} \times W,\quad\quad \Psi^i([\tilde\varphi_i(x),v])=(v,x)\text{.}
\end{equation*}
By inspection of the involved formulae, we see that $(\Phi^{i},\Psi^{i})$ is the pullback of the compatible local trivializations of \cref{re:locatrivN}, and hence is a compatible local trivialization of $(p_i^{*}\mathcal{N}^{\smooth})_c$. Consulting \cref{def:vonNeumannBimoduleBundle}, this proves that  $(\cup^{*}\mathcal{F}^{\smooth}(LM))_c$ is  a rigged von Neumann $(p_1^{*}\mathcal{N}^{\smooth})_c$-$(p_2^{*}\mathcal{N}^{\smooth})_c$-bimodule bundle. Concerning the transition between plots, there is nothing to add to the argument given in the proof of \cref{prop:spinorbundlecstarbimodule}. \end{proof}

The following result follows from \cref{lem:vonNeumannBimodule} via the theory of rigged von Neumann algebra bundles developed in \cref{sec:riggedvonNeumann}. 
First, for each $(\beta_1,\beta_2)\in PM^{[2]}$ we have by \cref{re:fibresofvonneumannbimodule} that the fibre $\mathcal{F}^{\smooth}(LM)_{\beta_1\cup\beta_2}$ is a rigged $\mathcal{N}^{\smooth}_{\beta_1}$-$\mathcal{N}^{\smooth}_{\beta_2}$-bimodule. \cref{rem:vonNeumanBimoduleCompletion} implies then the following result.

\begin{corollary}
The  completion  $\mathcal{F}(LM)_{\beta_1\cup\beta_2}$ of each fibre of the spinor bundle on loop space is an $\mathcal{N}_{\beta_1}$-$\mathcal{N}_{\beta_2}$-bimodule in the classical von Neumann-theoretical sense. 
\end{corollary}

In the next section it will be important to identify the bimodule $\mathcal{F}(LM)_{\beta_1\cup\beta_2}$ with the typical fibre bimodule $\mathcal{F}$ in a precise way, using the methods developed in the results above, reduced to a single point. For later reference, we summarize this in the following remark.

\begin{remark}
\label{cor:BimoduleFrames}
For a pair $(\beta_1,\beta_2)\in PM^{[2]}$ of paths with common endpoints, consider a lift  $(\varphi_1,\varphi_2)\in P\Spin(M)^{[2]}$ and a lift $\tilde\varphi\in \SStruct$ of $\varphi_1\cup\varphi_2$. Then, there exist unique  isomorphisms $\phi^i: \mathcal{N}^{\smooth}_{\beta_i} \to N^{\smooth}$ of rigged von Neumann algebras and a unique isometric isomorphism $\psi: \mathcal{F}^{\smooth}(LM)_{\beta_1\cup\beta_2}\to \mathcal{F}^{\smooth}$ of rigged Hilbert spaces with $\phi^i([\varphi_i,a])=a$ and $\psi([\tilde\varphi,v])=v$.
Moreover, $(\phi^1,\phi^{2},\psi)$ is an invertible unitary intertwiner from the rigged von Neumann $p_1^{*}\mathcal{N}^{\smooth}$-$p_2^{*}\mathcal{N}^{\smooth}$-bimodule $\mathcal{F}^{\smooth}(LM)_{\beta_1\cup\beta_2}$ to the rigged von Neumann $N^{\smooth}$-$N^{\smooth}$-bimodule $\mathcal{F}^{\smooth}$ in the sense of \cref{def:riggedvonNeumannbimodule}. The isomorphisms $\phi^i$ induce normal $\ast$-isomorphisms $u_i: N_{\beta_i}\to N$ of  von Neumann algebras, and by \cref{rem:vonNeumanBimoduleCompletion}, the isometric isomorphism $\psi$ induces a unitary map $\nu: \mathcal{F}^{\smooth}(LM)_{\beta_1\cup\beta_2} \to \mathcal{F}$, such that the triple $(u_1,u_2,\nu)$ is an intertwiner between the $\mathcal{N}_{\beta_1}$-$\mathcal{N}_{\beta_2}$-bimodule $\mathcal{F}(LM)_{\beta_1\cup\beta_2}$ and the $N$-$N$-bimodule $\mathcal{F}$. 
\end{remark}

\subsection{Fusion of spinors}

\label{sec:Fusion}

The  goal of this section is to construct the Connes fusion of spinors on loop space. Now we require $\SStruct$ to be a fusive spin structure on $LM$ (see \cref{def:FusiveSpinStructure}). Its fusion product is the essential ingredient to our construction, and will be denoted by $\lambda$.
The main  result is the following.

\begin{theorem}
\label{th:fusionfibrewise}
\label{rem:conditionfusion}
Let $(\beta_1,\beta_2,\beta_3)\in PM^{[3]}$ be a triple of paths with common endpoints, let $\mathcal{N}_{\beta_i}$ be the von Neumann algebras over the paths $\beta_i$, and let $\mathcal{F}(LM)_{\beta_i\cup\beta_j}$ be the von Neumann $\mathcal{N}_{\beta_i}$-$\mathcal{N}_{\beta_j}$-bimodules over the loops $\beta_i\cup\beta_j$.
Then, there exists a unique unitary intertwiner 
\begin{equation*}
\chi_{\beta_1,\beta_2,\beta_3}: \mathcal{F}(LM)_{\beta_1\cup\beta_2} \boxtimes_{\mathcal{N}_{\beta_2}} \mathcal{F}(LM)_{\beta_2\cup\beta_3} \to \mathcal{F}(LM)_{\beta_1\cup\beta_3}
\end{equation*}
of $\mathcal{N}_{\beta_1}$-$\mathcal{N}_{\beta_3}$-bimodules, where $\boxtimes$ is Connes fusion, such that the following condition is satisfied:

\noindent
For all $(\varphi_1,\varphi_2,\varphi_3)\in P\Spin(M)^{[3]}$ such that $\varphi_i$ lifts $\beta_i$,   and all $\tilde\varphi_{12},\tilde\varphi_{23},\tilde\varphi_{13}\in\widetilde{L\Spin}(M)$ such that $\tilde\varphi_{ij}$ lifts $\varphi_i\cup\varphi_j$ and such that $\lambda(\tilde\varphi_{12},\tilde\varphi_{23})=\tilde\varphi_{13}$, the diagram
\begin{equation}
\label{diag:FusionProductOnF}
\xymatrix@C=6em{\mathcal{F}(LM)_{\beta_1\cup\beta_2} \boxtimes_{\mathcal{N}_{\beta_2}} \mathcal{F}(LM)_{\beta_2\cup\beta_3}\ar[d]_{(u_1,u_3,\nu_{12} \boxtimes \nu_{23})} \ar[r]^-{\chi_{\beta_1,\beta_2,\beta_3}} & \mathcal{F}(LM)_{\beta_1\cup\beta_3} \ar[d]^{(u_1,u_3,\nu_{13})} \\ \mathcal{F} \boxtimes_{N} \mathcal{F} \ar[r]_-{\chi} & \mathcal{F}}
\end{equation} 
is commutative, where $(u_i,u_j,\nu_{ij})$ are the unitary intertwiners determined by $\varphi_i$, $\varphi_j$ and $\tilde\varphi_{ij}$ as in \cref{cor:BimoduleFrames}, and $\chi$ is the Connes fusion of the free fermions of \cref{sec:FockSpaceAsStandardForm}.
\end{theorem}

\begin{proof}
For uniqueness, we have to prove that choices of $\varphi_i$ and $\tilde\varphi_{ij}$ exist; then,  diagram \cref{diag:FusionProductOnF} fixes $\chi_{\beta_1,\beta_2,\beta_3}$ uniquely as the maps $\nu_{ij}$ are bijective. We may use \cref{lem:HalfCliffmult}  separately for $(\beta_1,\beta_2)$ and then $(\beta_2,\beta_3)$ to obtain choices of $(\varphi_1,\varphi_2)\in P\Spin(M)^{[2]}$ lifting $(\beta_1,\beta_2)$, and $(\varphi_2',\varphi_3')\in P\Spin(M)^{[2]}$ lifting $(\beta_2,\beta_3)$. Since $\Spin(M)$ is a principal $\Spin(d)$-bundle, there exists a unique $g\in P\Spin(d)$ such that $\varphi_2=\varphi_2' g$. Then, $(\varphi_2' g,\varphi_3 'g)\in P\Spin(M)^{[2]}$ also lifts $(\beta_2,\beta_3)$, and with $\varphi_3 \defeq \varphi_3' g$ we have $(\varphi_1,\varphi_2,\varphi_3)\in P\Spin(M)^{[3]}$. Now, choices of $\tilde\varphi_{12}$ and $\tilde\varphi_{23}$ obviously exist, since $\widetilde{L\Spin}(M) \to L\Spin(M)$ is surjective, and we set $\tilde\varphi_{13}\defeq \lambda(\tilde\varphi_{12},\tilde\varphi_{23})$. This finishes the proof of existence of choices. 

Next we prove that diagram \cref{diag:FusionProductOnF} can be used to define $\chi_{\beta_1,\beta_2,\beta_3}$; to this end, we have to show that all choices of $\varphi_i$ and $\tilde\varphi_{ij}$ lead to the same definition. Indeed, suppose $(\varphi_1',\varphi_2',\varphi_3')\in P\Spin(M)^{[3]}$ such that $\varphi_i$ lifts $\beta_i$,  suppose  $\tilde\varphi_{ij}' \in \SStruct$ lift $\varphi_i'\cup\varphi_j'$ and $\tilde\varphi_{13}' = \lambda(\tilde\varphi_{12}',\tilde\varphi_{23}')$. Then, there are unique elements $\tilde g_{12},\tilde g_{23} \in \BCE$ such that $\tilde\varphi_{12}' = \tilde\varphi_{12} \cdot \tilde g_{12}$ and $\tilde\varphi_{23}' = \tilde\varphi_{23} \cdot \tilde g_{23}$. The compatibility between the fusion products $\mu$ on $\BCE$ and $\lambda$ on $\SStruct$ in
\Cref{def:FusiveSpinStructure} then implies that $\tilde\varphi_{13}' = \tilde\varphi_{13} \cdot \mu(\tilde g_{12} \otimes \tilde g_{23})$.
We write $(u_i',u_j',\nu_{ij}')$ for the unitary intertwiners  corresponding to $\tilde\varphi_{ij}'$ and $\varphi_i'$ according to  \cref{cor:BimoduleFrames}.
Using the definition of $\nu_{ij}$ and $\nu_{ij}'$ via $\tilde\varphi_{ij}$ and $\tilde\varphi_{ij}'$, respectively, we obtain $\nu_{12}' = U_{12} \circ \nu_{12}$ and $\nu_{23}' = U_{23} \circ \nu_{23}$, for the implementers $U_{12}\defeq  \tilde\omega(\tilde g_{12})$  and $U_{23}\defeq  \tilde\omega(\tilde g_{23})$ in $\Imp_L^{\theta}(V)$.
\begin{comment}
By \cref{cor:BimoduleFrames} the isometric isomorphisms $\psi,\psi': \mathcal{F}^{\smooth}(LM)_{\beta_i\cup\beta_j} \to \mathcal{F}^{\smooth}$ are given by $\psi([\tilde\varphi_{ij},v])=v$ and $\psi'([\tilde\varphi'_{ij},v])=v$. Thus,
\begin{equation*}
\psi'([\tilde\varphi_{ij},v])=\psi'([\tilde\varphi_{ij}'\cdot \tilde g_{ij}^{-1},v])=\psi'([\tilde\varphi_{ij}',\tilde\omega(\tilde g_{ij})v])=\tilde\omega(\tilde g_{ij})v=\tilde\omega(\tilde g_{ij})\psi([\tilde\varphi_{ij},v])\text{.}
\end{equation*}
\end{comment}
By \cref{thm:FusionProductsAgree}, we have then $\nu_{13}' = \hat{\mu}(U_{12},U_{23}) \circ \nu_{13}$, where $\hat\mu$ denotes the Connes fusion of implementers defined in \cref{sec:FusionImpThroughFock}. 
Using properties and the definition of $\hat\mu$, and the functoriality of Connes fusion, we compute
\begin{align*}
 (\nu_{13}')^{*} \circ \chi \circ (\nu_{12}' \boxtimes \nu_{23}') &= \nu_{13}^{*} \circ \hat{\mu}(U_{12}^{*},U_{23}^{*}) \circ \chi \circ (U_{12} \boxtimes U_{23}) \circ ( \nu_{12} \boxtimes \nu_{23}) \\
 &= \nu_{13}^{*} \circ \chi \circ (U_{12}^{*}\boxtimes U_{23}^{*}) \circ (U_{12} \boxtimes U_{23}) \circ ( \nu_{12} \boxtimes \nu_{23}) \\
 &= \nu_{13}^{*} \circ \chi \circ (\nu_{12} \boxtimes \nu_{23}).
\end{align*}
This proves that one can define $\chi_{\beta_1,\beta_2,\beta_3} \defeq \nu_{13}^{*} \circ \chi \circ (\nu_{12} \boxtimes \nu_{23})$ for arbitrary choices of $\varphi_i$ and $\tilde\varphi_{ij}$. 
It remains to notice that this definition indeed yields a unitary intertwiner of $\mathcal{N}_{\beta_1}$-$\mathcal{N}_{\beta_3}$-bimodules, as all  three components are unitary intertwiners. 
\end{proof}

The collection $\chi=(\chi_{\beta_1,\beta_2,\beta_3})$ with $(\beta_1,\beta_2,\beta_3)$ ranging over $PM^{[3]}$ will be called the \emph{Connes fusion product} on the spinor bundle  the loop space. Its construction proves the conjectured \quot{Theorem 1} of \cite{stolz3}. In the remainder of this section we derive  three fundamental  consequences and properties.

\begin{proposition}
For a path $\beta\in PM$, the   $\mathcal{N}_{\beta}$-$\mathcal{N}_{\beta}$-bimodule $\mathcal{F}(LM)_{\beta\cup\beta}$ is neutral with respect to Connes fusion. Moreover, if $(\beta_1,\beta_2)\in PM^{[2]}$,  the bimodules $\mathcal{F}(LM)_{\beta_1\cup\beta_2}$ and $\mathcal{F}(LM)_{\beta_2\cup\beta_1}$ are inverses of each other with respect to Connes fusion. In particular, $\mathcal{F}(LM)_{\beta_1\cup\beta_2}$ is a canonical Morita equivalence between the von Neumann algebras $\mathcal{N}_{\beta_1}$ and $\mathcal{N}_{\beta_2}$.
\end{proposition}

\begin{proof}

The $\mathcal{N}_{\beta}$-$\mathcal{N}_{\beta}$-bimodule $\mathcal{F}(LM)_{\beta\cup\beta}$ is a standard form of $\mathcal{N}_{\beta}$, since it is isomorphic via a unitary intertwiner  to the $N$-$N$-bimodule $\mathcal{F}$, which is a standard form of $N$ (\cref{cor:BimoduleFrames,prop:standardform}). In particular, $\mathcal{F}(LM)_{\beta\cup\beta}$ is neutral with respect to Connes fusion. If $(\beta_1,\beta_2)\in PM^{[2]}$, then evaluating the Connes fusion product over the triples  $(\beta_1,\beta_2,\beta_1)$ and $(\beta_2,\beta_1,\beta_2)$ exhibits the bimodules $\mathcal{F}(LM)_{\beta_1\cup\beta_2}$ and $\mathcal{F}(LM)_{\beta_2\cup\beta_1}$ as inverses of each other.  
\begin{comment}
Indeed, we have unitary intertwiners
  \begin{align*}
  \chi_{\beta_1,\beta_{2},\beta_{1}}: \mathcal{F}(LM)_{\beta_1\cup\beta_{2}} \boxtimes_{\mathcal{N}_{\beta_2}} \mathcal{F}(LM)_{\beta_{2}\cup\beta_{1}} &\to \mathcal{F}(LM)_{\beta_1\cup\beta_{1}}, \\
  \chi_{\beta_{2},\beta_{1},\beta_{2}}: \mathcal{F}(LM)_{\beta_{2}\cup\beta_{1}} \boxtimes_{\mathcal{N}_{\beta_1}} \mathcal{F}(LM)_{\beta_{1}\cup\beta_{2}} &\to \mathcal{F}(LM)_{\beta_{2}\cup\beta_{2}}.
 \end{align*}
\end{comment}
\end{proof}

For the second result, we recall that Connes fusion is coherently associative (\cref{lem:ConnesFusionAssociative}), which allows us to omit bracketing of multiple Connes fusions.

\begin{proposition}
\label{th:fusionass}
The Connes fusion product on the spinor bundle  is associative in the sense that the diagram
 \begin{equation*}
 \begin{gathered}
 \begin{tikzpicture}[scale=1.6]
    \node (A) at (0,1) {$\mathcal{F}(LM)_{\beta_1\cup\beta_2} \boxtimes_{\mathcal{N}_{\beta_2}} \mathcal{F}(LM)_{\beta_2\cup\beta_3} \boxtimes_{\mathcal{N}_{\beta_3}} \mathcal{F}(LM)_{\beta_3\cup\beta_4}$};
    \node (B) at (5.5,1) {$\mathcal{F}(LM)_{\beta_1\cup\beta_2} \boxtimes_{\mathcal{N}_{\beta_2}} \mathcal{F}(LM)_{\beta_2\cup\beta_4}$};
    \node (C) at (0,0) {$\mathcal{F}(LM)_{\beta_1\cup\beta_3} \boxtimes_{\mathcal{N}_{\beta_3}} \mathcal{F}(LM)_{\beta_3\cup\beta_4}$};
    \node (D) at (5.5,0) {$\mathcal{F}(LM)_{\beta_1\cup\beta_4}$};
    \path[->,font=\scriptsize]
    (A) edge node[above]{$\mathds{1} \boxtimes \chi_{\beta_2,\beta_3,\beta_4}$} (B)
    (A) edge node[left]{$\chi_{\beta_1,\beta_2,\beta_3} \boxtimes \mathds{1}$} (C)
    (B) edge node[right]{$\chi_{\beta_1,\beta_2,\beta_4}$} (D)
    (C) edge node[below]{$\chi_{\beta_1,\beta_3,\beta_4}$} (D);
 \end{tikzpicture}
 \end{gathered}
 \end{equation*}
 is commutative for all $(\beta_{1},\beta_{2},\beta_{3},\beta_{4}) \in PM^{[4]}$.
 \end{proposition}

\begin{proof}
Let $(\beta_{1},\beta_{2},\beta_{3},\beta_{4}) \in PM^{[4]}$.
Iterating the procedure described at the beginning of the proof of \cref{th:fusionfibrewise}, one can see that there exists $(\varphi_1,\varphi_2,\varphi_3,\varphi_4)\in P\Spin(M)^{[4]}$ such that $\varphi_i$ lifts $\beta_i$. Choose $\tilde\varphi_{ij} \in \SStruct$ lifting $\varphi_i\cup \varphi_j$, for $ij  \in \{12,23,34 \}$. Then, we set $\tilde\varphi_{13} \defeq \lambda(\tilde\varphi_{12} \otimes \tilde\varphi_{23})$, $\tilde\varphi_{24} \defeq \lambda(\tilde\varphi_{23} \otimes \tilde\varphi_{34})$, and $\tilde\varphi_{14} \defeq \lambda(\tilde\varphi_{13} \otimes \tilde\varphi_{34})$. Notice that by the associativity of the fusion product $\lambda$ (\cref{def:fusion product}), we have that $\tilde\varphi_{14} = \lambda(\tilde\varphi_{12} \otimes \tilde\varphi_{24})$.
We let $(u_i,u_j,\nu_{ij})$ be the unitary intertwiners induced by $\varphi_i$ and $\tilde\varphi_{ij}$ according to \cref{cor:BimoduleFrames}. We consider the following diagram of unitary isomorphisms: 
\begin{equation*}
 \begin{gathered}
 \begin{tikzpicture}[scale=1.6]
    \node (A) at (0,3) {$\mathcal{F}(LM)_{\beta_1\cup\beta_2} \boxtimes_{\mathcal{N}_{\beta_2}} \mathcal{F}(LM)_{\beta_2\cup\beta_3} \boxtimes_{\mathcal{N}_{\beta_3}} \mathcal{F}(LM)_{\beta_3\cup\beta_4}$};
    \node (A2) at (1.6,2) {$\mathcal{F} \boxtimes_{N} \mathcal{F} \boxtimes_{N} \mathcal{F}$};
    \node (B) at (5.5,3) {$\mathcal{F}(LM)_{\beta_1\cup\beta_2} \boxtimes_{\mathcal{N}_{\beta_2}} \mathcal{F}(LM)_{\beta_2\cup\beta_4}$};
    \node (B2) at (4.5,2) {$\mathcal{F} \boxtimes_{N} \mathcal{F}$};
    \node (C) at (0,0) {$\mathcal{F}(LM)_{\beta_1\cup\beta_3} \boxtimes_{\mathcal{N}_{\beta_3}} \mathcal{F}(LM)_{\beta_3\cup\beta_4}$};
    \node (C2) at (1.6,1) {$\mathcal{F} \boxtimes_{N} \mathcal{F}$};
    \node (D) at (5.5,0) {$\mathcal{F}(LM)_{\beta_1\cup\beta_4}$};
    \node (D2) at (4.5,1) {$\mathcal{F}$};
    \path[->,font=\scriptsize]
    (A) edge node[right]{$\nu_{12}\boxtimes\nu_{23}\boxtimes \nu_{34}$} (A2)
    (B) edge node[left]{$\nu_{12}\boxtimes\nu_{24}$} (B2)
    (C) edge node[right]{$\nu_{13}\boxtimes\nu_{34}$} (C2)
    (D) edge node[left]{$\nu_{14}$} (D2)
    (A) edge node[above]{$\mathds{1} \boxtimes \chi_{\beta_2,\beta_3,\beta_4}$} (B)
    (A2) edge node[above]{$\mathds{1} \boxtimes \chi$} (B2)
    (A) edge node[left]{$\chi_{\beta_1,\beta_2,\beta_3} \boxtimes \mathds{1}$} (C)
    (A2) edge node[left]{$\chi \boxtimes \mathds{1}$} (C2)
    (B) edge node[right]{$\chi_{\beta_1,\beta_2,\beta_4}$} (D)
    (B2) edge node[right]{$\chi$} (D2)
    (C) edge node[below]{$\chi_{\beta_1,\beta_3,\beta_4}$} (D)
    (C2) edge node[below]{$\chi$} (D2);
 \end{tikzpicture}
 \end{gathered}
 \end{equation*}
 The four outer diagrams are commutative by  \cref{th:fusionfibrewise}, using above relations between fusion products of the various $\tilde\varphi_{ij}$.
The ones on top and on the left additionally using the functoriality of Connes fusion (\cref{lem:ConnesFusionOfMaps}). The diagram in the middle is commutative by \cref{lem:FockFusionAssociative}; this completes the proof.
\end{proof}

The third result concerns the smoothness of  the Connes fusion product. Since we have not yet been able to lift Connes fusion to the setting of rigged Hilbert spaces, we cannot  claim that the fibre-wise maps $\chi_{\beta_1,\beta_2,\beta_3}$ assemble into a smooth (or only continuous) bundle homomorphism. Instead we will describe smoothness by showing that certain smooth sections are mapped to smooth sections.

We consider the (set-theoretical) bundle $p_{12}^{*}\mathcal{F}(LM) \boxtimes p_{23}^{*}\mathcal{F}(LM)$
over $PM^{[3]}$ , whose fibre over $(\beta_1,\beta_2,\beta_3)\in PM^{[3]}$ is the Hilbert space $\mathcal{F}(LM)_{\beta_1\cup\beta_2} \boxtimes_{\mathcal{N}_{\beta_2}} \mathcal{F}(LM)_{\beta_2\cup\beta_3}$. The Connes fusion product on the spinor bundle assembles then into a (set-theoretical) bundle morphism
\begin{equation*}
\chi: p_{12}^{*}\mathcal{F}(LM) \boxtimes p_{23}^{*}\mathcal{F}(LM) \to p_{23}^{*}\mathcal{F}(LM)
\end{equation*}
over $PM^{[3]}$. We will now specify a natural choice of smooth sections into $p_{12}^{*}\mathcal{F}(LM) \boxtimes p_{23}^{*}\mathcal{F}(LM)$. To start with, we recall that a smooth section $\tilde\varphi:U \to \SStruct$ and a smooth map $v:U \to \mathcal{F}^{\smooth}$ induce a smooth section $\sigma_{\tilde\varphi,v}:U \to \mathcal{F}^{\smooth}(LM)$ into the spinor bundle, with $\sigma_{\tilde\varphi,v}(x)=[\tilde\varphi(x),v(x)]$. Since $\mathcal{F}^{\smooth}\subset \mathcal{F}$ is dense,  the image of all such sections is dense in each fibre.
The following lemma guarantees that similar local sections exist in a situation appropriate for fusion.

\begin{lemma}
\label{lem:smoothlocalsectionsforfusion}
Let $(\beta_1,\beta_2,\beta_3):U \to PM^{[3]}$ be a plot. Then, for each $x\in U$ there exists an open neighborhood $W \subset U$ of $x$,  smooth maps $\varphi_1,\varphi_2,\varphi_3: W \to P\Spin(M)$, and smooth maps $\tilde\varphi_{12},\tilde\varphi_{23},\tilde\varphi_{13}: W \to \widetilde{L\Spin}(M)$ such that $(\varphi_1(x),\varphi_2(x),\varphi_3(x))\in P\Spin(M)^{[3]}$ for all $x\in W$,  $\varphi_i$ lifts $\beta_i$ for all $i\in\{1,2,3\}$,  $\tilde\varphi_{ij}$ lifts $\varphi_i\cup\varphi_j$ for all $ij\in\{12,23,13\}$, and $\tilde\varphi_{13}(x)=\lambda(\tilde\varphi_{12}(x),\tilde\varphi_{23}(x))$ for all $x\in W$. \end{lemma}

\begin{proof}
As in the proof of \cref{th:fusionfibrewise}, we apply \cref{lem:comptriv} separately to the plots $(\beta_1,\beta_2)$ and $(\beta_2,\beta_3)$, obtaining smooth maps $(\varphi_1,\varphi_2)$ and $(\varphi_2',\varphi_3')$. The difference between $\varphi_2$ and $\varphi_2'$ is now a smooth map $g:W \to P\Spin(d)$ with $\varphi_2(x)=\varphi_2'(x)g(x)$ for all $x\in W$. With $\varphi_3(x)\defeq \varphi_3'(x)g(x)$ we obtain the desired triple $(\varphi_1,\varphi_2,\varphi_3):W \to P\Spin(M)^{[3]}$. As explained in the proof of \cref{lem:comptriv}, there exist $\tilde\varphi_{12}$ and $\tilde\varphi_{13}$ as claimed. Finally, since the fusion product $\lambda$ on $\SStruct$ is a smooth in the sense of \cref{def:fusion product}, the pointwise definition $\tilde\varphi_{13}(x)\defeq\lambda(\tilde\varphi_{12}(x),\tilde\varphi_{23}(x))$ yields a smooth map.
\end{proof}

In the situation of \cref{lem:smoothlocalsectionsforfusion},
at each point $x\in W$,  the elements $\tilde\varphi_{ij}(x)$ and $\varphi_i(x)$ induce unitary intertwiners $(u^1_x,u^2_x,\nu^{12}_x)$  and $(u^2_x,u^3_x,\nu^{23}_x)$ as in \cref{cor:BimoduleFrames}.
 For any smooth map $v:W \to \mathcal{F}^{\smooth}$, we then obtain a  section $\sigma$ into the bundle  $p_{12}^{*}\mathcal{F}(LM) \boxtimes p_{23}^{*}\mathcal{F}(LM)$ by setting 
\begin{equation}
\label{eq:localsmoothsection}
\sigma(x) \defeq (\nu^{12}_x \boxtimes \nu^{23}_x)^{*}\chi^{*}(v(x))\in (p_{12}^{*}\mathcal{F}(LM) \boxtimes p_{23}^{*}\mathcal{F}(LM))_{c(x)}
\end{equation}
for all $x\in W$.
We will call any  section that is locally of this form \emph{smooth}. Again, since each fibre of the bundle  $p_{12}^{*}\mathcal{F}(LM) \boxtimes p_{23}^{*}\mathcal{F}(LM)$ is isomorphic to $\mathcal{F}$, the image of all smooth sections is dense. We have the following result. 

\begin{proposition}
\label{prop:smoothness2}
The Connes fusion product $\chi$ on $\mathcal{F}(LM)$ is smooth in the sense that it sends smooth sections to smooth sections. 
\end{proposition}

\begin{proof}
Let $\sigma$ be a smooth section. Smoothness of $\chi \circ \sigma$ can be checked locally on an open set $W$, on which sections as in  \cref{lem:smoothlocalsectionsforfusion} and a smooth map $v:W \to \mathcal{F}^{\smooth}$ exist such that \cref{eq:localsmoothsection} holds. The conditions in \cref{lem:smoothlocalsectionsforfusion} imply that at each $x\in W$,  diagram \cref{diag:FusionProductOnF} in \cref{rem:conditionfusion} is commutative, saying that 
\begin{equation*}
\sigma(x)=(\nu^{13}_x)^{-1}(v(x))\in (p_{13}^{*}\mathcal{F}(LM))_{c(x)}\text{.}
\end{equation*}
Thus, $\sigma|_W$ is (the completion of) the smooth section $\sigma_{\tilde\varphi_{13},v}:W \to p_{13}^{*}\mathcal{F}^{\smooth}(LM)$, and hence smooth.
\begin{comment}
Using the lemma, we obtain the result that, for each plot $(\beta_1,\beta_2,\beta_3):U \to PM^{[3]}$ and each $x\in U$, there exists an open neighborhood $W \subset U$ of $x$, smooth maps $\varphi_1,\varphi_2,\varphi_3:W \to P\Spin(M)$ and $\tilde\varphi_{12},\tilde\varphi_{23},\tilde\varphi_{13}:W \to \SStruct$ as in the lemma, such that the induced local trivializations $\nu_{ij}(u)$ make the diagram
\begin{equation*}
\xymatrix@C=8em{\mathcal{F}(LM)_{\beta_1(u)\cup\beta_2(u)} \boxtimes_{\mathcal{N}_{\beta_2(u)}} \mathcal{F}(LM)_{\beta_2(u)\cup\beta_3(u)}\ar[d]_{\nu_{12}(u) \boxtimes \nu_{23}(u)} \ar[r]^-{\chi_{\beta_1(u),\beta_2(u),\beta_3(u)}} & \mathcal{F}(LM)_{\beta_1(u)\cup\beta_3(u)} \ar[d]^{\nu_{13}(u)} \\ \mathcal{F}_L \boxtimes_{N} \mathcal{F}_L \ar[r]_-{\chi} & \mathcal{F}_L}
\end{equation*} 
commutative for all $u\in W$. 
\end{comment}
\end{proof}

\subsection{The stringor bundle}
 
 \label{sec:stringorbundle}
 
 Given the spinor bundle $\mathcal{F}^{\smooth}(LM)$ on the loop space $LM$ together with its Connes fusion product, it is possible to assemble a structure that deserves to be called the \emph{stringor bundle} of the string manifold $M$ (this terminology is due to Stolz and Teichner \cite{stolz3}). While Stolz and Teichner understood the stringor bundle  as the pair of the spinor bundle on loop space and its Connes fusion product (both not yet constructed at that time), our aim is to exhibit the stringor bundle as a higher structure on the manifold $M$, rather than on its loop space.

Suitable for our stringor bundle  is the framework  of 2-vector bundles. Roughly speaking, a \emph{2-vector bundle} is meant to be a bundle version of a \emph{2-vector space}, and the bicategory of 2-vector spaces is by definition the bicategory of algebras, bimodules, and intertwiners \cite{ST04,Schreiber2009}. 2-vector bundles will be discussed in a comprehensive way in \cite{Kristel2021a,Kristel2021b}.
Let us first take all algebras and modules to be finite-dimensional and over $\C$. In this setting, 
a 2-vector bundle over a smooth manifold $M$ consists of a surjective submersion $\pi:Y \to M$, for instance, the disjoint union of open sets of a cover, and:
\begin{enumerate}[(a)]

\item 
an algebra bundle $\mathcal{A}$ over $Y$

\item
an invertible  $p_1^{*}\mathcal{A}$-$p_2^{*}\mathcal{A}$-bimodule bundle $\mathcal{M}$ over the double fibre product $Y^{[2]}$

\item
an invertible intertwiner $\mu: p_{12}^{*}\mathcal{M} \otimes_{p_2^{*}\mathcal{A}} p_{23}^{*}\mathcal{M} \to p_{13}^{*}\mathcal{M}$ of $p_1^{*}\mathcal{A}$-$p_3^{*}\mathcal{A}$-bimodule bundles over $Y^{[3]}$

\end{enumerate}
such that $\mu$ satisfies an associativity condition over $Y^{[4]}$.

An example of a 2-vector bundle is a line bundle gerbe; there, the algebra bundle $\mathcal{A}$ is trivial with fibre $\C$, and the invertible bimodule bundle $\mathcal{M}$ is just a complex line bundle. A result in bundle gerbe theory is
that (every) bundle gerbe can be obtained, via a procedure called \emph{regression}, from a line bundle $\mathcal{L}$ on loop space and a fusion product $\lambda$ in the sense of \cref{def:fusion product}, see \cite{waldorf10,waldorf11}, at least if one allows $Y$ to be a diffeological space. Namely, one chooses a base point $x\in M$, and considers $Y=P_xM$, the diffeological space of smooth paths with sitting instants starting at $x$, with $\pi\defeq ev_1$ the end-point evaluation. Then, we have the smooth map $\cup: P_xM^{[2]} \to LM$ and obtain the line bundle $\mathcal{M} \defeq  \cup^{*}\mathcal{L}$. The intertwiner $\mu$ is the restriction of the fusion product $\lambda$ to $P_xM^{[3]} \subset PM^{[3]}$.

We extend  regression to the spinor bundle and its Connes fusion product. Let $M$ be a string manifold with a fixed string structure and a fixed base point $x\in M$. Then,  the  \emph{stringor 2-vector bundle} on $M$ consists of the following structure:
\begin{enumerate}[(a)]

\item 
The diffeological space $P_xM$ together with the end-point evaluation $ev_1: P_xM \to M$.

\item
The rigged von Neumann algebra bundle $\mathcal{N}^{\smooth}$ over $P_xM$, obtained as the restriction of the rigged von Neumann algebra bundle constructed in \cref{sec:algebrabundlepathspace} to $P_xM \subset PM$. 

\item
The rigged von Neumann $p_1^{*}\mathcal{N}^{\smooth}$-$p_2^{*}\mathcal{N}^{\smooth}$-bimodule bundle $\cup^{*}\mathcal{F}^{\smooth}(LM)$ over $P_xM^{[2]}$, constructed in \cref{lem:vonNeumannBimodule} and restricted to $P_xM^{[2]} \subset PM^{[2]}$. 

\item
The Connes fusion product 
\begin{equation*}
\chi: p_{12}^{*}\cup^{*}\mathcal{F}(LM) \boxtimes_{p_2^{*}\mathcal{N}} p_{23}^{*}\cup^{*}\mathcal{F}(LM) \to p_{13}^{*}\cup^{*}\mathcal{F}(LM)
\end{equation*}
over $P_xM^{[3]}$, constructed fibrewise in \cref{th:fusionfibrewise}, which is associative (\cref{th:fusionass}).

\end{enumerate}
Unfortunately, we only have an incomplete understanding of a version of 2-vector bundles that could be a home for our  stringor bundle. There exists a bicategory of von Neumann algebras, bimodules, and intertwiners, with the composition by Connes fusion \cite{Brouwer2003,ST04}. However, as explained in \cref{sec:HilbertBundle}, a proper discussion of a bundle version requires \emph{rigged} von Neumann algebras and bimodules. So far we have not been able to lift Connes fusion to the setting of rigged von Neumann bimodules. Thus, we currently do not have a properly defined bicategory of rigged von Neumann algebra bundles, bimodule bundles, and intertwiners, and hence cannot establish the stringor bundle as an object in a corresponding bicategory of \quot{rigged von Neumann 2-vector bundles}.

\begin{comment}
 The problem appears  in (d) of the above definition of the stringor bundle,  where the domain of $\chi$ is not a rigged von Neumann bimodule bundle; thus, we cannot claim that $\chi$ is a morphism of rigged von Neumann bimodules.
\end{comment}

Yet, our stringor bundle  realizes another claim of Stolz and Teichner \cite[Cor. 5.0.4]{ST04}, namely that a string structure on $M$ gives rise to a family of von Neumann algebras parameterized by the points of $M$. Namely, for  $y\in M$,  we obtain from our stringor bundle a family $(\mathcal{N}_{\beta})$ of von Neumann algebras, indexed by the set $P_{x,y}M$ of paths connecting the base point $x$ with $y$. Moreover, these von Neumann algebras are pairwise related by a canonical invertible $\mathcal{N}_{\beta_1}$-$\mathcal{N}_{\beta_2}$-bimodule $\mathcal{F}(LM)_{\beta_1\cup\beta_2}$, in a way compatible with triples and quadruples of indices. Such a structure is as good as a single von Neumann algebra; thus, our stringor bundle is a family of von Neumann algebras parameterized by the points of $M$, associated to a string structure.

On the other hand, it is clear that the described definition of the stringor bundle is not quite optimal. For example, it depends on the choice of a base point $x$, which is typical disadvantage of regression. Then, its construction depends rather indirectly on the string structure, and the string 2-group makes no direct appearance. We expect that there will be a more direct construction in the future that will avoid these problems. But even if such a construction is found, we believe that the results of the present article will remain useful, for instance to construct differential operators acting on spinors on loop spaces, whereas no theory of  operators acting on sections of 2-vector bundles is available.

\appendix

\section{Bimodules of von Neumann algebras}\label{sec:VNABimodules}

\subsection{Standard forms of a von Neumann algebra}

\label{sec:StandardForms}

In  this appendix we recall some facts about standard forms of von Neumann algebras. 
We first recall the notion of  modules of von Neumann algebras.
Let $\mc{A}_{1}$ and $\mc{A}_{2}$ be von Neumann algebras.
A \emph{left $\mc{A}_{1}$-module} is a Hilbert space $H$ equipped with a normal (i.e.~$\sigma$-weakly continuous) $*$-homomorphism $\mc{A}_{1} \rightarrow \mc{B}(H)$.
\begin{comment}
 The $\sigma$-weak (also known as the ultraweak) topology on $\mc{B}(H)$ is the topology generated by the semi-norms $T \mapsto |\mathrm{tr}(\rho T)|$ for $\rho$, where $\rho$ ranges over the trace-class operators on $H$.
It is stronger (i.e., finer, i.e., larger) than the weak operator topology.
Thus, if $\mathcal{A} \to \mathcal{B}(\mathcal{H})$ is $\sigma$-weakly continuous, then it is also weakly continuous; the condition of being $\sigma$-weakly continuous is strictly stronger than being weakly continuous. 

Every $*$-homomorphism from a $\sigma$-finite von Neumann algebra into the bounded operators on a separable Hilbert space is automatically normal.
Takesaki I, Proposition 3.19 on page 78 tells us the following:
If $A$ is a von Neumann algebra on $H$, and if $H$ admits a countable separating subset (this holds in particular if $H$ admits a separating vector), then $A$ is $\sigma$-finite.
\end{comment}

We adopt the notation $a \lact v$ for the element $a \in \mc{A}_{1}$ acting on the vector $v \in H$ from the left.
A \emph{right $\mc{A}_{2}$-module} is a Hilbert space $H$ equipped with a normal $*$-homomorphism $\mc{A}_{2}^{\text{opp}} \rightarrow \mc{B}(H)$, (where $\mc{A}^{\mathrm{opp}}_{2}$ is the opposite algebra of $\mc{A}_{2}$).
We adopt the notation $v \ract a$ for the right action of $a \in \mc{A}_{2}$ on $v \in H$.
An \emph{$\mc{A}_{1}$-$\mc{A}_{2}$-bimodule} is a Hilbert space $H$ which is a left $\mc{A}_{1}$-module and at the same time a right $\mc{A}_{2}$-module, such that the left and right actions commute
\begin{comment}
I.e.~$(a_{1} \lact v) \ract a_{2} = a_{1} \lact (v \ract a_{2})$ for all $a_{1} \in \mc{A}_{1}$, all $a_{2} \in \mc{A}_{2}$, and all $v \in H$. 
\end{comment}
If $H$ is an $\mathcal{A}_1$-$\mathcal{A}_2$-bimodule, and $\tilde H$ is an $\widetilde{\mathcal{A}}_1$-$\widetilde{\mathcal{A}}_2$-bimodule, then an \emph{intertwiner} from $H$ to $\tilde H$ is a triple $(f_1,f_2,T)$ consisting of normal $*$-homomorphisms $f_1: \mathcal{A}_1\to\widetilde{\mathcal{A}}_1$ and $f_2: \mathcal{A}_2\to\widetilde{\mathcal{A}}_2$, and of a bounded linear operator $T: H \to \tilde H$ that intertwines both actions along $f_1$ and $f_2$, i.e.,
\begin{equation*}
T(a_1\lact v \ract a_2) = f_1(a_1) \lact T(v) \ract f_2(a_2)
\end{equation*}
for all $a_1\in \mathcal{A}_1$, $a_2\in \mathcal{A}_2$, and $v\in H$. If $f_1$ and $f_2$ are identities, we just say that $T$ is an intertwiner of $\mathcal{A}_1$-$\mathcal{A}_2$-bimodules. 

Next we recall the notion of a standard form, see \cite[Chapter IX, Definition 1.13]{Ta03} or \cite{Ha75}.
\begin{definition}\label{def:GeneralStandardForm}
        A \emph{standard form} of a von Neumann algebra $\mc{A}$ is a quadruple $(\mc{A}, H, J, P )$, where $H$ is a left $\mc{A}$-module, $J$ is an anti-linear  isometry with $J^{2} = 1$, and $P$ is a closed self-dual cone in $H$, subject to the following conditions:
        \begin{enumerate}
                \item $J\mc{A}J = \mc{A}'$
                \item $JaJ = a^{*}$ for all $a \in \mc{A} \cap \mc{A}'$
                \item $J v = v$ for all $v \in P$
                \item $aJaJP \subseteq P$ for all $a \in \mc{A}$
        \end{enumerate}
\end{definition}

The following result, proved in \cite[Theorem 2.3]{Ha75}, tells us that standard forms are unique up to unique isomorphism.
\begin{theorem}\label{thm:StandardFormIsUnique}
        Suppose that $(\mc{A}_{1}, H_{1}, J_{1}, P_{1} )$ and $(\mc{A}_{2}, H_{2}, J_{2}, P_{2} )$ are standard forms, and that $\pi$ is an isomorphism of $\mc{A}_{1}$ onto $\mc{A}_{2}$. Then, there exists a unique unitary operator $u$ from $H_{1}$ onto $H_{2}$ such that
        \begin{enumerate}
                \item $\pi(a) = uau^{*}$ for all $a \in \mc{A}_{1}$
                \item $J_{2} = u J_{1}u^{*}$
                \item $P_{2} = u P_{1}$
        \end{enumerate}
\end{theorem}

Let $H$ be a left $\mc{A}$-module.
A vector $\xi \in H$ is called \emph{cyclic} if $\mc{A} \lact \xi$ is dense in $H$, and it is called \emph{separating} if the map $\mc{A} \rightarrow H, a \mapsto a \lact \xi$ is injective.
A vector $\xi \in H$ is separating for $\mc{A}$ if and only if it is cyclic for the commutant $\mc{A}'$ of $\mathcal{A}$ in $\mathcal{B}(H)$. 
If a cyclic and separating vector $\xi \in H$ is given, then one can equip $H$ with the structure of a standard form of $\mc{A}$, a fact that we make use of in \cref{sec:FockSpaceAsStandardForm}.
We give the main points of the construction here.

First, we consider the densely defined Tomita operator $S: H \rightarrow H$. It is defined to be the closure of the operator
\begin{equation*}
 \mc{A} \lact \xi \rightarrow \mc{A} \lact \xi, \quad a \lact \xi \mapsto a^{*} \lact \xi.
\end{equation*}
We then write $S = J \Delta^{1/2}$ for the polar decomposition of $S$, where $J$ is an anti-unitary map called the \emph{modular conjugation} and $\Delta^{1/2}$ is a positive unbounded map called the \emph{modular operator}.
\begin{comment}
We remark that these operators depend on the choice of the cyclic and separating vector $\xi$, and 
where confusion may arise we will write $S_{\xi}$, $J_{\xi}$ and $\Delta^{1/2}_{\xi}$.
\end{comment}
A fundamental result of Tomita-Takesaki theory is then the following, proven in, for example \cite[Chapter IX]{Ta03}.
\begin{theorem}\label{th:tota}
 The assignment $a \mapsto Ja^{*}J$ is an anti-isomorphism of von Neumann algebras from $\mc{A}$ onto its commutant $\mc{A}'$.
\end{theorem}
It is proved in \cite[Lemma 3]{Ara74} that if $a \in \mc{A} \cap \mc{A}'$, then $JaJ = a^{*}$.
We define $P\subset H$ to be the closure of $\{J a J a \lact \xi \in H \mid a \in \mc{A} \}$. Then $P$ is a closed self-dual cone in $H$.
It is proved in \cite[Theorem 4]{Ara74} that $Jv = v$ for all $v \in P$ and that $aJaJP \subseteq P$ for all $a \in \mc{A}$. In conclusion, we have the following result.

\begin{proposition}\label{lem:TTstandardform}
 The quadruple $(\mc{A}, H, J, P)$ is a standard form of $\mc{A}$.
\end{proposition}

\begin{remark}
\label{re:bimodulefromstandardform}
We remark that any standard form $(\mc{A},H,J,P)$ of a von Neumann algebra $\mc{A}$ can be equipped with the structure of an $\mc{A}$-$\mc{A}$-bimodule by defining the right action as
\begin{align*}
 H \otimes \mc{A} \rightarrow H, \quad
 v \otimes a \mapsto Ja^{*}J \lact v;
\end{align*}
\cref{th:tota} readily shows that left and right action commute.
We shall then write $v \ract a$ for the right action of $a$ on $v$. Further, in \cref{thm:StandardFormIsUnique}, the triple $(\pi,\pi,u)$ is automatically a unitary  intertwiner of bimodules.
\end{remark}

It is well-known that, through the GNS construction, any normal state on a von Neumann algebra $\mc{A}$ produces a representation of $\mc{A}$ which has a cyclic vector.
If the state is faithful, then it has a cyclic and separating vector.
This construction will be useful later, so we review the main steps now.
Let $\phi: \mc{A} \rightarrow \C$ be a faithful and normal state on a von Neumann algebra $\mc{A}$. Then the assignment $\mc{A} \times \mc{A} \rightarrow \C, (a,b) \mapsto \phi(b^{*}a)$ is a non-degenerate sesquilinear form on $\mc{A}$.
We write $L^{2}_{\phi}(\mc{A})$ for the completion of $\mc{A}$ with respect to this sesquilinear form. The algebra $\mc{A}$ acts from the left on $L^{2}_{\phi}(\mc{A})$ by (the extension of) left multiplication. Clearly the identity $\mathds{1} \in \mc{A} \subseteq L^{2}_{\phi}(\mc{A})$ is a cyclic and separating vector for this left action. It follows that the quadruple $(\mc{A}, L^{2}_{\phi}(\mc{A}), J, P )$ is a standard form of $\mc{A}$.

\begin{lemma}\label{lem:PutInStandardForm}
    Let $H$ be a left $\mc{A}$-module, and let $\xi \in H$ be a cyclic and separating vector.
    Let $\phi: \mc{A}\rightarrow \C$ be the faithful and normal state $\phi(a) = \langle a \lact \xi, \xi \rangle$.
    The unitary map $u: H \rightarrow L^{2}_{\phi}(\mc{A})$ from \cref{thm:StandardFormIsUnique} with respect to $\pi=1$ is given by the closure of the map
     \begin{equation*}
      u: a \lact \xi \mapsto a.
     \end{equation*}
\end{lemma}

\begin{proof}
 A straightforward verification shows that  $u$ satisfies properties 1-3 from \cref{thm:StandardFormIsUnique}.
\end{proof}

\begin{comment}
\begin{proof}
        The fact that $f$ intertwines the inner products, (i.e.~is unitary) is almost tautological. The fact that $f$ intertwines the left actions on the nose is obvious. Let us prove that $f$ intertwines the right actions on the nose. First, we claim that $f$ intertwines the Tomita operators, that is, for all $a \in \mc{A}_{2}$ we have
        \begin{equation*}
        f(S a \lact \xi) = f(a^{*} \lact \xi) = a^{*} = S_{\phi}(a) = S_{\phi}(f(a \lact \xi)).
        \end{equation*}
        In other words $S = f^{*} S_{\phi} f$. Because of the uniqueness of the polar decomposition, this implies that $J = f^{*}J_{\phi}f$.
%        Now let $b \in \mc{A}_{2}$ and let $a \lact \xi \in \mc{A}_{2} \lact \xi$ be arbitrary and let us compute
%        \begin{align*}
%        f(a \lact \xi \ract b) &= a \lact f(\xi \ract b) \\ 
%        &= a \lact f(J b^{*} J \lact \xi) \\
%        &= a \lact J_{\phi} f(b^{*} \lact \xi) \\
%        &= a \lact J_{\phi} b^{*} J_{\phi} \lact \mathds{1} \\
%        &= f(a \lact \xi) \ract b. 
%        \end{align*}
        Let $JaJa \xi \in P_{\xi} \subset H$, then we have
        \begin{equation*}
         uJaJa \xi = JaJa \mathds{1},
        \end{equation*}
        whence $uP_{\xi} \subseteq P_{\mathds{1}} \subset L_{\phi}^{2}(\mc{A})$. Similarly we have $u^{*}JaJa \mathds{1} = JaJa \xi$, whence $u^{*}P_{\mathds{1}} \subseteq P_{\xi}$, which completes the proof.
\end{proof}
\end{comment}

\subsection{Connes fusion of bimodules}\label{sec:ConnesFusion}

In this appendix we review Connes fusion, and provide results about functoriality and associativity specifically adapted to our situation.
Let $\mc{A}_{1},\mc{A}_{2}$, and $\mc{A}_{3}$ be von Neumann algebras.
Let $H$ be an $\mc{A}_{1}$-$\mc{A}_{2}$-bimodule and let $K$ be an $\mc{A}_{2}$-$\mc{A}_{3}$-bimodule.
Our goal is to explain the definition of the Connes fusion product of $H$ with $K$.
In short, the Connes fusion of $H$ with $K$ is an $\mc{A}_{1}$-$\mc{A}_{3}$-bimodule $H \boxtimes_{\phi} K$, which is defined with respect to some faithful and normal state $\phi: \mc{A}_{2} \rightarrow \C$.
The result will be, up to unique isomorphism, independent of the state $\phi$.
Our main references are \cite{thom11} and \cite[Chapter IX]{Ta03}.
The articles \cite{Brouwer2003} and \cite{Bartels2014} provide results similar to ours concerning functoriality and associativity of Connes fusion.

We consider the standard form $L^{2}_{\phi}(\mc{A}_{2})$ of $\mc{A}_{2}$ as an $\mathcal{A}_2$-$\mathcal{A}_2$-bimodule, as explained in \ref{sec:StandardForms}.
We write $\mc{D}(H,\phi) \defeq \Hom_{-,\mc{A}_{2}}(L^{2}_{\phi}(\mc{A}_{2}),H)$ for the space of bounded right module maps from $L^{2}_{\phi}(\mc{A}_{2})$ into $H$. This space is canonically an $\mc{A}_{1}$-$\mc{A}_{2}$-bimodule; explicitly, if $x \in \mc{D}(H,\phi)$, $v \in L^{2}_{\phi}(\mc{A}_{2})$, $a_{1} \in \mc{A}_{1}$ and $a_{2} \in \mc{A}_{2}$, then we have
\begin{align*}
(a_{1} \lact x)(v) \defeq a_{1} \lact x(v) \quad\text{ and }\quad
(x \ract a_{2})(v) \defeq x(a_{2} \lact v).
\end{align*}
We note that $\mc{D}(H,\phi)$ includes into $H$ as a dense subspace through the map $x \mapsto x(\mathds{1})$. This map is, however, not an intertwiner between bimodules on the nose, in fact, one may verify that the right action is twisted by conjugation by the modular operator $\Delta^{1/2}$.
\begin{comment}
We remark that when $\mc{D}(H, \phi)$ is viewed as a subspace of $H$, it is called the space of \emph{left-bounded vectors} in $H$.
\end{comment}

There is a canonical $\mc{A}_{2}$-valued inner product on $\mc{D}(H, \phi)$, defined as follows.
If $x \in \mc{D}(H,\phi)$, then its adjoint, written $x^{*}$, is an element of $\Hom_{-,\mc{A}_{2}}(H,L^{2}_{\phi}(\mc{A}_{2}))$. Hence, if $x,y \in \mc{D}(H,\phi)$, then
\begin{equation*}
y^{*}x \in \Hom_{-,\mc{A}_{2}}(L^{2}_{\phi}(\mc{A}_{2}),L^{2}_{\phi}(\mc{A}_{2})).
\end{equation*}
There is a canonical isomorphism $p_{\phi}: \Hom_{-,\mc{A}_{2}}(L^{2}_{\phi}(\mc{A}_{2}),L^{2}_{\phi}(\mc{A}_{2})) \rightarrow \mc{A}_{2}$, which is determined by the relation
$p_{\phi}(x) \lact v = x(v)$,
for $x \in \Hom_{-\mc{A}_{2}}(L^{2}_{\phi}(\mc{A}_{2}),L^{2}_{\phi}(\mc{A}_{2}))$ and $v \in L^{2}_{\phi}(\mc{A}_{2})$.
The aforementioned $\mc{A}_{2}$-valued inner product on $\mc{D}(H,\phi)$ is given by
$(x,y) = p_{\phi}(y^{*}x)$.
On the algebraic tensor product $\mc{D}(H,\phi) \otimes K$ we define a (degenerate) sesquilinear form by
        \begin{equation*}
        \langle (x \otimes v), (y \otimes w) \rangle_{\phi} = \langle p_{\phi}(y^{*}x) \lact v, w \rangle_{K}.
        \end{equation*}

\begin{definition}\label{def:ConnesFusionProduct}
        The \emph{Connes fusion product $H \boxtimes_{\phi} K$ of $H$ with $K$ relative to the faithful and normal state $\phi$} of $\mathcal{A}_2$ is  the completion of
        \begin{equation*}
        \mc{D}(H,\phi) \otimes K / \ker \langle \cdot, \cdot \rangle_{\phi}
        \end{equation*}
        with respect to the inner product $\langle \cdot, \cdot \rangle_{\phi}$, with the left $\mathcal{A}_1$ action obtained from the  $\mc{A}_{1}$-$\mc{A}_{2}$-bimodule structure of $\mathcal{D}(H,\phi)$, and the right $\mathcal{A}_3$ action obtained from the $\mc{A}_{2}$-$\mc{A}_{3}$ bimodule structure on $K$.
\end{definition}

The definition above does not treat $H$ and $K$ on equal footing, the following observation tells us that this is just an artifact of our description.
We define $\mc{D}'(K, \phi) = \Hom_{\mc{A}_{2}-}(L^{2}_{\phi}(\mc{A}_{2}),K)$ to be the space of bounded left module maps from $L^{2}_{\phi}(\mc{A}_{2})$ into $K$.
We then have that $\mc{D}'(K,\phi)$ includes into $K$ as a dense subspace, through the map $x\mapsto x(\mathds{1})$.
Using the canonical isomorphism $p'_{\phi}: \Hom_{\mc{A}_{2}-}(L_{\phi}^{2}(\mc{A}_{2}),L_{\phi}^{2}(\mc{A}_{2}))\rightarrow \mc{A}_{2}$ we define sesquilinear forms on $H \otimes \mc{D}'(K,\phi)$ and on $\mc{D}(H,\phi) \otimes \mc{D}'(K,\phi)$ respectively
\begin{align*}
 \langle (v \otimes x'), (w \otimes y') \rangle'_{\phi} &= \langle v \ract p'_{\phi}((y')^{*}x'), w \rangle_{H}, & & v,w \in H, \quad x',y' \in \mc{D}'(K,\phi), \\
 \langle (x \otimes x'), (y \otimes y') \rangle''_{\phi} &= \langle p_{\phi}(y^{*}x) \lact \mathds{1} \ract p'_{\phi}((y')^{*}x'), \mathds{1} \rangle_{L^{2}_{\phi}(\mc{A}_{2})} & & x,y \in \mc{D}(H,\phi) \quad x',y' \in \mc{D}'(K,\phi).
\end{align*}
In \cite[Section 5.2]{thom11} or \cite[Proposition 3.15 \& Definition 3.16]{Ta03} the following is proved:

\begin{lemma}\label{lem:AltFusionDefinition}
 The spaces $H \otimes \mc{D}'(K,\phi)/ \ker \langle \cdot, \cdot \rangle'_{\phi}$ and $\mc{D}(H,\phi) \otimes \mc{D}'(K,\phi)/\ker \langle \cdot, \cdot \rangle''_{\phi}$ are dense in $H \boxtimes_{\phi} K$.
\end{lemma}

\cref{lem:AltFusionDefinition} allows us to identify $H \boxtimes_{\phi} K$ with the completion of $H \otimes \mc{D}'(K,\phi)/\ker \langle \cdot, \cdot \rangle_{\phi}'$, which will be used in \cref{lem:ConnesFusionAssociative} in order to write down the associator for the Connes fusion product in a nice way.

Next is the functoriality of the Connes fusion product. Since we have not seen this written up in the way we need it, we will discuss this in detail. Let $\mc{B}_{1}, \mc{B}_{2}$ and $\mc{B}_{3}$ be further von Neumann algebras.
Let $H'$ be a $\mc{B}_{1}$-$\mc{B}_{2}$-bimodule and let $K'$ be a $\mc{B}_{2}$-$\mc{B}_{3}$-bimodule, let
$\nu_{i}:\mc{A}_{i} \rightarrow \mc{B}_{i}$, for $i=1,2,3$,
be isomorphisms of von Neumann algebras, and
let $\nu_{H}: H \rightarrow H'$ and $\nu_{K}: K \rightarrow K'$ be unitary maps, such that $(\nu_1,\nu_2,\nu_H)$ and $(\nu_2,\nu_3,\nu_K)$ are intertwiners.
Let $\phi':\mc{B}_{2} \rightarrow \C$ be a faithful and normal state. Finally, denote by $\psi: \mc{A}_{2} \rightarrow \C$ the faithful and normal state $\phi' \circ \nu_{2}$.
 Let $u: L_{\phi}^{2}(\mc{A}_{2}) \rightarrow L_{\psi}^{2}(\mc{A}_{2})$ be the unitary given by Theorem \ref{thm:StandardFormIsUnique}, with $\pi = \mathds{1}$.
 
\begin{proposition}\label{lem:ConnesFusionOfMaps}
Connes fusion is a functor. Explicitly,  the homomorphism $\nu_{2}:\mc{A}_{2} \rightarrow \mc{B}_{2}$ extends to a unitary map $\overline{\nu}_{2}:L_{\psi}^{2}(\mc{A}_{2}) \rightarrow L^{2}_{\phi'}( \mc{B}_{2})$ such that $(\nu_2,\nu_2,\overline{\nu}_2)$ is an intertwiner. Then, the map
\begin{align}\label{eq:DenselyDefinedFusion}
    \mc{D}(H, \phi) \otimes K \rightarrow \mc{D}(H', \phi') \otimes K', \quad
    x \otimes v \mapsto \nu_{H} x u^{*} \overline{\nu}_{2}^{*} \otimes \nu_{K}(v),
\end{align}
induces a unitary map
\begin{equation*}
    \nu_{H} \boxtimes \nu_{K}: H \boxtimes_{\phi} K \rightarrow H' \boxtimes_{\phi'} K',
\end{equation*}
such that $(\nu_1,\nu_3,\nu_{H} \boxtimes \nu_{K})$
is an  intertwiner.
Moreover, the construction of this intertwiner is compatible with  composition.
\end{proposition}

Here, compatibility with  composition means the following. Suppose we are given the following further data: von Neumann algebras $\mc{C}_{i}$ and von Neumann algebra isomorphisms $\nu_{i}': \mc{B}_{i} \rightarrow \mc{C}_{i}$ for $i=1,2,3$, a $\mc{C}_{1}$-$\mc{C}_{2}$-bimodule $H''$ and a $\mc{C}_{2}$-$\mc{C}_{3}$-bimodule $K''$, a faithful and normal state $\phi'':\mc{C}_{2} \rightarrow \C$, and unitary maps $\nu_{H'}: H' \rightarrow H''$ and $\nu_{K'}: K' \rightarrow K''$ such that $(\nu_1',\nu_2',\nu_{H'})$ and $(\nu_2',\nu_3',\nu_{K'})$ are intertwiners.
Then the following diagram commutes:
\begin{equation}
\label{pic:MultiplicativeConnesFusion}
\begin{tikzpicture}[scale=1.5]
 \node (A) at (0,0) {$H \boxtimes_{\phi} K$};
 \node (B) at (2.5,1) {$H' \boxtimes_{\phi'} K'$};
 \node (C) at (5,0) {$H'' \boxtimes_{\phi''} K''$};
 \path[->,font=\scriptsize]
 (A) edge node[above]{$\nu_{H} \boxtimes \nu_{K}$} (B)
 (B) edge node[above]{$\nu_{H'} \boxtimes \nu_{K'}$} (C)
 (A) edge node[below]{$\nu_{H'} \nu_{H} \boxtimes \nu_{K'} \nu_{K}$} (C);
\end{tikzpicture}
\end{equation}

\begin{proof}[Proof of \cref{lem:ConnesFusionOfMaps}]
 That $\nu_{2}$ extends to a unitary map follows from the fact that it intertwines the inner product $\mc{A}_{2} \times \mc{A}_{2} \rightarrow \C, (a_{1},a_{2}) \mapsto \psi(a_{2}^{*}a_{1})$ with the inner product $\mc{B}_{2} \times \mc{B}_{2} \rightarrow \C, (a_{1},a_{2}) \mapsto \phi'(a_{2}^{*}a_{1})$.
 That $(\nu_2,\nu_2,\overline{\nu}_2)$ is an intertwiner follows from the fact that $\nu_{2}$ is a isomorphism, which implies that $J_{\phi'} = \overline{\nu}_{2} J_{\phi} \overline{\nu}_{2}^{*}$.
 
 To prove that the map $x \otimes v \mapsto \nu_{H} x u^{*} \overline{\nu}_{2}^{*} \otimes \nu_{K}(v)$ induces an isomorphism it suffices to show that it intertwines the inner products. Explicitly, we need to prove that for all $x \otimes v$ and $y \otimes w$ in $\mc{D}(H,\phi) \otimes K$ we have
 \begin{equation*}
  \langle \nu_{H} x u^{*} \overline{\nu}_{2}^{*} \otimes \nu_{K}(v), \nu_{H} y u^{*} \overline{\nu}_{2}^{*} \otimes \nu_{K}(w) \rangle = \langle x \otimes v, y \otimes w \rangle.
 \end{equation*}
 We start from the left-hand side
 \begin{align*}
   \langle \nu_{H} x u^{*} \overline{\nu}_{2}^{*} \otimes \nu_{K}(v), \nu_{H} y u^{*} \overline{\nu}_{2}^{*} \otimes \nu_{K}(w) \rangle &= \langle p_{\phi'}( \overline{\nu}_{2} u y^{*}x u^{*} \overline{\nu}_{2}^{*} ) \nu_{K}(v), \nu_{K}(w) \rangle \\
   &= \langle \nu_{K}^{*} p_{\phi'}( \overline{\nu}_{2} u y^{*}x u^{*} \overline{\nu}_{2}^{*} ) \nu_{K}(v), w \rangle.
 \end{align*}
 Hence it suffices to show that
 \begin{equation}\label{eq:nuGoal}
  \nu_{K}^{*} p_{\phi'}( \overline{\nu}_{2} u y^{*}x u^{*} \overline{\nu}_{2}^{*} ) \nu_{K}(v) = p_{\phi}(y^{*}x)v.
 \end{equation}
 Using the fact that $\nu_{K}$ is an intertwiner along $\nu_{2}$ we obtain
 \begin{equation}\label{eq:nuHelp}
  \nu_{K}^{*} p_{\phi'}( \overline{\nu}_{2} u y^{*}x u^{*} \overline{\nu}_{2}^{*} ) \nu_{K} = \nu_{2}^{*}( p_{\phi'}( \overline{\nu}_{2} u y^{*}x u^{*} \overline{\nu}_{2}^{*} ))
 \end{equation}
 Now, we compute the action of the right hand side of \cref{eq:nuHelp} on an element $a \in L^{2}_{\psi}(\mc{B}_{2})$, using the definitions of $p_{\phi'}$ and $p_{\phi}$:
 \begin{align*}
  \nu_{2}^{*}( p_{\phi'}( \overline{\nu}_{2} u y^{*}x u^{*} \overline{\nu}_{2}^{*} )) a &= \overline{\nu}_{2}^{*}( p_{\phi'}( \overline{\nu}_{2} u y^{*}x u^{*} \overline{\nu}_{2}^{*} )\overline{\nu}_{2}(a)) 
= u y^{*} x u^{*} a   
= up_{\phi}(y^{*}x )u^{*}(a)  
= p_{\phi}(y^{*}x)a,
 \end{align*}
 since this holds for all $a$ in $L^{2}_{\psi}(\mc{B}_{2})$ we obtain
$\nu_{2}^{*}( p_{\phi'}( \overline{\nu}_{2} u y^{*}x u^{*} \overline{\nu}_{2}^{*} )) = p_{\phi}(y^{*} x)$.
 Together with \cref{eq:nuHelp}, this implies \cref{eq:nuGoal}.
 
 Finally, to prove that the diagram \cref{pic:MultiplicativeConnesFusion} commutes we shall prove that the following diagram commutes:
 \begin{equation*}
  \begin{tikzpicture}[scale=2]
\node (A) at (0,0) {$\mc{D}(H, \phi) \otimes K$};
\node (B) at (2,0) {$\mc{D}(H',\phi') \otimes K'$};
\node (C) at (4,0) {$\mc{D}(H'',\phi'') \otimes K''$};
\path[->,font=\scriptsize]
(A) edge (B)
(B) edge (C)
(A) edge[bend right] (C);
 \end{tikzpicture}
 \end{equation*}
 The composition of the arrows on top is the map
 \begin{equation}\label{eq:TopComposition}
  x \otimes v \mapsto \nu_{H'} \nu_{H} x u^{*}\overline{\nu}_{2}^{*} (u')^{*} (\overline{\nu}_{2}')^{*} \otimes \nu_{K'}\nu_{K}(v),
 \end{equation}
 where $u': L^{2}_{\phi'}(\mc{B}_{2}) \rightarrow L^{2}_{\phi'' \nu_{2}'}(\mc{B}_{2})$ is the unitary given by \cref{thm:StandardFormIsUnique}, with $\pi = \mathds{1}$.
 On the other hand, the bottom map is given by
 \begin{equation}\label{eq:BottomComposition}
  x \otimes v \mapsto \nu_{H'}\nu_{H} x (u'')^{*} \overline{\nu}_{2}^{*} (\overline{\nu}_{2}')^{*} \otimes \nu_{K'} \nu_{K}(v),
 \end{equation}
 where $u'' : L^{2}_{\phi}(\mc{A}_{2}) \rightarrow L^{2}_{\phi'' \nu_{2}' \nu_{2}}(\mc{A}_{2})$ is the unitary given by \cref{thm:StandardFormIsUnique}, with $\pi = \mathds{1}$.
 Note that in this expression, $\overline{\nu}_{2}$ is viewed as an isomorphism from $L^{2}_{\phi'' \nu_{2}' \nu_{2}}(\mc{A}_{2})$ into $L^{2}_{\phi'' \nu_{2}'}(\mc{B}_{2})$, (instead of from $L^{2}_{\phi' \nu_{2}}(\mc{B}_{2})$ into $L^{2}_{\phi'}(\mc{B}_{2})$ as before).
 Comparing \cref{eq:TopComposition} with \cref{eq:BottomComposition}, we see that it is sufficient to prove that
 \begin{equation*}
  u^{*}\overline{\nu}_{2}^{*} (u')^{*} (\overline{\nu}_{2}')^{*} = (u'')^{*} \overline{\nu}_{2}^{*} (\overline{\nu}_{2}')^{*}.
 \end{equation*}
 Taking the adjoint on both sides of the equation and cancelling $\overline{\nu}_{2}'$, we see that this is equivalent to
 \begin{equation*}
  u' \overline{\nu}_{2} u = \overline{\nu}_{2} u''.
 \end{equation*}
 One checks that both $u' \overline{\nu}_{2} u$ and $\overline{\nu}_{2}u''$ are isomorphisms from $L^{2}_{\phi}(\mc{A}_{2})$ into $L^{2}_{\phi'' \nu_{2}'}(\mc{B}_{2})$. We claim that they both satisfy properties (1-3) from \cref{thm:StandardFormIsUnique}, hence that by the uniqueness statement in that theorem we conclude they must be equal. Let us check that they do in fact satisfy properties (1-3):
 \begin{enumerate}
  \item Both maps intertwine the left action along $\nu_{2}$.
  \item By construction we have $\overline{\nu_{2}} u'' J_{\phi} (u'')^{*}  \overline{\nu_{2}}^{*} = J_{\phi'' \nu_{2}'} = u' \overline{\nu}_{2} u J_{\phi} u^{*} \overline{\nu}_{2}^{*} (u')^{*}$.
  \item From \cref{thm:StandardFormIsUnique} we have that $uP_{\phi} = P_{\phi' \nu_{2}}$. Now we argue that $\overline{\nu}_{2}P_{\phi' \nu_{2}} = P_{\phi'}$. Let $J_{\phi' \nu_{2}} a J_{\phi'\nu_{2}} a \in P_{\phi' \nu_{2}}$ be arbitrary. Then we compute
  \begin{equation*}
   \overline{\nu}_{2} J_{\phi' \nu_{2}} a J_{\phi' \nu_{2}} a = \overline{\nu}_{2} J_{\phi' \nu_{2}} \overline{\nu}_{2}^{*} \overline{\nu}_{2} a \overline{\nu}_{2}^{*} \overline{\nu}_{2} J_{\phi'\nu_{2}}\overline{\nu}_{2}^{*} \overline{\nu}_{2} a = J_{\phi'} \nu_{2}(a) J_{\phi'} \overline{\nu}_{2}a,
  \end{equation*}
  from which the claim follows. Then, another application of \cref{thm:StandardFormIsUnique} yields $u'P_{\phi'} = P_{\phi'' \nu_{2}'}$, hence $u' \overline{\nu_{2}}u P_{\phi} = P_{\phi'' \nu_{2}'}$. A similar argument then shows that $\overline{\nu}_{2} u'' P_{\phi} = P_{\phi'' \nu_{2}'}$.
 \end{enumerate}
 This completes the proof.
\end{proof}

\begin{remark}
In \cite[Theorem 6.23]{Bartels2014} a statement more general than \cref{lem:ConnesFusionOfMaps} is given, where none of the maps $\nu_{i}$, $\nu_{H}$, or $\nu_{K}$ is assumed to be an isomorphism.
However, in this more general setting Equation \cref{eq:DenselyDefinedFusion} does not make sense, and it is this explicit form that we make use of in the main text.
\end{remark}

\begin{comment}
\begin{remark}
 We have defined the map $\nu_{H} \boxtimes \nu_{K}$ by specifying it on the dense subspace $\mc{D}(H,\phi) \otimes K$.
 Of course we could have just as well defined the map $\nu_{H} \boxtimes \nu_{K}$ to be the extension of a similarly defined map on the dense subspace $H \otimes \mc{D}'(K,\phi)$.
 One may check that this procedure has the same result.
\end{remark}
\end{comment}

We may specialize \cref{lem:ConnesFusionOfMaps} to the case that all isomorphisms $\nu$ are identities, to conclude that for any two faithful and normal states $\phi, \phi'$ on $\mc{B}_{2}$, there exists a natural isomorphism
\begin{equation*}
 H \boxtimes_{\phi} K \rightarrow H \boxtimes_{\phi'} K.
\end{equation*}
The fact that the diagram \cref{pic:MultiplicativeConnesFusion} commutes then tells us that these isomorphisms are coherent, which allows  to define the Connes fusion product $H \boxtimes_{\mc{A}_{2}} K$ as the colimit of $H \boxtimes_{\phi} K$, where $\phi$ ranges over all faithful and normal states of $\mc{A}_{2}$. If no state is preferred, then we always refer to this limit, and if later a state is picked, then we have a unique isomorphism. In \cite[Exercise IX.3.8 (p.210)]{Ta03} another approach to defining a tensor product of bimodules without reference to a state is given.

\begin{proposition}
\label{lem:ConnesFusionAssociative}
\label{lem:FusionOfMapsCompatibleAssociative}
Connes fusion is  associative. Explicitly, if $H$ is an $\mc{A}_{1}$-$\mc{A}_{2}$-bimodule, $K$ is an $\mc{A}_{2}$-$\mc{A}_{3}$-bimodule, and $L$ is an $\mc{A}_{3}$-$\mc{A}_{4}$-bimodule, $\phi$ is a faithful and normal state on $\mc{A}_{2}$, and $\psi$ is a faithful and normal state on $\mc{A}_{3}$, then the map
 \begin{align*}
  (\mc{D}(H,\phi) \otimes K) \otimes \mc{D}'(L,\psi) &\rightarrow \mc{D}(H,\phi) \otimes (K \otimes \mc{D}'(L,\psi)) \\
  ((x \otimes v) \otimes y) & \mapsto (x \otimes (v \otimes y)),
 \end{align*}
 induces  a unitary intertwiner 
\begin{equation*} 
 \alpha_{H,K,L}:(H \boxtimes_{\phi} K) \boxtimes_{\psi} L \rightarrow H \boxtimes_{\phi} (K \boxtimes_{\psi} L)
\end{equation*}
of $\mathcal{A}_1$-$\mathcal{A}_4$-bimodules.
 Moreover, these intertwiners are natural and satisfy the pentagon identity.
\end{proposition}

Here, naturality means the following.
Let $H,K,L$ be $\mc{A}_{1}$-$\mc{A}_{2}$-, $\mc{A}_{2}$-$\mc{A}_{3}$-, and $\mc{A}_{3}$-$\mc{A}_{4}$-bimodules respectively, and we let $H',K',L'$ be $\mc{B}_{1}$-$\mc{B}_{2}$-, $\mc{B}_{2}$-$\mc{B}_{3}$-, and $\mc{B}_{3}$-$\mc{B}_{4}$-bimodules respectively. Moreover, we consider isomorphisms $\nu_{i}: \mc{A}_{i} \rightarrow \mc{B}_{i}$ for $i=1,2,3,4$, and unitary  maps $\nu_{H}: H \rightarrow H'$, $\nu_{K}: K \rightarrow K'$ and $\nu_{L}: L \rightarrow L'$ such that $(\nu_1,\nu_2,\nu_H)$, $(\nu_2,\nu_3,\nu_K)$, and $(\nu_3,\nu_4,\nu_L)$ are intertwiners. Finally, let $\phi$, $\psi$, $\phi'$, $\psi'$ be faithful and normal states on $\mathcal{A}_2$, $\mathcal{A}_3$, $\mathcal{B}_2$, and $\mathcal{B}_3$, respectively. Then, naturality means the commutativity of the diagram 
 \begin{equation*}
        \begin{tikzpicture}[scale=1.8]
        \node (A) at (0,1) {$(H \boxtimes_{\phi} K) \boxtimes_{\psi} L$};
        \node (B) at (3,1) {$H \boxtimes_{\phi} (K \boxtimes_{\psi} L)$};
        \node (C) at (0,0) {$(H' \boxtimes_{\phi'} K') \boxtimes_{\psi'} L'$};
        \node (D) at (3,0) {$H' \boxtimes_{\phi'} (K' \boxtimes_{\psi'} L')$.};
        \path[->,font=\scriptsize]
        (A) edge node[above]{$\alpha_{H,K,L}$} (B)
        (A) edge node[left]{$(\nu_{H} \boxtimes \nu_{K}) \boxtimes \nu_{L} $} (C)
        (B) edge node[right]{$\nu_{H} \boxtimes (\nu_{K} \boxtimes \nu_{L}) $} (D)
        (C) edge node[below]{$\alpha_{H',K',L'}$} (D);
        \end{tikzpicture}
\end{equation*}

\begin{proof}[Proof of \cref{lem:ConnesFusionAssociative}]
 That the above map indeed induces an isomorphism of bimodules is \cite[Theorem 3.20]{Ta03}.
 A proof of the pentagon identity can be found in \cite{Brouwer2003}. Naturality can be proved by a straightforward computation using the explicit forms of the isomorphisms on the appropriate dense subspaces from \cref{lem:ConnesFusionAssociative,lem:ConnesFusionOfMaps} \end{proof}

Finally, we discuss the fact that the standard form of a von Neumann algebra is neutral with respect to confusion.  

\begin{proposition}\label{lem:L2IsConnesNeutral}
        The $\mc{A}_{2}$-$\mathcal{A}_2$-bimodule $L^{2}_{\phi}(\mc{A}_{2})$ is neutral with respect to Connes fusion. Explicitly, for every $\mc{A}_{2}$-$\mc{A}_{3}$-module $K$ the map
        \begin{align*}
        \mc{D}(L^{2}_{\phi}(\mc{A}_{2}),\phi) \otimes K \rightarrow K, \quad
        x \otimes v  \mapsto p_{\phi}(x) \lact v,
        \end{align*}
induces a unitary intertwiner 
\begin{equation*}
\lambda_{K}: L_{\phi}^{2}(\mc{A}_{2}) \boxtimes_{\phi} K \rightarrow K
\end{equation*}
of $\mathcal{A}_2$-$\mathcal{A}_3$-bimodules. Likewise, for every $\mc{A}_{1}$-$\mc{A}_{2}$-bimodule $H$, the map
        \begin{equation*}
         H       \otimes \mc{D}'(L_{\phi}^{2}(\mc{A}_{2}),\phi) \rightarrow H, \quad w \otimes y \mapsto w \ract p'_{\phi}(y),
        \end{equation*}
induces a unitary intertwiner
\begin{equation*}
\rho_{H}: H \boxtimes_{\phi} L_{\phi}^{2}(\mc{A}_{2}) \rightarrow H
\end{equation*}
of $\mathcal{A}_1$-$\mathcal{A}_2$-bimodules.         Moreover, these maps are natural, and compatible with the associator in the sense that~
\begin{equation*}        
        \rho_{H} \boxtimes \mathds{1}_{K} = (\mathds{1}_{H} \boxtimes \lambda_{K}) \circ \alpha_{H,L_{\phi}^{2}(\mc{A}_{2}),K}\text{.}
\end{equation*}
        Finally, we have $\rho_{L^{2}_{\phi}(\mc{A}_{2})} = \lambda_{L^{2}_{\phi}(\mc{A}_{2})}$.
\end{proposition}
\begin{proof}
It is straightforward to see that the given maps induce the claimed intertwiners, and that naturality and  the compatibility condition is satisfied. For more detail, we refer to \cite{Brouwer2003}, see in particular Proposition 3.5.3 therein.
   To prove that $\rho_{L^{2}_{\phi}(\mc{A}_{2})} = \lambda_{L^{2}_{\phi}(\mc{A}_{2})}$ it suffices to prove that the diagram
        \begin{equation*}
        \xymatrix{
         \mc{D}(L^{2}_{\phi}(\mc{A}_{2}),\phi) \otimes\mc{D}'(L_{\phi}^{2}(\mc{A}_{2}),\phi) \ar[r] \ar[d] &          \mc{D}(L^{2}_{\phi}(\mc{A}_{2}),\phi) \otimes L_{\phi}^{2}(\mc{A}_{2}) \ar[d] \\
         L^{2}_{\phi}(\mc{A}_{2}) \otimes\mc{D}'(L_{\phi}^{2}(\mc{A}_{2}),\phi) \ar[r] & L_{\phi}^{2}(\mc{A}_{2}) 
        }
        \end{equation*}
        commutes, which follows from the computation
        \begin{equation*}
         x(\mathds{1}) \ract p'_{\phi}(y) = p_{\phi}(x) \lact \mathds{1} \ract p'_{\phi}(y)  = p_{\phi}(x)\lact y(\mathds{1}),
        \end{equation*}
        where $x \in \mc{D}(L_{\phi}^{2}(\mc{A}_{2}),\phi)$, $y \in \mc{D}'(L_{\phi}^{2}(\mc{A}_{2}),\phi)$ and $\mathds{1}$ is the unit of $\mc{A}_{2}$.
\end{proof}

Finally, we want to transfer the results of \cref{lem:L2IsConnesNeutral} from the canonical standard form to other standard forms. 
Let $I$ be a left $\mc{A}_2$-module, and let $\xi \in I$ be a cyclic and separating vector, so that $I$ becomes a standard form of $\mathcal{A}_2$, see \cref{lem:TTstandardform}.
Let $\phi: \mc{A}_2\rightarrow \C$ be the faithful and normal state $\phi(a) = \langle a \lact \xi, \xi \rangle$, and consider the corresponding standard form $L^{2}_{\phi}(\mc{A}_2)$. By \cref{thm:StandardFormIsUnique}, both standard forms are isomorphic under a unique isomorphism $u:  I \to  L^{2}_{\phi}(\mc{A}_2)$, which is by  \cref{lem:PutInStandardForm} given by the extension of the map
     $ a \lact \xi \mapsto a$. 
Further, we recall from \cref{re:bimodulefromstandardform} that standard forms are $\mathcal{A}_2$-$\mathcal{A}_2$-bimodules, and that $u$ is an intertwiner of $\mathcal{A}_2$-$\mathcal{A}_2$-bimodules.

\begin{corollary}
\label{co:standardformneutral}
Let $I$ be a left $\mc{A}_2$-module, and let $\xi \in I$ be a cyclic and separating vector. Then, $I$ is neutral with respect to Connes fusion. More explicitly, for every $\mathcal{A}_2$-$\mathcal{A}_3$-bimodule $K$ and every $\mathcal{A}_1$-$\mathcal{A}_2$-bimodule $H$ the unitary intertwiners 
\begin{align*}
\lambda_K^{I}: I \boxtimes_{\phi} K \to K, \quad \lambda^{I}_K\defeq  \lambda^{}_K \circ (u \boxtimes \mathds{1}_{K})
\\
\rho^{I}_H: H \boxtimes_{\phi} I\to H, \quad\rho^{I}_H\defeq \rho^{}_H \circ (\mathds{1}_{H}\boxtimes u)
\end{align*}
are natural and compatible with the associator. Moreover, we have $\lambda^{I}_I=\rho^{I}_I$. \end{corollary}

\begin{proof}
Compatibility with the associator follows from the definition of $\lambda_K^{I}$ and $\rho_H^{I}$ and the naturality of the associator proved in \cref{lem:FusionOfMapsCompatibleAssociative}. \begin{comment}
Indeed, the diagram
\begin{equation*}
\xymatrix@C=1cm{&&& (H \boxtimes L^2) \boxtimes K \ar[rr]^-{\alpha_{H,L^2,K}} \ar[dr]_<<<<{\rho_H \boxtimes \mathds{1}_{K}} &&  H \boxtimes (L^2 \boxtimes K) \ar[dl]^{\mathds{1}_H \boxtimes \lambda_K} \\ &&&& H \boxtimes K \\  (H \boxtimes I) \boxtimes K \ar[uurrr]^{(\id_H \boxtimes u) \boxtimes \id_K} \ar[rr]^{\alpha_{H,I,K}} \ar[dr]_{\rho_H^{u} \boxtimes \mathds{1}_K} &&  H\boxtimes (I \boxtimes K) \ar[uurrr]^<<<<<<<<<<<<<<<<<{\id_H\boxtimes (u \boxtimes \id_K)}\ar[dl]_{\mathds{1}_H \boxtimes \lambda_K^{u}} \\ & H \boxtimes K \ar@{=}[uurrr] }
\end{equation*}
is commutative.
\end{comment}
Naturality and the coincidence $\lambda^{I}_I=\rho^{I}_I$ follow from the naturality of $\lambda_K$ and $\rho_H$.  
\begin{comment}
Indeed,
\begin{equation*}
\xymatrix@C=5em@R=5em{&I \boxtimes L^2 \ar[d]_{u \boxtimes\id_{L^2}} \ar[r]^{\rho_I} & I \ar@{=}[dr] \ar[d]^{u}\\ I \boxtimes I \ar[ur]^{\id_I \boxtimes u} \ar[dr]_{u \boxtimes \id_I} &L^2 \boxtimes L^2 \ar[r]^-{\rho_{L^2} } \ar[r]_-{\lambda_{L^2}} \ar[d]_{\id_{L^2} \boxtimes u^{-1}} & L^2   \ar[d]^{u^{-1}} & I \\ &L^2 \boxtimes I \ar[r]_-{\lambda_{I}} & I \ar@{=}[ur] }
\end{equation*}
\end{comment}
\end{proof}

\newcommand{\etalchar}[1]{$^{#1}$}

\def\kobiburl#1{
   \IfSubStr
     {#1}
     {://arxiv.org/abs/}
     {\kobibarxiv{#1}}
     {\kobiblink{#1}}}
\def\kobibarxiv#1{\href{#1}{\texttt{[arxiv:\StrGobbleLeft{#1}{21}]}}}
\def\kobiblink#1{
  \StrSubstitute{#1}{\~{}}{\string~}[\myurl]
  \StrSubstitute{#1}{_}{\underline{\;\;}}[\mylink]
  \StrSubstitute{\mylink}{&}{\&}[\mylink]
  \StrSubstitute{\mylink}{/}{/\allowbreak}[\mylink]
  \newline Available as: \mbox{\;}
  \href{\myurl}{\texttt{\mylink}}}

\raggedright
\addcontentsline{toc}{section}{\refname}
\small
\bibliographystyle{kobib}
\bibliography{bibfile}

Peter Kristel ({\it peter.kristel@umanitoba.ca})

\smallskip

University of Manitoba\\
Department of Mathematics\\
186 Dysart Road\\
Winnipeg, MB R3T 2N2 \\
Canada

\medskip

Konrad Waldorf ({\it konrad.waldorf@uni-greifswald.de})

\smallskip

Universit\"at Greifswald\\
Institut f\"ur Mathematik und Informatik\\
Walther-Rathenau-Str. 47\\
17487 Greifswald\\
Germany

\end{document}